\documentclass[12pt,letterpaper]{amsart}
\usepackage{ragged2e} 
\usepackage{amsmath, amssymb, graphicx, url, amsthm, multicol, enumitem, textcomp, gensymb, mathtools, braket}
\usepackage{mdframed}
\usepackage{indentfirst}
\usepackage{units}
\usepackage{wrapfig}
\usepackage{tikz}
\usepackage{ccaption}
\usetikzlibrary{backgrounds,shapes.geometric,arrows}

\usepackage[dvips,centering,includehead,width=15.1cm,height=22.7cm,paperheight=23.8cm,paperwidth=17cm]{geometry}

\changecaptionwidth
\captionwidth{5.6in}

\newcommand{\R}{\mathbb{R}}
\newcommand{\C}{\mathbb{C}}

\newcommand{\Z}{\mathbb{Z}}

\newcommand{\bfS}{\mathbf{S}}

\newcommand{\kap}{K}
\newcommand{\neline}{{\scriptstyle\diagup}}
\newcommand{\seline}{{\scriptstyle\diagdown\,}}
\newcommand{\smallneline}{{\,\scriptscriptstyle\diagup}}
\newcommand{\smallseline}{{\scriptscriptstyle\diagdown\hspace{0.3pt}}}
\newcommand{\xdiamond}{\rotatebox[origin=c]{45}{$\boxplus\!\!$}}

\DeclareMathOperator{\Aut}{Aut}

\let\ol\overline

\let\vec\overrightarrow
\newcommand{\caymthree}{M^0_3}
\newcommand{\caymfour}{M^0_4}
\newcommand{\caymk}{M^0_k}
\newcommand{\caymnplusone}{M^0_{n+1}}
\newcommand{\scmthree}{M^\kap_3}
\newcommand{\scmfour}{M^\kap_4}
\newcommand{\scmm}{M^\kap_m}

\newcommand{\skap}{S^\kap}
\newcommand{\skaponetwothree}{S^\kap_{123}}
\newcommand{\skaponethreefour}{S^\kap_{134}}

\newcommand{\skapijk}{S^\kap_{ijk}}
\newcommand{\skapikl}{S^\kap_{ikl}}
\newcommand{\skapijl}{S^\kap_{ijl}}
\newcommand{\skapjkl}{S^\kap_{jkl}}
\newcommand{\hkap}{H^\kap}

\newcommand{\zpi}{\mathbf{z}_\pi}
\newcommand{\zscm}{\widehat{\mathbf{z}}}
\newcommand{\zfrieze}{\mathbf{z}}

\newcommand{\xS}{\mathbf{xS}}
\newcommand{\xx}{\mathbf{x}}

\usepackage{etoolbox}
\AtBeginEnvironment{proof}{\setlength{\parindent}{0.4cm}}
\AtBeginEnvironment{itemize}{\setlength{\parindent}{0.2cm}}
\AtBeginEnvironment{enumerate}{\setlength{\parindent}{0.2cm}}

\numberwithin{equation}{section}

\theoremstyle{definition}
\newtheorem{definition}{Definition}[section]
\newtheorem{remark}[definition]{Remark}
\newtheorem{example}[definition]{Example}

\theoremstyle{plain}
\newtheorem{theorem}[definition]{Theorem}
\newtheorem{corollary}[definition]{Corollary}
\newtheorem{lemma}[definition]{Lemma}
\newtheorem{proposition}[definition]{Proposition}
\newtheorem*{proposition-nonum}{Proposition}
\newtheorem*{lemma-nonum}{Lemma}
\newtheorem*{theorem-nonum}{Theorem}
\newtheorem*{example-nonum}{Example}
\newtheorem*{definition-nonum}{Definition}

\newcommand{\lr}[1]{\langle #1 \rangle}

\makeatletter
\renewcommand*\env@matrix[1][\arraystretch]{%
  \edef\arraystretch{#1}%
  \hskip -\arraycolsep
  \let\@ifnextchar\new@ifnextchar
  \array{*\c@MaxMatrixCols c}}
\makeatother

\newcommand{\dietmatrix}[1]{
\setlength\arraycolsep{2 pt}
#1
}

\newcommand{\circleat}[2]{\filldraw[](#1,#2) circle (2pt)}

\newcommand{\mytfrac}[2]{\textstyle\frac{#1}{#2}}

\makeatletter

\newcommand{\checknextarg}{\@ifnextchar\bgroup{\gobblenextarg}{\right]}}
\newcommand{\gobblenextarg}[1]{\enskip #1\@ifnextchar\bgroup{\gobblenextarg}{\right]}}
\makeatother

\makeatletter

\newcommand{\checknextargd}{\@ifnextchar\bgroup{\gobblenextargd}{\right]}}
\newcommand{\gobblenextargd}[1]{\, #1\@ifnextchar\bgroup{\gobblenextargd}{\right]}}
\makeatother

\usepackage{labelschanged}

\usepackage{hyperref}

\hypersetup{
  colorlinks   = true, 
  urlcolor     = blue, 
  linkcolor    = blue,
  citecolor   = red 
}

\begin{document}

\title{Spherical Friezes}

\author{Katie Waddle}
\address{Department of Mathematics, University of Michigan, Ann Arbor, MI 48109, USA}
\email{waddle@umich.edu}


\thanks{Partially supported by NSF grants DMS-1840234, DMS-2054231 and DMS-2348501.}


\date{\today}

\begin{abstract}
A fundamental problem in spherical distance geometry aims to recover an~$n$-tuple of points on a 2-sphere in~$\mathbb{R}^3$, viewed up to oriented isometry, from~$O(n)$ input measurements. 
We solve this problem using algorithms that employ only the four arithmetic operations.
Each algorithm recursively produces output data that we arrange into a new type of frieze pattern.
These frieze patterns exhibit glide symmetry and a version of the Laurent phenomenon.

\end{abstract}

\maketitle

\section{Introduction and main results}

\subsection*{Introduction}
Distance geometry studies point configurations via measurements of  pairwise distances between the points.
More generally, one can consider any
\linebreak[3]
``intrinsic''
measurements that do not depend on a choice of coordinate system.
In~a~Euclidean space, such measurements can include distances, angles, (oriented) volumes,~etc.
A~typical problem asks for a small set of inputs, or initial data,
that uniquely determine a configuration up to orientation-preserving isometry.
One can further ask for explicit formulas that express, in terms of the initial data, arbitrary measurements of a particular kind---say,
all pairwise distances in a configuration.

In this paper, we study the above problems in the setting of spherical geometry.
Let~$\bfS$ be a sphere of radius~$R$ in 
3-dimensional Euclidean space.
Viewed up to oriented isometry, the space of~$n$-point configurations on~$\bfS$ has dimension~$2n-3$.
We seek to identify a subset of~$O(n)$ initial measurements  
from which all pairwise distances between~$n$ given points can be recovered.
We require the corresponding formulas to be \emph{rational},  
in particular to avoid the use of trigonometric functions or their inverses.
In service of this goal, we measure pairwise distances between points on~$\bfS$
in the ambient Euclidean space, rather than via geodesics on the sphere.
(The two types of distances can be readily converted into each other.)

We present two solutions to the  problem above.
In the first solution, the initial input consists of~$2n-3$ squared Euclidean distances between certain pairs of points in an~$n$-point configuration on the sphere,
together with~$n-2$ additional binary bits of information
specifying the orientations of certain spherical triangles.
The second solution utilizes~$3n-4$ squared distances.
In each case, the initial data determine the~$n$-point configuration uniquely up to an isometry of the sphere.

\pagebreak[3]

To~capture the measurement data involved in our algorithms and the algebraic relations among them,
we introduce two different but closely related record-keeping devices that we call \emph{spherical friezes}.
These friezes extend from Euclidean to spherical settings  
the Heronian friezes of Fomin and Setiabrata~\cite{fomin_heronian_2021},
which were in turn inspired by the classical Coxeter-Conway friezes~\cite{coxeter_frieze_1971,conway_triangulated_1973}.
The spherical friezes display
how the initial measurements propagate to determine the others. They can also be viewed as purely algebraic instruments: numerical arrays
whose neighboring entries satisfy certain relations.
We prove that, like other types of friezes, spherical friezes  exhibit \emph{glide symmetry}
and a form of the \emph{Laurent phenomenon}.


\subsection*{Main results}
Fix a sphere~$\bfS$ of radius~$R$ in 3-dimensional Euclidean space.
A \emph{spherical polygon}~$A_1\cdots A_n$
is an ordered~$n$-tuple of points~$A_1,\dots,A_n\in\bfS$.
The group~$\Aut(\bfS)$ of orientation-preserving isometries of~$\bfS$ acts on the space~$\bfS^n$ of such~$n$-tuples.    
Spherical distance geometry does not distinguish between spherical polygons that lie in the same~$\Aut(\bfS)$-orbit.
Accordingly, we will only be interested in  numerical quantities (\emph{measurements})
associated with a spherical polygon that are \emph{intrinsic}, i.e.~$\Aut(\bfS)$-invariant.
Such a measurement is a function~$\bfS^n/\Aut(\bfS)\to\mathbb{R}$.
The basic example is a squared Euclidean distance
\begin{equation}
\label{eq: xij-intro}
x_{ij}=x(A_i,A_j)
\end{equation}between two points~$A_i,A_j$ lying on the sphere~$\bfS$.

It is natural to seek a coordinatization of the configuration space~$\bfS^n/\Aut(\bfS)$
that enables calculations performed entirely in terms of  
intrinsic measurements.
Since~$\dim\bfS^n/\Aut(\bfS)=2n-3$, we would like this coordinatization to involve at most~$O(n)$~measurements whose values
uniquely determine a spherical polygon, up to the~$\Aut(\bfS)$ action.
We~additionally ask that the formulas that express all basic measurements
in terms of the chosen coordinates be rational, i.e., involve only the four arithmetic operations.
This rules out the standard formulas found in the spherical geometry literature, as they involve
square roots, trigonometric functions, and/or their inverses, see for example~\cite[Chapter 8]{todhunter_spherical_1886}\cite{fillmore_formulas_2020}\cite{bajgonakova_bretschneiders_2012}.

We begin by interpreting a spherical polygon as a purely combinatorial object,
mapping its vertices~$A_1,\dots,A_n$ to the vertices of a convex~$n$-gon on a real plane.
Consider a triangulation of this~$n$-gon obtained by drawing~$n-3$ diagonals. The~$2n-3$ sides and diagonals of this triangulation correspond to~$2n-3$ measurements~$x_{ij}$ in the spherical polygon.  However, these~$2n-3$ measurements are not sufficient to determine the spherical polygon up to the~$\Aut(\bfS)$ action, because each of the~$n-2$ spherical triangles~$A_iA_jA_k$ is only determined by the measurements~$x_{ij}, x_{ik}, x_{jk}$
up to a (potentially orientation-altering) isometry.
Since each of these triangles can be placed on the sphere in two different ways,
we get, up to the~$\Aut(\bfS)$ action,~$2^{n-2}$ distinct spherical~$n$-gons
(or~$2^{n-3}$ distinct~$n$-gons, up to isometry).
This ambiguity can be resolved by supplementing the initial data with additional measurements that ``rigidify'' the geometry of the configuration.

We implement two versions of this strategy. In the first version we make use of the following quantity that is invariant under the action of~$\Aut(\bfS)$.
\begin{definition}[cf.~Definition~\ref{def: skappa def}]
Let~$A_1,A_2,A_3\in\bfS$ be three points on a sphere~$\bfS$ with center~$O$, radius~$R$, and Gaussian curvature~$K=\frac1{R^2}$.
We denote
\begin{equation}
\label{eq: skap-intro}
\skap(A_1,A_2,A_3)=\frac{12V(OA_1A_2A_3)}{R}\, ,
\end{equation}
where~$V(OA_1A_2A_3)$
is the signed volume of the tetrahedron~$O A_1 A_2 A_3$.  
\end{definition}
\begin{remark}
    Informally,~$\skap(A_1,A_2,A_3)$ is a stand-in for the (normalized) oriented area of the triangle~$A_1A_2A_3$.  Rigid motions that change the orientation of the triangle also change the sign of this measurement, so it can be used to distinguish between the two triangles with given side lengths.
\end{remark}

The following proposition illustrates the relationship between the measurements \eqref{eq: xij-intro} and~\eqref{eq: skap-intro}.

\begin{proposition}[cf.~Proposition~\ref{prop: sph heron}]\label{prop: heron intro}
For~$A_1,A_2,A_3\in\bfS$,
we have
    \begin{equation}
        \left(\skap(A_1,A_2,A_3)\right)^2=\det\begin{bmatrix}
            \frac{\kap}{2} & 1 & 1 &1 \\
            1 & 0 & x_{12} & x_{13} \\
            1 & x_{21} & 0 & x_{23} \\
            1 & x_{31} & x_{32} & 0
        \end{bmatrix}
        =H^\kap(x_{12},x_{13},x_{23})\, ,
        \label{eq: her eq}
        \end{equation}
        where we use the notation
        \begin{equation}\label{eq: h kap}
           H^\kap(x_{12},x_{13},x_{23}) =-\!x_{12}^2\!-\!x_{13}^2\!-\!x_{23}^2\!+\!2x_{12}x_{13}\!+\!2x_{12}x_{23}\!+\!2x_{13}x_{23}\!-\!\kap x_{12}x_{13}x_{23}.
    \end{equation}
\end{proposition}

\begin{remark}
In the \emph{flat limit}~$R\to\infty$ (i.e., for~$K=0$),  
Proposition~\ref{prop: heron intro} recovers a well-known formula for the squared area
of a triangle~$A_1A_2A_3$ on the Euclidean plane in terms of the squared lengths~$x_{ij}$ of its sides~$A_iA_j$, an application of the \emph{Cayley-Menger determinant}.

The case~$\kap=0$ of Proposition~\ref{prop: heron intro}  can also be seen as a restatement of the classical \emph{Heron's formula} for the area of a Euclidean triangle.
We therefore call equation~\eqref{eq: her eq} the \emph{(spherical) Heron equation}.

The determinant in~\eqref{eq: her eq} is known as the spherical Cayley-Menger determinant, which we learned about in T. Tao's blog post~\cite{tao_spherical_2019}, cf.~\cite{audet_determinants_2011, blumenthal_distribution_1943}. 
\end{remark}

The squared side distances~$x_{ij},x_{ik},x_{jk}$, together with one of the two possible values of~$\skap(A_i,A_j,A_k)$
allowed by~\eqref{eq: her eq},
determine a  spherical triangle~$A_iA_jA_k$ 
up to oriented isometry.  
This leads to our first coordinatization of~$\bfS^n/\Aut(\bfS)$ that  
involves the~$2n-3$ squared distances~$x_{ij}$ associated to the sides and diagonals in a triangulation~$T$ of an~$n$-gon,
plus the~$n-2$ measurements~$\skap(A_i,A_j,A_k)$ associated to the triangles of~$T$.  

\medskip

We next write formulas expressing all basic measurements for a spherical polygon
in terms of this initial data.  We begin with the case of a spherical quadrilateral~$A_1A_2A_3A_4$, see Figure~\ref{fig: intro quad}.

\begin{figure}[htbp!]
    \centering
\vspace{-4pt}
    \begin{tikzpicture}

\coordinate [label=left:$A_1$] (A_1) at (-2,2);
\circleat{-2}{2};

\coordinate [label=below:$A_4$](A_4) at (-0.5,1);
\circleat{-0.5}{1};

\coordinate [label=right:$A_3$] (A_3) at (2,2);
\circleat{2}{2};

\coordinate[label=above:$A_2$] (A_2) at (1,4);
\circleat{1}{4};

\coordinate [label=left:$q$] (q) at (-2.5,1);
\coordinate (q_1) at (-1,1.75);
\draw (q_1) to [out=180,in=0] (q) ;

\coordinate [label=left:$p$] (p) at (-2,3);
\coordinate (p_1) at (-1.25,2.25);
\draw (p_1) to [out=135,in=0] (p) ;

\coordinate [label=left:$r$] (r) at (-1,4);
\coordinate (r_1) at (0,3);
\draw (r_1) to [out=90,in=0] (r);

\coordinate [label=right:$s$] (s) at (2,4);
\coordinate (s_1) at (1,3);
\draw (s_1) to [out=90,in=180] (s);

\draw (A_1) to node[above left] {$b$} (A_2) ;
\draw (A_2) to node[above right] {$c$} (A_3) ;
\draw (A_3) to node[below] {$d$} (A_4) ;
\draw (A_4) to node[below left] {$a$} (A_1) ;
\draw (A_1) to node[above, near end ] {$e$} (A_3) ;
\draw[dotted] (A_2) to node[left, near start] {$f$} (A_4);

\end{tikzpicture}\hspace{1 cm}
\begin{tikzpicture}[every node/.style={circle,fill=white}]

        \begin{scope}[on background layer]
       
        \draw[dashed] (-2,4) node {$d$}
        --(2,0);
        \draw[dashed] (-2,0) node {$b$}
        --(2,4);
    \end{scope}
        \draw[very thick]
        (0,0) node {$c$}
        -- (-1,1) node {$p$}
        -- (-2,2) node {$e$}
        -- (-1,3) node {$q$}
        -- (0,4) node {$a$}
        -- (1,3) node {$r$}
        -- (2,2) node {$f$}
        -- (1,1) node {$s$}
        -- (0,0);
\end{tikzpicture}    
\vspace{-8pt}
\caption{Left:
Measurements in a spherical quadrilateral~$A_1A_2A_3A_4$  include  
(i)~six squared distances:~$a=x_{14}, b=x_{12}, c=x_{23}, d=x_{34}, e=x_{13}, f=x_{24}$;  
(ii)~four~measurements~$S^K(A_i,A_j,A_k)$ associated with the triangles~$A_iA_jA_k$: \\
$p\!=\!\skap(A_1,A_2,A_3),q\!=\!\skap(A_1,A_3,A_4),r\!=\!\skap(A_1,A_2,A_4),s\!=\!\skap(A_2,A_3,A_4)$.
Right: the same measurement data, arranged in a spherical Heronian diamond.  
}
\vspace{-4pt}
    \label{fig: intro quad}
\end{figure}
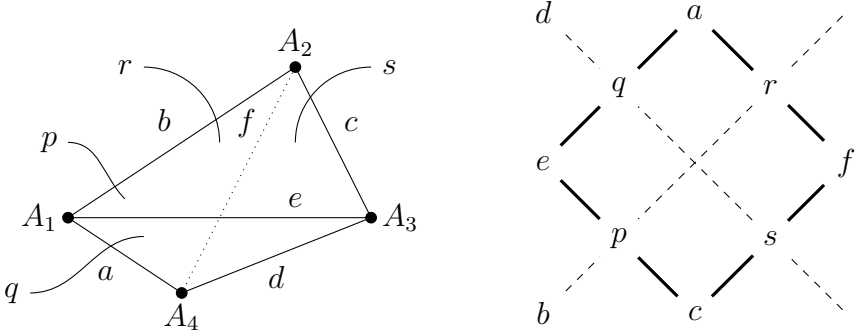

\pagebreak[3]

The ten measurements associated with the quadrilateral~$A_1A_2A_3A_4$ (the~${\binom{4}{2}=6}$ squared distances~$x_{ij}$  and the  
${\binom{4}{3}=4}$ measurements~$S^K(A_i,A_j,A_k)$)
satisfy four instances of the Heron equation~\eqref{eq: her eq}:
\begin{align}\label{eq: four instances}
\begin{split}
        p^2 &= \hkap(b,c,e)\, , \\ 
    q^2 &= \hkap(a,d,e)\, , \\
    r^2 &= \hkap(a,b,f) \, , \\ 
    s^2 &= \hkap(c,d,f) \, .
    \end{split}
\end{align}
They also satisfy three additional relations stated in Proposition~\ref{prop: 2.14 analogue intro} below.

\begin{proposition}[cf.~Proposition~\ref{prop: analogue of 2.14}]\label{prop: 2.14 analogue intro}
    Let~$(a,b,c,d,e,p,q)$ be a~$7$-tuple of complex numbers arising from a spherical quadrilateral~$A_1A_2A_3A_4$ as shown in Figure~\ref{fig: intro quad}. 
    Assuming points~$A_1,A_3$ are not coincident or antipodal (i.e.~$e\not\in\left\{0,\frac{4}{\kap}\right\}$), we can compute~$f,r,s$ using the following formulas:
    \begin{align}\label{eq: intro f}
    \begin{split}
        f&=\frac{(p+q)^2+(a-b+c-d)^2-\kap e(a-b)(c-d)}{4e\left(1-\frac{\kap e}{4}\right)} \, ,\\
        r&=\frac{p\left(e+a-d-\frac{\kap ae}{2}\right)+q\left(e-c+b-\frac{\kap be}{2}\right)}{2e\left(1-\frac{\kap e}{4}\right)}\, ,\\
        s&=\frac{p\left(e-a+d-\frac{\kap de}{2}\right)+q\left(e+c-b-\frac{\kap ce}{2}\right)}{2e\left(1-\frac{\kap e}{4}\right)}\, .
        \end{split}
    \end{align}
\end{proposition}
Proposition~\ref{prop: 2.14 analogue intro} shows that all ten basic measurements for a spherical quadrilateral can be explicitly written as rational expressions in terms of a subset of seven measurements (five~$x_{ij}$’s and two~$S^\kap$’s) associated with a particular triangulation of the quadrilateral. 
\begin{remark}
    In the special case~$\kap=0$, we recover the analogous result obtained in~\cite{fomin_heronian_2021}
    for tuples arising from quadrilaterals in the Euclidean plane. 
\end{remark}

\begin{definition}[cf.~\cite{fomin_heronian_2021} and Remark~\ref{rem: def equiv}]\label{def: intro def}
    Fix~$\kap\in\R$. A \emph{spherical Heronian diamond} is an ordered 10-tuple of numbers~$(a,b,c,d,e,f,p,q,r,s)$ arranged in the diamond pattern shown in Figure~\ref{fig: intro quad} and satisfying equations~\eqref{eq: four instances}--\eqref{eq: intro f}. Explicitly,
    \begin{itemize}
        \item $a, c, e, f$ are placed at the corners of the diamond;
        \item $p, q, r, s$ are placed at the midpoints of the sides of the diamond;
        \item the dashed line heading Southeast (resp. Southwest) is labeled~$d$ (resp.~$b$).
    \end{itemize}
(The entries are arranged so that the four ``measurements'' associated with each side of the diamond (e.g.,~$a, d, e, q$) satisfy the spherical Heron equation~\eqref{eq: her eq}.) 
\end{definition}


We finally set our sights on spherical polygons. Let~$P=A_1A_2\cdots A_n$ be a spherical polygon.  Each subset of four vertices~$A_i,A_j,A_k,A_l$ forms a spherical quadrilateral, which in turn gives rise to a spherical Heronian diamond consisting of ten measurements. We assemble some of these diamonds into a numerical array that we call a \emph{spherical Heronian frieze}, by analogy with~\cite{fomin_heronian_2021}. To~get a sense of what a spherical Heronian frieze looks like, see Figure~\ref{fig: sph heron example intro}. 

\begin{figure}[htbp!]
        \centering
        \begin{tikzpicture}
        [scale=3, greynode/.style={rectangle,fill=black!5,inner sep=2pt}, whitenode/.style={rectangle,fill=white, inner sep=2pt}, bordergreynode/.style={rectangle,draw=magenta,very thick,fill=white!5,inner sep=3pt}, borderwhitenode/.style={rectangle,draw=blue,fill=white,very thick,inner sep=3pt}]        

        \begin{scope}[on background layer] 
\filldraw[black!5] (-0.75,2) -- (-0.75,1.75) -- (1,0) -- (3,2) -- (-0.75,2);
        \draw[very thick] (-0.5,0.5) -- (0,0);

\draw[very thick] (-0.5,1.5) -- (1,0);

\draw[very thick] (0,2) -- (2,0);

\draw[very thick] (1,2) -- (3,0);

\draw[very thick] (2,2) -- (4,0);

\draw[very thick] (3,2) -- (4,1);

\draw[very thick](4,2) -- (4, 2);

\begin{scope}

\draw[dashed] (-0.75,.75) 
        --(.75,2.25) node[whitenode] {$18$};

        \draw[dashed] (-0.5,0) 
        --(1.75,2.25) node[whitenode] {$30$};

\draw[dashed] (0.25,-0.25)  node[whitenode] {$2$}
        -- (2.75 ,2.25);

\draw[dashed] (1.25,-0.25)  node[whitenode] {$6$}
        --(3.75,2.25);
        
\draw[dashed] (2.25,-0.25)  node[whitenode] {$18$}
        --(4.25,1.75);  

\draw[dashed] (3.25,-0.25) node[whitenode] {$30$} --(4.25,0.75);

\draw[dashed] (-0.75,1.25)  --(0.75,-0.25) node[whitenode] {$2$};

\draw[dashed] (-0.5,2)  --(1.75,-0.25) node[whitenode] {$6$};

\draw[dashed] (0.25,2.25)   --(2.75,-0.25) node[whitenode] {$18$};

\draw[dashed] (1.25,2.25) node[whitenode] {$30$} --(3.75,-0.25);

\draw[dashed] (2.25,2.25) node[whitenode] {$2$} --(4.25,0.25);

\draw[dashed] (3.25,2.25) node[whitenode] {$6$} --(4.25,1.25);

\end{scope}
    \end{scope}

\draw[very thick](-0.5,1.5) node[whitenode] {$6$} -- (0,2) node[whitenode] {$0$} ;

\draw[very thick](-0.5,0.5) node[whitenode] {$30$} -- (0,1) node[whitenode] {$24$} -- (0.5,1.5) node[greynode] {$18$} -- (1,2) node[bordergreynode] {$0$};

        \draw[very thick](0,0) node[whitenode] {$0$} -- (0.5,0.5) node[bordergreynode] {$2$} -- (1,1) node[bordergreynode] {$12$} -- (1.5,1.5) node[bordergreynode] {$30$} -- (2,2) node[whitenode] {$0$};

        \draw[very thick] (1,0) node[bordergreynode] {$0$} -- (1.5,0.5) node[whitenode] {$6$} -- (2,1) node[whitenode] {$24$} -- (2.5,1.5) node[borderwhitenode] {$2$} -- (3,2) node[borderwhitenode] {$0$};
        
        \draw[very thick](2,0) node[whitenode] {$0$} -- (2.5,0.5) node[whitenode] {$18$} -- (3,1) node[borderwhitenode] {$12$} -- (3.5,1.5) node[whitenode] {$6$} -- (4,2) node[whitenode] {$0$};

        \draw[very thick](3,0) node[borderwhitenode] {$0$} -- (3.5,0.5) node[borderwhitenode] {$30$} -- (4,1) node[whitenode] {$24$};

        \draw[very thick](4,0) node[whitenode] {$0$};

        \filldraw [black] (-0.25,0.25) circle (0.5pt) node[whitenode] {\tiny$0$};

         \filldraw [black] (-0.25,0.75) circle (0.5pt) node[whitenode] {\tiny$-4$};

        \filldraw [black] (-0.25,1.25) circle (0.5pt) node[whitenode] {\tiny$-12$};

        \filldraw [black] (-0.25,1.75) circle (0.5pt) node[greynode] {\tiny$0$};

        \filldraw [black] (0.25,0.25) circle (0.5pt) node[whitenode] {\tiny \tiny$0$};

        \filldraw [black] (0.25,0.75) circle (0.5pt) node[whitenode] {\tiny$-4$};

        \filldraw [black] (0.25,1.25) circle (0.5pt) node[greynode] {\tiny$-12$};

        \filldraw [black] (0.25,1.75) circle (0.5pt) node[greynode] {\tiny$0$};

        \filldraw [black] (0.75,0.25) circle (0.5pt) node[bordergreynode] {\tiny \tiny$0$};

        \filldraw [black] (0.75,0.75) circle (0.5pt) node[bordergreynode] {\tiny$4$};

        \filldraw [black] (0.75,1.25) circle (0.5pt) node[greynode] {\tiny$-12$};

        \filldraw [black] (0.75,1.75) circle (0.5pt) node[greynode] {\tiny$0$};

        \filldraw [black] (1.25,0.25) circle (0.5pt) node[whitenode] {\tiny \tiny$0$};

        \filldraw [black] (1.25,0.75) circle (0.5pt)  node[greynode] {\tiny$4$};

        \filldraw [black] (1.25,1.25) circle (0.5pt)  node[bordergreynode] {\tiny$-12$};

        \filldraw [black] (1.25,1.75) circle (0.5pt) node[bordergreynode] {\tiny$0$};

        \filldraw [black] (1.75,0.25) circle (0.5pt)  node[whitenode] {\tiny$0$};

        \filldraw [black] (1.75,0.75) circle (0.5pt)  node[whitenode] {\tiny$-12$};

        \filldraw [black] (1.75,1.25) circle (0.5pt) node[greynode] {\tiny$-4$};

        \filldraw [black] (1.75,1.75) circle (0.5pt) node[greynode] {\tiny$0$};

        \filldraw [black] (2.25,0.25) circle (0.5pt)  node[whitenode] {\tiny$0$};

        \filldraw [black] (2.25,0.75) circle (0.5pt) node[whitenode] {\tiny$-12$};

        \filldraw [black] (2.25,1.25) circle (0.5pt) node[whitenode] {\tiny$-4$};

        \filldraw [black] (2.25,1.75) circle (0.5pt) node[greynode] {\tiny$0$};

 \filldraw [black] (2.75,0.25) circle (0.5pt) node[whitenode] {\tiny$0$};

        \filldraw [black] (2.75,0.75) circle (0.5pt) node[whitenode] {\tiny$-12$};

        \filldraw [black] (2.75,1.25) circle (0.5pt) node[borderwhitenode] {\tiny$4$};

        \filldraw [black] (2.75,1.75) circle (0.5pt) node[borderwhitenode] {\tiny$0$};

        \filldraw [black] (3.25,0.25) circle (0.5pt) node[borderwhitenode] {\tiny$0$};

        \filldraw [black] (3.25,0.75) circle (0.5pt) node[borderwhitenode] {\tiny$-12$};

        \filldraw [black] (3.25,1.25) circle (0.5pt) node[whitenode] {\tiny$4$};

        \filldraw [black] (3.25,1.75) circle (0.5pt) node[whitenode] {\tiny$0$};

        \filldraw [black] (3.75,0.25) circle (0.5pt) node[whitenode] {\tiny$0$};

        \filldraw [black] (3.75,.75) circle (0.5pt) node[whitenode] {\tiny$-4$};

        \filldraw [black] (3.75,1.25) circle (0.5pt) node[whitenode] {\tiny$-12$};

        \filldraw [black] (3.75,1.75) circle (0.5pt) node[whitenode] {\tiny$0$};

        \end{tikzpicture}
        \caption{An example of a spherical Heronian frieze. To~see the glide symmetry, flip the gray shaded region upside down and paste it to the right of its original location. The magenta and blue entries form two traversing paths.}
        \vspace{-15pt}
        \label{fig: sph heron example intro}
    \end{figure}

We index the entries in a spherical Heronian frieze of order~$n$ by the set
\begin{equation*}
\left\{(i,j)\in\left(\Z\times\Z\right)\cup\left(\textstyle\frac{1}{2}\Z\times\Z\right)\cup\left(\Z\times\textstyle\frac{1}{2}\Z\right) \mid i\le j\le i+n\right\}\subseteq\left(\textstyle\frac{1}{2}\Z\right)^2
\end{equation*}
which is rotated by~$\pi/4$ so that the points~$(i,i)$ lie along the bottom of the strip and the points~$(i,i+n)$ run along the top, cf.~Figure~\ref{fig: sph heron index set}. We also associate an additional entry of a spherical Heronian frieze with each half-integer line
\begin{equation}\label{eq: line index set}
 \left\{\left(i+\textstyle\frac{1}{2},y\right)\mid y\in\R\right\} \text{ or }\left\{\left(x,j+\textstyle\frac{1}{2}\right)\mid x\in\R\right\} 
\end{equation}
(see the dashed lines in Figure~\ref{fig: sph heron example intro}). We denote these lines~$(i,\neline)$ and~$(\seline,j)$, respectively.

A spherical polygon~$P=A_1A_2\cdots A_n$ gives rise to a spherical Heronian frieze as follows.  Working modulo~$n$, we place each squared distance~$x_{ij}$ at the integer point~$(i,j)$ in the frieze. We place triangle measurements~$S^\kap(A_i,A_{i+1},A_j)$ and~$S^\kap(A_i,A_j,A_{j+1})$ at~$(i+\frac{1}{2},j)$ and~$(i,j+\frac{1}{2})$, respectively. Note that we only capture the~$S^\kap$ measurement for triangles that share a side with~$P$.  We associate the squared distances~$x_{i,i+1}$ and~$x_{j,j+1}$ with the lines~$(i,\neline)$ and~$(\seline,j)$ (see~\eqref{eq: line index set}). This arrangement of data mirrors the classical frieze patterns originally defined by Coxeter~\cite{coxeter_frieze_1971} in connection with the geometric phenomenon of the \emph{pentagramma mirificum}.

Spherical Heronian friezes provide our first solution to the original problem. We start with a triangulation~$T$ where each triangle shares a side with the polygon.  The corresponding initial data consists of~$2n-3$ squared distances~$x_{ij}$ associated with the sides and diagonals in~$T$ and~$n-2$ measurements~$S^\kap$ associated with the triangles in~$T$. Beginning the construction of a spherical Heronian frieze, we place this data at the appropriate locations as described above. (These locations form a \emph{traversing path} in the frieze going from top to bottom, see Figure~\ref{fig: sph heron example intro}.) We then recursively apply Proposition~\ref{prop: 2.14 analogue intro} to propagate outwards and recover the rest of the entries in the frieze, hence the remaining measurements of the spherical polygon.  The outwards propagation from a traversing path follows the pattern of the Conway-Coxeter  correspondence between triangulations of polygons and integral frieze patterns~\cite{conway_triangulated_1973}.

The rational formulas resulting from the propagation process are structurally simpler than one might expect. Like the classical Coxeter-Conway friezes, the entries in a spherical Heronian frieze  satisfy a version of the \emph{Laurent phenomenon}:
\begin{proposition}[cf.~Proposition~\ref{prop: analogue of cor 4.10}]\label{prop: laurent}
    Let~$\zpi$ be the initial data associated with a traversing path~$\pi$ in a spherical Heronian frieze~$\zfrieze$. Then each entry of~$\zfrieze$ can be written as a rational function of~$\zpi$ whose denominator is a product of (powers of)
    factors of the form~$x_{ij}$ and~$\ol{x}_{ij}=1-\frac{\kap}{4} x_{ij}$, for~$x_{ij}\in\zpi$.
\end{proposition}

\begin{remark}
     Proposition~\ref{prop: laurent} implies that the formulas for the entries in a spherical Heronian frieze arising from a spherical polygon~$A_1\cdots A_n$ are computationally stable as long as each pair of points~$(A_i,A_j)$,~$(i,j)\in\pi$, is neither coincident or antipodal. 
\end{remark}

Like classical frieze patterns, spherical Heronian friezes possess \emph{glide symmetry}.
\begin{theorem}[cf.~Theorem~\ref{thm: glide symm}]
Let~$\zpi$ be a sufficiently generic collection of numbers associated to a traversing path~$\pi$. Assume that these numbers satisfy the appropriate spherical Heron equations. Propagate outwards using equations~\eqref{eq: intro f} to obtain a spherical Heronian frieze~$\zfrieze$ of order~$n$.  Then~$\zfrieze$ is~$n$-periodic: the entries at locations~$(k,\ell)$ and~$(k+n,\ell+n)$ are equal.  In fact,~$\zfrieze$ possesses glide symmetry, cf.~Figure~\ref{fig: sph heron example intro}.
\end{theorem}

We now pivot to a wholly different approach to the original spherical distance geometry question. We supplement the initial~$2n-3$ measurements coming from a triangulation with~$n-1$ additional diagonal measurements, which can be thought of as ``bracing'' edges, cf.~\cite{jordan_global_2019}. Unlike the original~$2n-3$ squared distances, the expanded collection is not algebraically independent: they satisfy algebraic relations which we will state next.

\begin{lemma}[$5\times 5$ spherical Cayley-Menger determinant, \cite{tao_spherical_2019}, cf.~Theorem~\ref{thm: scm}]
    Let~$A_1,A_2,A_3,A_4$ be four points on the sphere~$\bfS$.
Then 
\begin{equation}\label{eq: sph caym eq}
\begin{bmatrix}
    \frac{\kap}{2} & 1 & 1 & 1 & 1\\
    1 & 0 & x_{12}  & x_{13} & x_{14} \\
    1 & x_{12} & 0 & x_{23} & x_{24}\\
    1 & x_{13} & x_{23} & 0 & x_{34} \\
    1 & x_{14} & x_{24} & x_{34} & 0
    \end{bmatrix}=0\, .
\end{equation}
\end{lemma}
We define a \emph{spherical Cayley-Menger diamond} (cf.~\cite{fomin_heronian_2021}) to be a 6-tuple of numbers~$x_{12},\dots,x_{34}$ satisfying equation~\eqref{eq: sph caym eq} and arranged as shown below:
\begin{center}
\begin{tikzpicture}
            [scale=0.6,every node/.style={rectangle,fill=white,inner sep=2pt}]

        \begin{scope}[on background layer]
        
        \draw[dashed] (-2,4) node {$x_{34}$}
        --(2,0);
        \draw[dashed] (-2,0) node {$x_{12}$}
        --(2,4);
    \end{scope}
        \draw[very thick]
        (0,0) node {$x_{23}$}
        -- (-2,2) node {$x_{13}$}
        -- (0,4) node {$x_{14}$}
        -- (2,2) node {$x_{24}$}
        -- (0,0);
        \end{tikzpicture}
\end{center}  
We tile these diamonds to form a spherical Cayley-Menger frieze, see Figure~\ref{fig: cayley-menger frieze example intro}.
\begin{figure}[htbp!]
        \centering
        \begin{tikzpicture}
        [scale=2,greynode/.style={rectangle,fill=black!5,inner sep=2pt}, whitenode/.style={rectangle,fill=white, inner sep=2pt}, bordergreynode/.style={rectangle,draw=magenta,very thick,fill=white!5,inner sep=3pt}, borderwhitenode/.style={rectangle,draw=blue,fill=white,very thick,inner sep=3pt}]

        \begin{scope}[on background layer] 
\filldraw[black!5] (-0.75,2) -- (-0.75,1.75) -- (1,0) -- (4,3) -- (-0.75,3) -- (-0.75,2);
        \draw[very thick](-0.5,0.5) -- 
(0,0);

\draw[very thick](-0.5,1.5) -- (1,0);

\draw[very thick](-0.5,2.5) -- (2,0);

\draw[very thick](0,3) -- (3,0);

\draw[very thick](1,3) -- (4,0);

\draw[very thick](2,3) -- (5,0);

\draw[very thick](3,3) -- (5.5,0.5);

\draw[very thick](4,3)-- (5.5,1.5);

\draw[very thick](5,3) -- (5.5,2.5);

\begin{scope}
\draw[dashed] (-0.75,1.75) 
        --(0.75,3.25) node[whitenode] {\tiny$98$}; 

\draw[dashed] (-0.75,.75) 
        --(1.75,3.25) node[whitenode] {\tiny$74$};

        \draw[dashed] (-0.6,-0.1) node[whitenode] {\tiny$26$}
        --(2.75,3.25) node[whitenode] {\tiny$26$};

\draw[dashed] (0.25,-0.25)  node[whitenode] {\tiny$140$}
        -- (3.75 ,3.25) node[whitenode] {\tiny$140$};

\draw[dashed] (1.25,-0.25)  node[whitenode] {\tiny$70$}
        --(4.75,3.25) node[whitenode] {\tiny$70$};
        
\draw[dashed] (2.25,-0.25)  node[whitenode] {\tiny$74$}
        --(5.6,3.1) node[whitenode] {\tiny$74$};  

\draw[dashed] (3.25,-0.25) node[whitenode] {\tiny$98$} --(5.75,2.25);

\draw[dashed] (4.25,-0.25) node[whitenode] {\tiny$74$} --(5.75,1.25);

\draw[dashed] (-0.75,1.25)  --(0.75,-0.25) node[whitenode] {\tiny$140$};

\draw[dashed] (-0.75,2.25)  --(1.75,-0.25) node[whitenode] {\tiny$70$};

\draw[dashed] (-0.6,3.1)  node[whitenode] {\tiny$74$} --(2.75,-0.25) node[whitenode] {\tiny$74$};

\draw[dashed] (0.25,3.25) node[whitenode] {\tiny$98$} --(3.75,-0.25) node[whitenode] {\tiny$98$};

\draw[dashed] (1.25,3.25) node[whitenode] {\tiny$74$} --(4.75,-0.25) node[whitenode] {\tiny$74$};

\draw[dashed] (2.25,3.25) node[whitenode] {\tiny$26$} --(5.6,-0.1) node[whitenode] {\tiny$26$};

\draw[dashed] (3.25,3.25) node[whitenode] {\tiny$140$} -- (5.75,0.75);

\draw[dashed] (4.25,3.25) node[whitenode] {\tiny$70$} -- (5.75,1.75);

\end{scope}
    \end{scope}
\draw[very thick](-0.5,2.5) node[greynode] {$74$} -- (0,3) node[whitenode] {$0$};
    
\draw[very thick](-0.5,1.5) node[whitenode] {$14$} -- (0,2) node[bordergreynode] {$50$} -- (0.5,2.5) node[bordergreynode] {$98$} -- (1,3) node[whitenode] {$0$};

\draw[very thick](-0.5,0.5) node[whitenode] {$26$} -- (0,1) node[whitenode] {$126$} -- (0.5,1.5) node[bordergreynode] {$52$} -- (1,2) node[greynode] {$170$} -- (1.5,2.5) node[greynode] {$74$} -- (2,3) node[whitenode] {$0$};

        \draw[very thick](0,0) node[whitenode] {$0$} -- (0.5,0.5) node[bordergreynode] {$140$} -- (1,1) node[bordergreynode] {$26$} -- (1.5,1.5) node[greynode] {$116$} -- (2,2) node[greynode] {$106$} -- (2.5,2.5) node[greynode] {$26$} -- (3,3) node[whitenode] {$0$};

        \draw[very thick](1,0) node[whitenode] {$0$} -- (1.5,0.5) node[whitenode] {$70$} -- (2,1) node[whitenode] {$56$} -- (2.5,1.5) node[whitenode] {$14$} -- (3,2) node[whitenode] {$126$} -- (3.5,2.5) node[borderwhitenode] {$140$} -- (4,3) node[whitenode] {$0$};
        
        \draw[very thick](2,0) node[whitenode] {$0$} -- (2.5,0.5) node[whitenode] {$74$} -- (3,1) node[borderwhitenode] {$50$} -- (3.5,1.5) node[borderwhitenode] {$52$} -- (4,2) node[borderwhitenode] {$26$} -- (4.5,2.5) node[whitenode] {$70$} -- (5,3) node[whitenode] {$0$};

        \draw[very thick](3,0) node[whitenode] {$0$} -- (3.5,0.5) node[borderwhitenode] {$98$} -- (4,1) node[whitenode] {$170$} -- (4.5,1.5) node[whitenode] {$116$} -- (5,2) node[whitenode] {$56$} -- (5.5,2.5) node[whitenode] {$74$};

        \draw[very thick](4,0) node[whitenode] {$0$} -- (4.5,0.5) node[whitenode] {$74$} -- (5,1) node[whitenode] {$106$} -- (5.5,1.5) node[whitenode] {$14$};

        \draw[very thick](5,0) node[whitenode] {$0$} -- (5.5,0.5) node[whitenode] {$26$};

        \end{tikzpicture}
        \caption{A spherical Cayley-Menger frieze of order 6.}
        \label{fig: cayley-menger frieze example intro}
    \end{figure}

    The spherical Cayley-Menger determinant is quadratic in the squared distances, so it cannot be used  to propagate unambiguously from an initial data set associated with a traversing path. We solve this problem by deriving an additional \emph{coherence relation} satisfied by the entries of four interlocking Cayley-Menger diamonds in a Cayley-Menger frieze arising from a polygon; we call a Cayley-Menger frieze \emph{coherent} if this relation is satisfied throughout.  Crucially, the coherence equation is linear with respect to the rightmost entry in a configuration of four interlocking diamonds. This allows outwards propagation.

    Our second solution then proceeds similarly to the first.
    As propagation via the coherence equation involves four Cayley-Menger diamonds, we have to  begin with initial data along a \emph{thickening} of a traversing path, consisting of the entries on the path together with another path next to it. This initial data set contains the squared distances corresponding to a triangulation of an~$n$-gon, together with the squared lengths of~$n-3$ extra diagonals. We then obtain the following theorem.
    \begin{theorem}[cf.~Theorem~\ref{thm: 5.18 analogue}]
    A sufficiently generic coherent spherical Cayley-Menger frieze is uniquely determined by the entries lying on the thickening of a traversing path.
    \end{theorem}

    Spherical Heronian friezes and spherical Cayley-Menger friezes turn out to be closely related.
    \begin{theorem}[cf.~Theorem~\ref{thm: big final thm}]
        
    \begin{enumerate}[leftmargin=*]
    \item[]
        
        \item Given a sufficiently generic spherical Heronian frieze, restricting to the integer-indexed entries (equivalent to forgetting the triangle measurements~$\skap$) yields a spherical Cayley-Menger frieze.

        \item Given a sufficiently generic coherent spherical Cayley-Menger frieze~$\zscm$, there are exactly two spherical Heronian friezes that extend~$\zscm$.  These two friezes differ from each other by a global change in sign.
    \end{enumerate}
    
    \end{theorem}

\subsection*{Structure of the paper}

Sections~\ref{sec: triangles}-\ref{sec: denominators} develop spherical Heronian friezes and their properties.  We begin in Section~\ref{sec: triangles} with geometric preliminaries including the spherical triangle measurement~$\skap$, Heron's formula, and the Cayley-Menger determinant. 
We expand in Section~\ref{sec: quadrilaterals} from spherical triangles to spherical quadrilaterals, establishing a set of algebraic identities satisfied by the measurements associated to a spherical quadrilateral.  
In Section~\ref{sec: sph her friezes} we define spherical Heronian diamonds and friezes, provide examples, and give our first solution to the main problem.
The discussion of spherical Heronian friezes continues in Section~\ref{sec: denominators} where we establish a version of the Laurent phenomenon.

In Section~\ref{sec: cayley menger} we develop our second solution. We begin with a discussion of the spherical analogue of the Cayley-Menger determinant. We define spherical Cayley-Menger friezes, provide examples, and establish the coherence equation in the geometric setting. In Section~\ref{sec: frieze equiv}, we establish the relationship between the two kinds of spherical friezes.

\section*{Acknowledgments}

I would like to thank Sergey Fomin for many productive discussions that were integral to the development of this work, and for editorial comments that greatly improved the final product. I performed all computations in \texttt{Sage}, and explored examples using \texttt{Geogebra}.

The main results of this paper have been reported at several seminars and at 
the conference ``Frieze patterns in algebra, combinatorics and geometry'' held at the Centre International de Rencontres Mathematiques, Luminy, in May 2025.
Thank you to the organizers of these events for the invitations. 
I thank Sergei Tabachnikov, Anna Felikson, and Linus Setiabrata for interesting discussions and feedback.

\newpage

\section{Measurements associated to spherical triangles}\label{sec: triangles}

Let~$V_\R\cong \R^3,$ a three-dimensional real Euclidean space, with a scalar product 
\begin{align*}
    \langle \, ,  \rangle: V_\R\times V_\R&\to \R\\
    (x,y)&\mapsto \langle x, y\rangle
\end{align*}
that is symmetric, bilinear, and positive definite.  Define a trilinear skew-symmetric volume form
\begin{align*}
    [ \, ] : V_\R\times V_\R\times V_\R&\to \R\\
    (x,y,z)& \mapsto
    \begin{dietmatrix}{\begin{bmatrix}
    x & y & z
    \end{bmatrix}}
    \end{dietmatrix}.
\end{align*}
In Cartesian coordinates, if~$x=\left[\begin{smallmatrix} x_1\\x_2\\x_3\end{smallmatrix}\right]$, $y=\left[\begin{smallmatrix}
    y_1 \\ y_2 \\ y_3
\end{smallmatrix}\right]$,  and~$z=\left[\begin{smallmatrix}
    z_1 \\ z_2\\ z_3
\end{smallmatrix}\right]$, then 
\begin{equation*}
    \langle x,y\rangle = x_1y_1 + x_2y_2 + x_3y_3\, ,
    \quad \text{ and } \quad 
    \begin{dietmatrix}\begin{bmatrix}
        x & y & z
    \end{bmatrix}
    \end{dietmatrix}=\det\left[\begin{smallmatrix}
        x_1 & y_1 & z_1 \\
        x_2 & y_2 & z_2\\
        x_3 & y_3 & z_3
    \end{smallmatrix}\right].
\end{equation*}
This construction can be complexified, producing a three-dimensional complex vector space~$V=V_\C$ endowed with both a scalar product and a volume form given by the same formulas.  While the reader can consider the results of this paper in the context of~$V_\R\cong\R^3$ to aid understanding, we will generally work in~$V= V_\C$.

\begin{definition}\label{def: squared distance}
We will implicitly identify~$V$ with an affine space that has a distinguished point~$O$, the origin. For two points~$A,B\in V$, we denote by~$\overrightarrow{AB}$ a vector from~$A$ to~$B$ and by~$x(A,B)$ the scalar product 
\begin{equation*}\label{eq: euc dist}
    x(A,B)=\langle \vec{AB},\vec{AB}\rangle.
\end{equation*}
Note that~$x(A,B)=x(B,A).$ In~$V_\R=\R^3$, the quantity~$x(A,B)$ is the square of the Euclidean distance from~$A$ to~$B$. For points~$A_1,A_2,\dots$, we denote~${x_{ij}=x(A_i,A_j)}$.
\end{definition}

Let~$\bfS\subseteq V$ be a sphere of radius~$R$ centered at~$O$. We view~$\bfS$ as a space of constant Gaussian curvature~$\kap=\mytfrac{1}{R^2}$.  All points of interest to us will lie on~$\bfS$ with the exception of~$O$.  Let~$\Aut(\bfS)$ denote the group of orientation-preserving isometries of~$\bfS$. 

We will work with triangulations of polygons on~$\bfS$ (cf.~Definitions~\ref{def: sph poly}-\ref{def: tri polygon}). As shown below, up to the action of~$\Aut(\bfS)$, there are two triangles on~$\bfS$ with the squared side distances~$a,b,c$. We will need an additional measurement that distinguishes such triangles from each other.
\begin{figure}[htbp!]
    \centering
    \begin{tikzpicture}
\coordinate  (A) at (-3,0);
    \node at (A) [left] {$A$};
\circleat{-3}{0};
\coordinate  (B) at (-1,0);
    \node at (B) [right] {$B$};
\circleat{-1}{0};
\coordinate  (C) at (-1.5,.5);
    \node at (C) [above, inner sep=5pt] {$C$};
    \circleat{-1.5}{.5};

\draw (A) to node[below] {$c$} (B) ;
\draw (B) to node[above right] {$a$} (C) ;
\draw (A) to node[above] {$b$} (C) ;

\coordinate  (A) at (3,0);
    \node at (A) [right] {$A$};
    \circleat{3}{0};

\coordinate  (B) at (1,0);
    \node at (B) [left] {$B$};
\circleat{1}{0};
\coordinate  (C) at (1.5,.5);
    \node at (C) [above, inner sep=5pt] {$C$};
    \circleat{1.5}{.5};
\draw (A) to node[below] {$c$} (B) ;
\draw (B) to node[above left] {$a$} (C) ;
\draw (A) to node[above] {$b$} (C) ;

    \end{tikzpicture}
    
    \label{fig: three distances}
\end{figure}

In two-dimensional real Euclidean space, a natural way to distinguish triangles is by their signed oriented area. Unfortunately, there is no consistent way to extend this notion to the three-dimensional space~$V$ or to a two dimensional sphere~$\bfS\subseteq V$. 
 Instead, we will use the following skew-symmetric measurement~$\skap$  that is invariant under cyclic permutations of the vertices of a triangle.
 
 \begin{definition}\label{def: skappa def}
For~$A_1,A_2,A_3\in\bfS$, consider the Euclidean tetrahedron~$OA_1A_2A_3$ with vertices~$O,A_1,A_2,A_3$, as shown below.
\begin{figure}[htbp!]
    \centering
    \begin{tikzpicture}[scale=0.8]
        \coordinate [label=left:$O$] (O) at (-1,0);
        \circleat{-1}{0};
        \coordinate [label=below right:$A_3$] (A3) at (2,1);
        \circleat{2}{1};
        \coordinate [label=above left:$A_1$] (A1) at (0,2);
        \circleat{0}{2};
        \coordinate [label=above:$A_2$] (A2) at (1,2);
        \circleat{1}{2};

\draw (A1) to  (A2) to (A3) to (A1) ;
\draw (A1) to (O) to (A2) ;
\draw (O) to (A3);
        \coordinate  (left) at (-1,3);
        \coordinate (middle) at (2.5,2);
        \coordinate  (right) at (3.5,0);

    \end{tikzpicture}
   
    \label{fig: tetrahedron 1}
\end{figure}

\noindent The signed Euclidean volume~$V$ of~$OA_1A_2A_3$ is given by
\begin{equation}\label{eq:voltet1}
    V(OA_1A_2A_3)=\mytfrac{1}{6} \dietmatrix{\begin{bmatrix}
\overrightarrow{OA_1} & \overrightarrow{OA_2} & \overrightarrow{OA_3}
    \end{bmatrix}}\, .
\end{equation}
    Define the measurement
\begin{equation}
        \skap(A_1,A_2,A_3)=\mytfrac{12}{R}V(OA_1A_2A_3)\, .
    \end{equation}
 When convenient, for~$A_i,A_j,A_k\in\bfS$ we will abbreviate~$\skapijk=\skap(A_i,A_j,A_k)$. The measurement~$\skap$ is sensitive to the specified ordering of points: when the points~$A_i,A_j,A_k$ are permuted,~$\skap$ changes its sign depending on whether the permutation is odd or even.
Informally,~$\skap$ is a spherical analogue for the rescaled signed area of a triangle~$A_1A_2A_3$ (see Remarks~\ref{rem: area motivation},~\ref{rem: four}). 
\end{definition}

\begin{remark}\label{rem: area motivation}
    Let~$OA_1A_2A_3$ be a tetrahedron in~$V_\R$. Let~$h$ be the unsigned distance from~$O$ to the plane~$(A_1A_2A_3)$. The unsigned volume of~$OA_1A_2A_3$ is given by
    \begin{equation}\label{eq: unsigned}
    \begin{array}{c}
        \text{\small unsigned volume of~$OA_1A_2A_3$} 
    \end{array}={\mytfrac{1}{3}}\cdot (\text{\small unsigned area of triangle }A_1A_2A_3 )\cdot h\, .
    \end{equation}
    When points~$A_1,A_2,A_3$ are close together on~$\bfS$, the radius~$R$ is a good approximation of height~$h$. 
   Equation~\eqref{eq: unsigned} motivates our use of~$\mytfrac{3}{R} V(OA_1A_2A_3)$
    in lieu of the signed area of the triangle~$A_1A_2A_3$.
\end{remark}

One can measure volumes in an~$n$-dimensional Euclidean space using the~$n\times n$ Cayley-Menger determinant, which we define next.

\begin{definition}[{\cite[Section 9.7.3]{berger_geometry_1977}}]\label{def: 3x3 euclidean cayley-meng}
    Let~$A_1,A_2,A_3\in V$. Denote~$x_{ij}=x(A_i,A_j)$.
    The associated \emph{Cayley-Menger determinant} is given by 
    \begin{equation*}\label{eq: CM_3}
    \caymthree(x_{12},x_{13},x_{23})=\det\left[\begin{smallmatrix}
        0 & 1 & 1 & 1\\
        1 & 0 & x_{12} & x_{13}\\
        1 & x_{12} & 0 & x_{23} \\
        1 & x_{13} & x_{23} & 0  \\
    \end{smallmatrix}\right]\, .
    \end{equation*}
    The Cayley-Menger determinant~$\caymk$ associated to a collection of~$k$ points in a Euclidean space is  defined analogously.  (The superscript~0 in our notation is explained in Remark~\ref{rem: explaining zero}.)
\end{definition}

\begin{proposition}[{\cite[Section~10.6.5]{berger_geometry_1977}}]\label{prop: gencaymeng}
    Consider a simplex~$A_1\cdots A_{n+1}$ in~$\R^n$, with volume~$V$.  Then 
    \begin{equation}\label{eq: vol of simplex}
        V^2=\frac{(-1)^{n+1}}{(n!)^2\cdot 2^n}\caymnplusone(x_{12},x_{13},\dots,x_{23},x_{24},\dots,x_{n,n+1})\, .
    \end{equation}
\end{proposition}

\begin{remark}
Cayley obtained this formula for~$n\le 4$ and used it to derive an algebraic relation amongst the  distances between~$n+1$ points in an~$(n-1)$-dimensional Euclidean space. Menger extended Cayley's result to compute the volume of simplices in any dimension~\cite{liberti_six_2015}. 

The~$n=3$ case of~\eqref{eq: vol of simplex} is a classical result (e.g.~{\cite[Section~161]{todhunter_spherical_1886}})
 that {\cite[Section~68]{dorrie_100_1965}} attributes to Euler. The~$n=2$ case can be viewed as a restatement of the well-known \emph{Heron's formula}, see below. As discussed in~\cite{heath_history_1921}, while Heron published it in his work \emph{Metrica} sometime before~250~CE, this was a collection of both new and previously known results and Archimedes knew Heron's formula two centuries prior to Heron.
\end{remark}
\begin{theorem}[Heron's Formula]\label{thm: heron}
Let~$A_1,A_2,A_3\in\R^2$. Denote $x_{ij}=x(A_i,A_j)$. For each pair $1\le i<j\le 3$, select a sign $\pm$ and set
\begin{equation*}
    y=\pm\sqrt{x_{12}}, \quad
    z=\pm\sqrt{x_{13}}, \quad
    w=\pm\sqrt{x_{23}}, \quad
    s=\mytfrac{y+z+w}{2}.
\end{equation*}
Regardless of the choices made, we then have
\[(\textnormal{area of } A_1A_2A_3)^2=s(s-y)(s-z)(s-w).\]
\end{theorem}
It is instructive to deduce this classical result from Proposition~\ref{prop: gencaymeng}:
\begin{proof}
    By Proposition~\ref{prop: gencaymeng}, 
    \begin{align}
        (4\cdot \text{area of } A_1A_2A_3)^2
        &=-\caymthree(x_{12},x_{13},x_{23})\label{eq: cm area heron}\\
        &=-x_{12}^2-x_{13}^2-x_{23}^2+2x_{12}x_{13}+2x_{12}x_{23}+2x_{13}x_{23}\nonumber\\
        &=-y^4-z^4-w^4+2y^2z^2+2y^2w^2+2z^2w^2\nonumber\\
        &=(y+z+w)(-y+z+w)(y-z+w)(y+z-w)\nonumber\qedhere
    \end{align}
\end{proof}
\begin{remark}\label{rem: four}
The factor of~$4$ in equation~\eqref{eq: cm area heron} motivates the appearance of the same factor in the definition of~$\skap$.
\end{remark}

We will show that our measurement~$\skap$ satisfies a spherical version of Heron's formula.  In order to state this formula, we introduce the following notation.

\begin{definition}
    For~$a,b,c\in\C$, define
\begin{equation*}\label{eq: h kappa def}
    \hkap (a,b,c) =-
    a^2-b^2-c^2+2ab+2ac+2bc-\kap abc\, .
    \end{equation*}
\end{definition}

\begin{proposition}[Spherical Heron's Formula, cf.~Remark~\ref{rem: dim three cay meng}]\label{prop: sph heron} 
Let~$A_1,A_2,A_3\in\bfS$. Denote~${x_{ij}=x(A_i,A_j)}$.  Then
\begin{equation*}\label{eq:kappaheron}
    \hkap (x_{12},x_{13},x_{23}) =  \left(\skap(A_1,A_2,A_3)\right)^2\, .
\end{equation*}
   \end{proposition}

\begin{proof}
Note that~$x(O,A_i)=R^2$ for~$i=1,2,3$. Using Proposition~\ref{prop: gencaymeng} with~$A_4=O$, we get:
\begin{align*}
(\skap(A_1,A_2,A_3))^2&=\left(\mytfrac{12V(OA_1A_2A_3)}{R}\right)^2=\mytfrac{1}{2R^2}\caymfour(x_{12},x_{13},R^2,x_{23},R^2,R^2)\, .
\end{align*}
Therefore (recall that~$\kap=\frac{1}{R^2}$)
\begin{align*}
        (\skap(A_1,A_2,A_3))^2&=\mytfrac{1}{2R^2}\det\left[\begin{smallmatrix}
    0 & 1 & 1 & 1 & 1 \\
    1 & 0 & x_{12} & x_{13} & R^2 \\
    1 & x_{12} & 0 & x_{23} & R^2 \\
    1 & x_{13} & x_{23} & 0 & R^2 \\
    1 & R^2 & R^2 & R^2 & 0
\end{smallmatrix}\right]\\
&=-x_{12}^2-x_{13}^2-x_{23}^2+2x_{12}x_{13}+2x_{12}x_{23}+2x_{13}x_{23}-\kap x_{12}x_{13}x_{23}\\
&=\hkap(x_{12},x_{13},x_{23}).\qedhere
\end{align*}
\end{proof}

A spherical triangle $ABC$ is uniquely determined, up to the action of~$\Aut(\bfS)$, by its squared side distances~$a,b,c$ together with the additional data of~$\skap(A,B,C)$:
\begin{lemma}\label{lem: 2.4analogue} 
    Let~$A,C\in\bfS$ with~$x(A,C)=b\not\in\left\{ 0,\frac{4}{\kap}\right\}$. 
 Let~$a,c,p\in\C$ be such that 
    \begin{equation}\label{eq: pfact}
        p^2=\hkap (a,b,c)\, .
    \end{equation}  Then there exists a unique point~$B\in\bfS$ such that~$x(B,C)=a$,~$x(A,B)=c$, and~$\skap(A,B,C)=p$.

\end{lemma}
\begin{proof}
    Let~$A=\left[\begin{smallmatrix} u'\\u''\\u'''\end{smallmatrix}\right]$,~$C=\left[\begin{smallmatrix}
        v'\\v''\\v'''
    \end{smallmatrix}\right]$. Since~$A,C\in \bfS$, we have
    \begin{equation}\label{eq: veq}
        (u')^2+(u'')^2+(u''')^2=R^2  \qquad \text{ and } \qquad
        (v')^2+(v'')^2+(v''')^2=R^2\, .
    \end{equation}
    The condition~$x(A,C)=b$ gives
    \begin{equation*}
        (u'-v')^2+(u''-v'')^2+(u'''-v''')^2=b\, .
    \end{equation*}
    Combining these facts, we see that 
\begin{equation}\label{eq: bfact}
u'v'+u''v''+u'''v'''=R^2-\mytfrac{b}{2}\, .
    \end{equation}
    We seek a  point~$B=\left[\begin{smallmatrix}
        w'\\w''\\w'''
\end{smallmatrix}\right]\in\bfS$ with~$w',w'',w'''\in\C$ that satisfies the conditions
    \begin{align}
        (w')^2+(w'')^2+(w''')^2&=R^2\, , \nonumber\\
        (u'-w')^2+(u''-w'')^2+(u'''-w''')^2&=a\, ,\label{eq: adist}\\
        (v'-w')^2+(v''-w'')^2+(v'''-w''')^2&=c \label{eq: cdist}\, .
    \end{align}
    Combining these conditions and simplifying, equations~\eqref{eq: adist}-\eqref{eq: cdist} translate into
    \begin{align}  u'w'+u''w''+u'''w'''&=R^2-\mytfrac{a}{2}\, ,\label{eq: system 1}\\
v'w'+v''w''+v'''w'''&=R^2-\mytfrac{c}{2} \label{eq: system 2}\, .
    \end{align}
    To~achieve~$\skap(A,B,C)=p$ we need
\begin{equation}\label{eq: detvolp}
        p=\mytfrac{2}{R}\det\dietmatrix{\left[\begin{smallmatrix}
            u' & v' & w'\\
            u'' & v'' & w''\\
            u''' & v''' & w'''
\end{smallmatrix}\right]}\, .
    \end{equation}
    
    The system of three linear equations~\eqref{eq: system 1}-\eqref{eq: detvolp} in~$w',w'',w'''$ can be represented with the matrix
    \begin{equation}   \label{eq: matrix}
    \left[\begin{array}{ccc|c}
        u' & u'' & u''' & R^2-\frac{a}{2}\\
        v' & v'' & v''' & R^2-\frac{c}{2}\\
        \frac{2}{R}(u''v'''-v''u''') & \frac{2}{R}(u'''v'-u'v''') & \frac{2}{R}(u'v''-v'u'') & p
    \end{array}\right]\, ;
        \end{equation}
    the system will have a single solution if the determinant of the left~$3\times 3$ submatrix
    is non-zero.
    The Laplace expansion along the bottom row of the submatrix  is
\begin{equation*}
    \Delta:=\mytfrac{2}{R}(u'v''-v'u'')^2+\mytfrac{2}{R}(u'''v'-u'v''')^2+\mytfrac{2}{R}(u''v'''-v''u''')^2\, .
\end{equation*}
Let us verify
$\Delta\neq 0$
under the conditions given. From equations~\eqref{eq: veq} we  deduce:
\begin{align*}
    R^4&=
    (u'v')^2+(u''v'')^2+(u'''v''')^2+2(u'u''v'v''+u'u'''v'v'''+u''u'''v''v''')+\Delta\, .
\end{align*}
Recalling equation~\eqref{eq: bfact}, this implies
\begin{equation}\label{eq: deltafact}
    R^4=\left(R^2-\mytfrac{b}{2}\right)^2+\Delta \, .
    \end{equation}
Since~$b\not\in \left\{ 0, \frac{4}{\kap}\right\}$, we conclude that the system of equations represented by~\eqref{eq: matrix} has a unique solution, namely
\begin{align*}
    w'&=\mytfrac{1}{\Delta}\Bigl(\left(R^2-\mytfrac{a}{2}\right)(v''(u'v''-u''v')+v'''(u'v'''-u'''v'))\\
    &\quad\quad\quad -\left(R^2-\mytfrac{c}{2}\right)(u''(u'v''-u''v')+u'''(u'v'''-u'''v'))
    +\mytfrac{Rp}{2}(u''v'''-u'''v'')\Bigr)\, ,\\
    w''&=\mytfrac{1}{\Delta}\Bigl(\left(R^2-\mytfrac{a}{2}\right)(v'(u''v'-u'v'')+v'''(u''v'''-u'''v''))\\
    &\quad\quad\quad -
    \left(R^2-\mytfrac{c}{2}\right)(u'(u''v'-u'v'')+u'''(u''v'''-u'''v''))
    +\mytfrac{Rp}{2}(u'''v'-u'v''')\Bigr) \, ,\\
    w'''&=\mytfrac{1}{\Delta}\Bigl(\left(R^2-\mytfrac{a}{2}\right)(v'(u'''v'-u'v''')+v''(u'''v''-u''v'''))\\
    &\quad\quad\quad -
    \left(R^2-\mytfrac{c}{2}\right)(u'(u'''v'-u'v''')+u''(u'''v''-u''v'''))
    +\mytfrac{Rp}{2}(u'v''-u''v')\Bigr)\, .
\end{align*}

It remains to verify that~
$(w')^2+(w'')^2+(w''')^2=R^2.$
Expanding, making use of equations~\eqref{eq: veq}-\eqref{eq: bfact}, and simplifying, we obtain
\begin{equation*}
    (w')^2+(w'')^2+(w''')^2
     =\mytfrac{R^2}{\Delta}\left(R^2b-\mytfrac{ab}{2} 
     - \mytfrac{bc}{2} + \mytfrac{c^2}{4} +\mytfrac{a^2}{4}+ \mytfrac{p^2}{4}- \mytfrac{ac}{2}+\mytfrac{acb}{4R^2}\right)\, .
\end{equation*}
We are now able to  substitute for~$p^2$ using the given condition~\eqref{eq: pfact}.  We simplify the resulting expression using equation~\eqref{eq: deltafact} to complete the proof:
\begin{equation*}
    (w')^2+(w'')^2+(w''')^2=\mytfrac{R^2}{\Delta}\left(R^2b-\mytfrac{b^2}{4}\right)=R^2\, .\qedhere
\end{equation*}
\end{proof}

\begin{remark}\label{rem: why conditions}
Condition~$b\ne 0$ in Lemma~\ref{lem: 2.4analogue}  prevents points~$A$ and~$C$ from being coincident.  In the real setting, the condition~$b\ne \frac{4}{\kap}=(2R)^2$ prohibits~$A$ and~$C$ from being antipodal on the sphere. This prohibition is necessary: for antipodal~$A$ and~$C$, there are infinitely many points~$B$ satisfying the conditions in Lemma~\ref{lem: 2.4analogue}.
\end{remark}

\begin{lemma}\label{lem: 2.3 analogue}  Let~$a,b,c\in\C$ with at least one of~$a,b,c$ nonzero and not equal to~$\frac{4}{\kap}$. 
 If~$p^2=\hkap (a,b,c)$,
 then there exist points~$A,B,C\in\bfS$ such that
 \begin{equation*}
     x(B,C)=a, \quad 
     x(A,C)=b, \quad 
     x(A,B)=c, \quad \text{ and } \quad
     \skap(A,B,C)=p.
 \end{equation*}  
 Moreover, such a triple~$(A,B,C)$ is unique up to the action of~$\Aut(\bfS)$.
    
\end{lemma}

\begin{proof}
    Without loss of generality, we can assume that~$b\not \in\left\{ 0,\frac{4}{\kap}\right\}$.  Let~${A,C\in\bfS}$ such that~$x(A,C)=b$.  Then by Lemma~\ref{lem: 2.4analogue} there exists a unique point~$B\in \bfS$ such that~$x(B,C)=a$,~$x(A,B)=c$, and~$\skap(A,B,C)=p$.
\end{proof}

\newpage

\section{Spherical quadrilaterals}\label{sec: quadrilaterals}

Moving towards spherical polygons, in this section we will establish key algebraic relations satisfied by the squared side distances and~$\skap$ measurements formed by four points on the sphere~$\bfS$.
\begin{example}\label{ex: main tet example}
A spherical quadrilateral~$A_1A_2A_3A_4$ can be subdivided into two spherical triangles in two different ways.  
Each triangle (e.g.,~$A_1A_2A_3$) gives rise to a tetrahedron (say,~$OA_1A_2A_3$) as shown in Figure~\ref{fig:tetdecomp} below. The associated measurements are: 
\begin{align}
    & a=x_{14}, \quad b = x_{12}, \quad c = x_{23}, \quad d=x_{34}, \quad e= x_{13}, \quad f=x_{24}, \nonumber\\
    & g=x(O,A_1), \quad h=x(O,A_2), \quad i=x(O,A_3), \quad j=x(O,A_4), \nonumber\\
    &p=\skap(A_1,A_2,A_3), \quad q=\skap(A_1,A_3,A_4)\label{eq: h equations 1}\\ &r=\skap(A_1,A_2,A_4), \quad s=\skap(A_2,A_3,A_4). \label{eq: h equations 2}
\end{align}

\begin{figure}[htbp!] 
    \begin{center}
\begin{tikzpicture}[scale=0.9]

\coordinate [label=below:$O$] (O) at (0,0);
\circleat{0}{0};

\coordinate [label=left:$A_1$] (A_1) at (-2,4);
\circleat{-2}{4};

\coordinate [label=below left:$A_4$](A_4) at (-0.5,3);
\circleat{-0.5}{3};

\coordinate [label=right:$A_3$] (A_3) at (2,4);
\circleat{2}{4};

\coordinate[label=above:$A_2$] (A_2) at (1,6);
\circleat{1}{6};

\coordinate [label=right:$q$] (q) at (2,2);
\coordinate (q_1) at (1,3.75);
\draw (q_1) to [out=0,in=135] (q) ;

\coordinate [label=right:$p$] (p) at (2,6);
\coordinate (p_1) at (1.25,5);
\draw (p_1) to [out=45,in=180] (p) ;

\draw (O) to node[left] {} (A_1) ;
\draw[dotted] (O) to node[] {} (A_2) ;
\draw (O) to node[right] {} (A_3) ;
\draw (O) to node[right] {} (A_4) ;

\draw (A_1) to node[above left] {$b$} (A_2) ;
\draw (A_2) to node[above right] {$c$} (A_3) ;
\draw (A_3) to node[below, near end] {$d$} (A_4) ;
\draw (A_4) to node[below] {$a$} (A_1) ;
\draw (A_1) to node[above ] {$e$} (A_3) ;

\fill[black!20,opacity=0.3] (A_1) -- (A_2) -- (A_3) -- (A_1);
\fill[black!20,opacity=0.3] (A_1) -- (A_3) -- (A_4) -- (A_1);

\coordinate [label=left:$g$] (g) at (-.5,1);
\coordinate (h1) at (.3,1);
\coordinate[label=right:$h$] (h) at (0.3,2.2);

\coordinate[label=right:$i$] (i) at (.5,.75);

\coordinate[label=right:$j$] (j) at (-.3,1.6);

\end{tikzpicture}\hspace{1cm}
\begin{tikzpicture}[scale=0.9]

\coordinate [label=below:$O$] (O) at (0,0);
\circleat{0}{0};

\coordinate [label=left:$A_1$] (A_1) at (-2,4);
\circleat{-2}{4};

\coordinate [label=below 
 left:$A_4$](A_4) at (-0.5,3);
 \circleat{-0.5}{3};

\coordinate [label=right:$A_3$] (A_3) at (2,4);
\circleat{2}{4};

\coordinate[label=above:$A_2$] (A_2) at (1,6);
\circleat{1}{6};

\coordinate [label=left:$r$] (r) at (-2,5);
\coordinate (r_1) at (-1,4);
\draw (r_1) to [out=100,in=300] (r) ;

\coordinate [label=right:$s$] (s) at (2.5,5);
\coordinate (s_1) at (1,4.5);
\draw (s_1) to [out=45,in=180] (s) ;

\draw (O) to node[left] {} (A_1) ;
\draw (O)[dotted] to node[right] {} (A_2) ;
\draw (O) to node[right] {} (A_3) ;
\draw (O) to node[right] {} (A_4) ;

\draw (A_1) to node[above left] {$b$} (A_2) ;
\draw (A_2) to node[above right] {$c$} (A_3) ;
\draw (A_3) to node[below, near end] {$d$} (A_4) ;
\draw (A_4) to node[below] {$a$} (A_1) ;
\draw (A_2) to node[left ] {$f$} (A_4) ;

\fill[black!20,opacity=0.3] (A_1) -- (A_2) -- (A_4) -- (A_1);
\fill[black!20,opacity=0.3] (A_2) -- (A_3) -- (A_4) -- (A_2);

\coordinate [label=left:$g$] (g) at (-.5,1);
\coordinate (h1) at (.3,1);
\coordinate[label=right:$h$] (h) at (0.3,2.2);

\coordinate[label=right:$i$] (i) at (.5,.75);

\coordinate[label=right:$j$] (j) at (-.3,1.6);

\end{tikzpicture}
\end{center}
    \caption{Measurements associated to  triangulations of a spherical quadrilateral.}
    \label{fig:tetdecomp}
\end{figure}
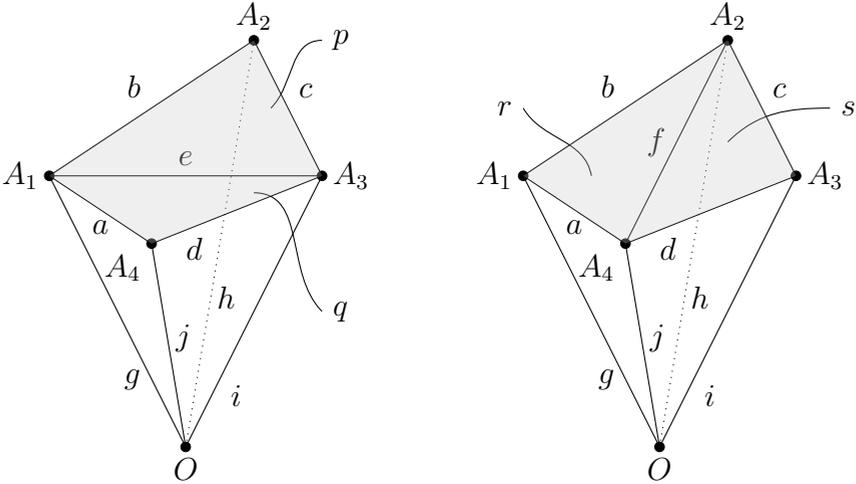
\vspace{-.3in}

\end{example}

We will derive our first type of relation amongst the measurement data for a spherical quadrilateral from the following elementary lemma.
    \begin{lemma}[cf.~{\cite[Chapter III, Section 17, the $J_3$ relation]{weyl_classical_1946}}]\label{lem: j3 relation}
        Let~$\overrightarrow{v_1},\dots,\overrightarrow{v_{d+1}}$ be vectors in a~$d$-dimensional space~$W$ equipped with a volume form
        \begin{align*}
            [ \, ]: W^d&\to \R\\
            (\overrightarrow{w_1},\dots,\overrightarrow{w_d})&\mapsto
            \dietmatrix{\begin{bmatrix}
                \overrightarrow{w_1} &  \cdots & \overrightarrow{w_d}
            \end{bmatrix}}.
        \end{align*}
        Then
        \begin{equation}\label{eq: lemdot}
            \sum_{i=1}^{d+1} (-1)^{i+1}
            \dietmatrix{
            \begin{bmatrix}
                \overrightarrow{v_1} & \cdots & \widehat{\overrightarrow{v_i}} &  \cdots & \overrightarrow{v_{d+1}}
            \end{bmatrix}
            }
            \overrightarrow{v_i}  = 0\, ,
        \end{equation} where~$\widehat{\overrightarrow{v_i}}$ indicates that we omit~$\overrightarrow{v_i}$ when computing the volume form.
       \end{lemma}

       \begin{proof}
        Let~$v_{1,i},v_{2,i},\dots,v_{d+1,i}$ be the~$i$th coordinates of vectors~$\overrightarrow{v_1},\dots,\overrightarrow{v_{d+1}}$ respectively. Consider the  matrix of column vectors~$\overrightarrow{v_1}, \overrightarrow{v_2}, \dots,\overrightarrow{v_{d+1}}.$
        Select a row and add it to the top, the resulting matrix is singular. Laplace expansion then gives
        \begin{equation}\label{eq: repeat row det}
        \dietmatrix{
            \begin{bmatrix}
            v_{1,j} & v_{2,j} & \dots & v_{d+1,j}\\
                \overrightarrow{v_1} 
                & \overrightarrow{v_2} & \dots & \overrightarrow{v_{d+1}}
            \end{bmatrix}}= \sum_{i=1}^{d+1} (-1)^{i+1}\dietmatrix{\begin{bmatrix}
                \overrightarrow{v_1} & \cdots & \widehat{\overrightarrow{v_i}} & \cdots & \overrightarrow{v_{d+1}}
            \end{bmatrix}}v_{i,j}  = 0\, .
        \end{equation}
        As this works for any row~$j$ we get the claim.
    \end{proof}

    \begin{remark}\label{rem: lemdot}
         Taking scalar products with an arbitrary vector in equation~\eqref{eq: lemdot} yields an algebraic relation.  E.g., if~$\overrightarrow{x},\overrightarrow{y},\overrightarrow{z},\overrightarrow{w}$ are vectors in a~2-dimensional Euclidean space, then we have
        \begin{equation*}
            \dietmatrix{\begin{bmatrix}
                \overrightarrow{x} & \overrightarrow{y}
            \end{bmatrix}}\langle\overrightarrow{z},\overrightarrow{w}\rangle - \dietmatrix{\begin{bmatrix}
                \overrightarrow{x} & \overrightarrow{z}
            \end{bmatrix}}\langle\overrightarrow{y},\overrightarrow{w}\rangle + \dietmatrix{\begin{bmatrix}
                \overrightarrow{y} & \overrightarrow{z}
            \end{bmatrix}}\langle
            \overrightarrow{x},\overrightarrow{w}\rangle = 0\, .
            \end{equation*}
    \end{remark}

\begin{proposition}\label{prop: p+q identities}
     In the setting of Example~\ref{ex: main tet example}, the following identities hold:
    \begin{align}
        p+q&=r+s+{\mytfrac{\kap}{2}}(ap+bq-er) \label{eq: pq1}\, ,\\
        p+q&=r+s+{\mytfrac{\kap}{2}}(fp-cr-bs)\label{eq: pq2}\, ,\\
        p+q&=r+s+{\mytfrac{\kap}{2}}(dp+cq-es)\label{eq: pq3}\, ,\\
        p+q&=r+s+{\mytfrac{\kap}{2}}(fq-dr-as)\label{eq: pq4}.
    \end{align}
\end{proposition}
    
\begin{proof}

Take, e.g.,~$\overrightarrow{g}$ as shorthand for the vector~$\overrightarrow{OA_1}$, and~$g_i$ as its~$i$th coordinate.  Let~$V_{ijk}$ be the volume of the tetrahedron~$OA_iA_jA_k$.

Making use of Remark~\ref{rem: lemdot} and the fact that the determinant of a matrix of three column vectors is six times the volume of the tetrahedron they define (equation~\eqref{eq:voltet1}), we get
    \begin{equation*}
        6V_{123}j_1-6V_{124}i_1+6V_{134}h_1+6V_{234}g_1=0\, .
    \end{equation*} 
    Multiply by~$2/R$ to get
    \begin{equation*}
    pj_1-ri_1+qh_1+sg_1=0 \, .   
    \end{equation*}
   Repeat with the 2nd and 3rd coordinates to get corresponding relations. The relation therefore holds for the vectors, giving
    \begin{equation}\label{eq:scalid}
    p\overrightarrow{j}-r\overrightarrow{i}+q\overrightarrow{h}+s\overrightarrow{g}=0\, .
    \end{equation}

    Take scalar products of equation~\eqref{eq:scalid} with~$\overrightarrow{g},\overrightarrow{h},\overrightarrow{i},\overrightarrow{j}$ respectively. Then use the identities of the form~$\langle \overrightarrow{i},\overrightarrow{j}\rangle=R^2-\frac{d}{2}$ to obtain
\begin{align*}
        p(2-\kap a)-r(2-\kap e)+q(2-\kap b)-2s &=0\, ,
\\
        p(2-\kap f)-r(2-\kap c)+2q-s (2-\kap b)&=0\, ,
    \\
        p(2-\kap d)-2r +q(2-\kap c)-s(2-\kap e)&=0\, ,
\\
        2p  - r(2-\kap d)+q(2-\kap f)-s(2-\kap a)&=0.
\end{align*} 
Rearranging yields the claim.
   \end{proof}

\begin{remark}

Taking~$R\to\infty$ (equivalently,~$\kap\to 0$), each of equations~\eqref{eq: pq1}-\eqref{eq: pq4} becomes
\[
p+q=r+s.
\]
For configurations in~$V_\R$ with~$K=0$, this statement reflects the fact that in a triangulation of a quadrilateral, the areas of the two triangles add up to the area of the quadrilateral.
\end{remark}

\begin{corollary}\label{prop: nokapid}
    In the setting of Example~\ref{ex: main tet example}, the following relations hold:
    \begin{align}
        a(p+s)&=q(f-b)+r(e-d)\, ,\nonumber\\
        b(s+q)&=p(f-a)+r(e-c)\, ,\nonumber\\
        c(q+r)&=p(f-d)+s(e-b)\, ,\nonumber\\
        d(r+p)&=q(f-c)+s(e-a)\, ,\nonumber\\
        e(r-s)&=p(a-d)+q(b-c) \, , \label{eq: ers}\\
        f(p-q)&=r(c-d)+s(b-a)\, .\nonumber
    \end{align}
\end{corollary}
\begin{proof}
    Each of these relations can be obtained by subtracting one equation in Proposition~\ref{prop: p+q identities} from another and rearranging.
\end{proof}
\begin{remark}
Any three equations in Corollary~\ref{prop: nokapid} imply the rest.  Note that equation~\eqref{eq: ers} recovers~\cite[(2.19)]{fomin_heronian_2021}.
\end{remark}

The final identities that we will establish for spherical quadrilaterals are analogues of the following classical result.

\begin{theorem}[Bretschneider's Formula]\label{thm: classical bret} 
    Let~$V$ be 2-dimensional Euclidean space. For~$A_1,A_2,A_3,A_4\in V$, define~$a,b,c,d,e,f,p,q,r,s$ as in Example~\ref{ex: main tet example}.
    Then 
    \begin{equation*}
        4ef=(p+q)^2+(a-b+c-d)^2 =(r+s)^2+(a-b+c-d)^2.
    \end{equation*}
\end{theorem}

\begin{remark}
    In 1842, Bretschneider~\cite{bretschneider_untersuchung_1842} and Strehlke~\cite{strehlke_zwei_1842} each published proofs that in the Euclidean planar 
 version ($V=V_\R$ and~$\kap=0$) of Example~\ref{ex: main tet example}  above, we have
    \[
    (p+q)^2=(s-a)(s-b)(s-c)(s-d)-abcd\cos^2\left({\mytfrac{\alpha+\gamma}{2}}\right)\, ,
    \]
where~$s=\frac{a+b+c+d}{2}$ and~$\alpha$ and~$\gamma$ are the angles between sides labeled by~$a$ and~$b$ and~$c$ and~$d$ respectively.
    The version appearing in Theorem~\ref{thm: classical bret} above and used in~\cite{fomin_heronian_2021} seems to be due to Coolidge~\cite{coolidge_historically_1939}. Coolidge's formulation generalizes a similar result known as Brahmagupta's formula  ($6$th century CE) for the area of a cyclic quadrilateral; see~\cite{ayoub_generalizations_2007} for an exposition of Brahmagupta and Bretschneider's formulas and related results. Of note, Heron's formula (Theorem~\ref{thm: heron}) can be regarded as a degenerate case of Brahmagupta's formula (see~\cite{jose_garcia_two_2020}), and hence a special case of Bretschneider's formula.
\end{remark}

Bretschneider's formula has been proven in many different ways.  One common method is to represent the sides and diagonals of a quadrilateral by vectors and to relate the area of the quadrilateral to the cross product of the vectors representing the diagonals.  We will provide a proof based on the following classical fact. The method of this proof will then be used to establish a spherical version.

\begin{lemma}[{\cite[Sections 67-68]{muir_treatise_1882}}]\label{lem: vector mult}
Let~$\overrightarrow{u_1},\dots,\overrightarrow{u_n},\overrightarrow{v_1},\dots,\overrightarrow{v_n}\in V$.  Then
    \begin{equation}\label{eq: vectormult}
        \det\dietmatrix{\begin{bmatrix}
\langle\overrightarrow{u_1},\overrightarrow{v_1}\rangle & \dots & \langle \overrightarrow{u_1},\overrightarrow{v_n}\rangle\\
          \vdots & \ddots & \vdots\\
          \langle \overrightarrow{u_n},\overrightarrow{v_1}\rangle & \dots & \langle \overrightarrow{u_n},\overrightarrow{v_n}\rangle\end{bmatrix}}
          =
          \dietmatrix{\begin{bmatrix}
            \overrightarrow{u_1} & \cdots & \overrightarrow{u_n}
    \end{bmatrix}}\dietmatrix{\begin{bmatrix}
            \overrightarrow{v_1} & \cdots & \overrightarrow{v_n}
        \end{bmatrix}}\, .
    \end{equation}
\end{lemma}
\begin{remark}\label{rem: gram det}
    The determinant on the left-hand side of equation~\eqref{eq: vectormult} is commonly called a Gram determinant (see, e.g.,~\cite[Section 8.11.5]{berger_geometry_1977}). The identity~\eqref{eq: vectormult} appears in~\cite[Chapter IV, Section 6]{turnbull_theory_1960}, is the first relation listed for~$SL_n$ in~\cite[p. 255]{popov_invariant_1994}, and is deemed a ``well-known relation'' by Weyl~\cite[p. 53]{weyl_classical_1946}.
\end{remark}

\begin{proof}[Proof of Theorem~\ref{thm: classical bret}]
 Use the shorthand~$\overrightarrow{a}$ (resp.,~$\overrightarrow{b}$,~$\overrightarrow{e}$) for the vector~$\overrightarrow{A_1A_4}$ (resp.,~$\overrightarrow{A_1A_2}$,~$\overrightarrow{A_1A_3}$). By Lemma~\ref{lem: vector mult},
    \begin{align*}
       \det \dietmatrix{\begin{bmatrix}
        2\langle \overrightarrow{a},\overrightarrow{b}\rangle & 2\langle \overrightarrow{a},\overrightarrow{e}\rangle\\
        2\langle \overrightarrow{e},\overrightarrow{b}\rangle & 2\langle \overrightarrow{e},\overrightarrow{e}\rangle
        \end{bmatrix}}
        &=2\dietmatrix{\begin{bmatrix}
            \overrightarrow{a} &\overrightarrow{e}
        \end{bmatrix}}\cdot
        2\dietmatrix{\begin{bmatrix}
            \overrightarrow{b} &\overrightarrow{e}
        \end{bmatrix}}\, .
    \end{align*}
    Using the identities analogous to 
    \begin{equation*}
        2\langle \overrightarrow{a},\overrightarrow{b}\rangle = a+b-f
    \end{equation*}
     and recalling that~$2\dietmatrix{\begin{bmatrix}
        \overrightarrow{a} & \overrightarrow{e}\end{bmatrix}}=q$, and $2\dietmatrix{\begin{bmatrix}
        \overrightarrow{b} &\overrightarrow{e}
    \end{bmatrix}}=p$, 
    we get
    \begin{equation*}
        \det\dietmatrix{\begin{bmatrix}
            a+b-f& a+e-d\\
            b+e-c& 2e
        \end{bmatrix}}=-qp\, .
    \end{equation*}
    Therefore
    \begin{align}
       2ef-be-ce+e^2-ae-de+ab-ac-bd+dc&=qp\label{eq: pq in euclidean case}\, ,
       \end{align}
       which implies, by virtue of Heron's formula:

\begin{align}\label{eq: virtue of}
        (p+q)^2&=p^2+q^2+2pq\nonumber\\
&=(-b^2-c^2-e^2+2bc+2be+2ce)\nonumber\\
&\quad+(-a^2-d^2-e^2+2ad+2ae+2de)\\
&\quad+2(2ef-be-ce+e^2-ae-de+ab-ac-bd+dc)\nonumber\\
&=4ef-(a-b+c-d)^2 \, .\nonumber\qedhere
\end{align}
\end{proof}

We next obtain the following spherical analogues of Bretschneider's formula.  In the case~$\kap= 0$, both equations recover Bretschneider's formula.

\begin{theorem} [Spherical Analogues of Bretschneider's Formula]\label{thm: sph Bret}
    In the setting of Example~\ref{ex: main tet example}, the following two equations hold:
\begin{align}
        4ef\left(1-{\textstyle\frac{\kap e}{4}}\right)&=(p+q)^2+(a-b+c-d)^2-\kap e (a-b)(c-d)\, , \label{eq: sphericalbret}\\
        4ef\left(1-{\textstyle\frac{\kap f}{4}}\right)&=(r+s)^2+(a-b+c-d)^2-\kap f(a-d)(c-b)\, .\label{eq: sphericalbret2}
    \end{align}

\end{theorem}

\begin{proof}
We follow the general plan of the above proof of Theorem~\ref{thm: classical bret}.    By Lemma~\ref{lem: vector mult},
    \begin{equation*}
        \det\begin{bmatrix}
            \langle \overrightarrow{g},\overrightarrow{g}\rangle & \langle \overrightarrow{g},\overrightarrow{i}\rangle & \langle \overrightarrow{g},\overrightarrow{j}\rangle\\
            \langle \overrightarrow{h},\overrightarrow{g}\rangle & \langle \overrightarrow{h},\overrightarrow{i}\rangle & \langle \overrightarrow{h},\overrightarrow{j}\rangle\\
            \langle \overrightarrow{i},\overrightarrow{g}\rangle & \langle \overrightarrow{i},\overrightarrow{i}\rangle & \langle \overrightarrow{i},\overrightarrow{j}\rangle
        \end{bmatrix}=\dietmatrix{\begin{bmatrix}
            \overrightarrow{g} & \overrightarrow{h} & \overrightarrow{i}
        \end{bmatrix}}\cdot\dietmatrix{\begin{bmatrix}
            \overrightarrow{g} & \overrightarrow{i} & \overrightarrow{j}
        \end{bmatrix}}\, .
    \end{equation*}
    Manipulating scalar products and recalling definitions, this becomes
    \begin{equation*}
        \det\begin{bmatrix}
            g & \frac{g+i-e}{2} & \frac{g+j-a}{2} \\
            \frac{h+g-b}{2} & \frac{h+i-c}{2} & \frac{h+j-f}{2} \\
            \frac{i+g-e}{2} &
            i & \frac{i+j-d}{2}
        \end{bmatrix}=6V_{123}\cdot 6V_{134} = \mytfrac{R}{2}\skaponetwothree\cdot \mytfrac{R}{2}\skaponethreefour\, .
    \end{equation*}
    Noting that~$g,i,h,j=R^2$, we deduce
    \begin{equation*}
       4\kap \cdot \det\begin{bmatrix}
            \frac{1}{\kap} & \frac{1}{\kap}-\frac{e}{2} & \frac{1}{\kap}-\frac{a}{2} \\
            \frac{1}{\kap}-\frac{b}{2} & \frac{1}{\kap}-\frac{c}{2} & \frac{1}{\kap}-\frac{f}{2} \\
            \frac{1}{\kap}-\frac{e}{2} &
            \frac{1}{\kap} & \frac{1}{\kap}-\frac{d}{2}
        \end{bmatrix} = pq\, .
    \end{equation*}
    Computing the determinant and simplifying, we get
    \begin{equation}\label{eq: eq for pq bret}
        ab - ac - bd + cd - ae - be - ce - de + e^2 + 2ef + {\textstyle\frac{\kap}{2}}(ace+bde-e^2f)=pq\, .
    \end{equation}
Using this in a calculation analogous to~\eqref{eq: virtue of} (but with the formulas~\eqref{eq: h equations 1}-\eqref{eq: h equations 2} for~$p^2$ and~$q^2$) produces equation~\eqref{eq: sphericalbret}.
    Equation~\eqref{eq: sphericalbret2} can be derived in a similar way.
\end{proof}

The following collection of identities will prove useful later.
\begin{proposition} \label{prop:prop210analoge}

Let~$A_1,A_2,A_3,A_4\in\bfS$ and consider the spherical quadrilateral~$A_1A_2A_3A_4$ with measurements~$a,b,c,d,e,f,p,q,r,s$ as in Example~\ref{ex: main tet example}.  Then
\begin{align}
    p^2 &= \hkap(b,c,e) = -b^2-c^2-e^2+2bc+2be+2ce- \kap bce \label{eq:def1sphp}\, ,\\
    q^2 &= \hkap(a,d,e) = -a^2-d^2-e^2+2ad+2ae+2de- \kap ade \label{eq:def1sphq}\, ,\\
    r^2 &= \hkap(a,b,f) = -a^2-b^2-f^2+2ab+2af+2bf- \kap abf \label{eq:def1sphr}\, ,\\
    s^2 &= \hkap(c,d,f) = -c^2-d^2-f^2+2cd+2cf+2df-\kap cdf \label{eq:def1sphs}\, ,\\
    p+q&=r+s+{\textstyle\frac{\kap}{2}}(ap+bq-er)\label{eq:def1sph p+qanalogue}\, ,\\
    4ef\left(1-{\textstyle\frac{\kap e}{4}}\right)&=(p+q)^2+(a-b+c-d)^2-\kap e(a-b)(c-d)\label{eq: def1sph bretanalogue}\, ,\\
    e(r-s)&=p(a-d)+q(b-c)\label{eq: def1sph erminuss}\, .
\end{align}
    
\end{proposition}
\begin{proof}
    Equations~\eqref{eq:def1sphp}, \eqref{eq:def1sphq}, \eqref{eq:def1sphr}, and~\eqref{eq:def1sphs} are all instances of equation~\eqref{eq:kappaheron}.  Equation~\eqref{eq:def1sph p+qanalogue} is equation~\eqref{eq: pq1}.  Equation~\eqref{eq: def1sph bretanalogue} is equation~\eqref{eq: sphericalbret}.  Equation~\eqref{eq: def1sph erminuss} is equation~\eqref{eq: ers}.
\end{proof}

\begin{remark}
    Specializing~$\kap=0$ recovers~\cite[(2.13)-(2.19)]{fomin_heronian_2021}.
\end{remark}

\newpage

\section{Spherical Heronian friezes}\label{sec: sph her friezes}

We established that the data of three squared distances~$x_{ij}$ together with an associated~$\skap$ measurement is sufficient to determine a spherical triangle (cf.~Lemma~\ref{lem: 2.3 analogue}).  In this section, we consider a more general question: Which measurement data are sufficient to determine a spherical polygon?  Some definitions and results in this section are spherical counterparts of those in \cite[Sections 2, 3]{fomin_heronian_2021}.

\begin{definition} \label{def: sph poly}
    For~$n\ge 3$, a labeled \emph{spherical polygon} (or \emph{$n$-gon}) in~$\bfS$ is an ordered~$n$-tuple of points 
    \[
    P=(A_1,\dots,A_n)\in\bfS^n.
    \] 
  For such a polygon, we collect the measurements~$x_{ij}=x_{ij}(P)$ and~$\skap_{ijk}=\skap_{ijk}(P)=\skap (A_i,A_j,A_k)$ for all distinct~$i,j,k\in\{1,n\}$ and form the tuple
    \begin{equation*}
        \xS(P)=(x_{ij})\sqcup (\skap_{ijk}).
    \end{equation*}
    Note that by Proposition~\ref{prop: sph heron} these measurements satisfy the equations
    \begin{equation}\label{eq: heroneq}
        (\skap_{ijk})^2=\hkap(x_{ij},x_{jk},x_{ik}).
    \end{equation}

\end{definition}

\begin{definition}\label{def: triangulated}
Let~$G$ be an undirected graph on~$n$ vertices~$v_1,\dots,v_n$ with~$2n-3$ edges, of which 
 a distinguished size~$n$ set, without loss of generality 
 \[
 (v_1,v_2),(v_2,v_3),\dots,(v_{n-1},v_n),(v_n,v_1),
 \]
 form a cycle. The~$n-3$ edges of~$G$ that are not in the distinguished cycle are called \emph{diagonals}. We call~$G$ a \emph{triangulated cycle}, or refer to~$G$ as a \emph{triangulation}, if there exists an embedding of~$G$ into the real plane such that the distinguished cycle is represented by a convex~$n$-gon (with~$v_1,\dots,v_n$ appearing in clockwise order) and the diagonals are represented by pairwise noncrossing diagonals.
\end{definition}

\begin{definition}\label{def: tri polygon} 
    A \emph{triangulated spherical polygon}~$T=(P,G)$ is a spherical polygon~$P$ together with a choice of triangulation~$G$ of the underlying cycle.  For such an object, we consider the subcollection of measurements: 
    \[
    \xS_G(P)\subset \xS(P),
    \]
    which contains all the measurements~$x_{ij}$ and~$\skap_{ijk}$ that correspond to the sides, diagonals, and triangles that appear in the triangulation~$G$.
\end{definition}

\begin{proposition}\label{prop: cor 2.9 analogue} 
    Let~$G$ be a triangulated cycle with vertices~$v_1,\dots,v_n$ as in Definition~\ref{def: triangulated}. Let~$\xS_G=(x_{ij})\sqcup (\skapijk)$ be a tuple of complex numbers such that
    \begin{itemize}
        \item the numbers~$x_{ij}\in\C-\{0,\frac{4}{K}\}$, for~$i<j$, are labeled by the edges~$\{v_i,v_j\}$ present in~$G$, and
        \item the numbers~$\skap_{ijk}$,~$i<j<k$, are labeled by the triangles~$(v_i, v_j, v_k)$ in~$G$.  
    \end{itemize}
    Assume that equation~\eqref{eq: heroneq} 
    holds for each~$i<j<k$:
    \[
(\skapijk)^2=\hkap(x_{ij},x_{ik},x_{jk}).
    \]
  Then there exists a spherical~$n$-gon~$P$ on~$\bfS$ such that~$\xS_G(P)=\xS_G.$ Moreover, the polygon~$P$ is unique up to the action of~$\Aut(\bfS)$. In particular, all measurements in~$\xS(P)$ are uniquely determined by~$\xS_G(P)$.
    
\end{proposition}
\begin{proof}
We proceed by induction on the number of vertices in~$G$. 
 If~$G$ has one triangle, we can apply Lemma~\ref{lem: 2.3 analogue}.  Assume the claim holds for~$k<n$ and consider~$G$ with~$n$ vertices.  Then~$G$ has~$n-2$ triangles, so one triangle must include two edges from the distinguished cycle.  Apply Lemmas~\ref{lem: 2.4analogue}~and~\ref{lem: 2.3 analogue} to this triangle.  The rest of~$G$ has less than~$n$ vertices in its distinguished cycle, and only shares one edge with the triangle we have constructed.  By induction we get the claim.
\end{proof}

We will next write formulas expressing the measurements for a spherical polygon~$P$ in terms of the initial data associated to a triangulation of the vertices of~$P$.  We  begin with the case of a quadrilateral, then proceed to the more general case.  We will use the following object to organize the measurement data of a quadrilateral.

\begin{definition}  \label{def:original definition}
    Let~$K\in\C$.  A \emph{spherical Heronian diamond} is an ordered {10-tuple} of numbers~$(a,b,c,d,e,f,p,q,r,s)$ satisfying equations~\eqref{eq:def1sphp}-\eqref{eq: def1sph erminuss}.  We display the entries in a diamond shape as shown on the right in Figure~\ref{fig: diamond}. Explicitly,
    \begin{itemize}
        \item $a,c,e,f$ are placed at the corners of the diamond;
        \item $p,q,r,s$ are placed at the midpoints of the sides of the diamond;
        \item the dashed line going Southeast (resp. Southwest) is labeled~$d$ (resp.~$b$).
    \end{itemize}
    The entries are arranged so that the four ``measurements'' associated with each side of the diamond (e.g.,~$a,d,e,q$) satisfy the corresponding spherical Heron equation (e.g., equation~\eqref{eq:def1sphq}).

\end{definition}

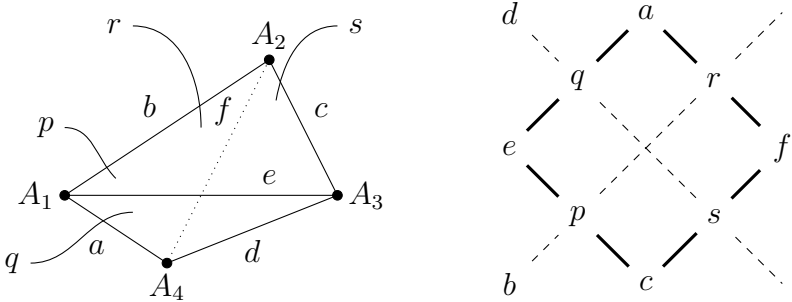
\begin{figure}[htbp!]
    \centering
    \begin{tikzpicture}[scale=0.9]

\coordinate [label=left:$A_1$] (A_1) at (-2,2);
\circleat{-2}{2};

\coordinate [label=below:$A_4$](A_4) at (-0.5,1);
\circleat{-0.5}{1};

\coordinate [label=right:$A_3$] (A_3) at (2,2);
\circleat{2}{2};

\coordinate[label=above:$A_2$] (A_2) at (1,4);
\circleat{1}{4};

\coordinate [label=left:$q$] (q) at (-2.5,1);
\coordinate (q_1) at (-1,1.75);
\draw (q_1) to [out=180,in=0] (q) ;

\coordinate [label=left:$p$] (p) at (-2,3);
\coordinate (p_1) at (-1.25,2.25);
\draw (p_1) to [out=135,in=0] (p) ;

\coordinate (r) at (0,3);
\coordinate[label=left:$r$] (r1) at (-1,4.5);
\draw (r) to [out=90,in=0] (r1);

\coordinate (s) at (1.1,3.4);
\coordinate[label=right:$s$] (s1) at (2,4.5);
\draw (s) to [out=80,in=180] (s1);

\draw (A_1) to node[above left] {$b$} (A_2) ;
\draw (A_2) to node[above right] {$c$} (A_3) ;
\draw (A_3) to node[below] {$d$} (A_4) ;
\draw (A_4) to node[below left] {$a$} (A_1) ;
\draw (A_1) to node[above, near end ] {$e$} (A_3) ;
\draw[dotted] (A_2) to node[left, near start] {$f$} (A_4);

\end{tikzpicture}\hspace{1 cm}
\begin{tikzpicture}[scale=0.9, every node/.style={circle,fill=white}]

        \begin{scope}[on background layer]
        
        \draw[dashed] (-2,4) node {$d$}
        --(2,0);
        \draw[dashed] (-2,0) node {$b$}
        --(2,4);
    \end{scope}
        \draw[very thick]
        (0,0) node {$c$}
        -- (-1,1) node {$p$}
        -- (-2,2) node {$e$}
        -- (-1,3) node {$q$}
        -- (0,4) node {$a$}
        -- (1,3) node {$r$}
        -- (2,2) node {$f$}
        -- (1,1) node {$s$}
        -- (0,0);
\end{tikzpicture}    
\caption{Left: A spherical quadrilateral~$A_1A_2A_3A_4$ as discussed in Example~\ref{ex: main tet example}. Right: A spherical Heronian diamond.}
    \label{fig: diamond}
\end{figure}

\begin{remark}\label{rem: sph heronian diamond visual}
    A spherical quadrilateral gives rise to a spherical Heronian diamond.
\end{remark}

A priori, a spherical Heronian diamond is a purely algebraic object, an ordered~$10$-tuple of numbers that satisfy certain relations.  We will next derive formulas to express three of these numbers ($f,r,s$) in terms of the others.  When applied in the geometric context, these formulas will allow us to express the ten measurements~$\xS(P)$ of a spherical quadrilateral in terms of the seven measurements~$\xS_G(P)$ associated to its triangulation.

\begin{proposition}\label{prop: analogue of 2.14}
    Let~$(a,b,c,d,e,p,q)$ be a 7-tuple of complex numbers satisfying equations~\eqref{eq:def1sphp}-\eqref{eq:def1sphq}, with~$e\not\in \left\{0,\frac{4}{\kap}\right\}$.  Then there exist unique~$f,r,s\in\C$ such that~$(a,b,c,d,e,f,p,q,r,s)$ is a spherical Heronian diamond.  Specifically
    \begin{align}
        f&=\frac{(p+q)^2+(a-b+c-d)^2-\kap e(a-b)(c-d)}{4e\left(1-\frac{\kap e}{4}\right)} \label{eq: eq for f}\, ,\\
        r&=\frac{p\left(e+a-d-\frac{\kap ae}{2}\right)+q\left(e-c+b-\frac{\kap be}{2}\right)}{2e\left(1-\frac{\kap e}{4}\right)}\label{eq: eq for r}\, ,\\
        s&=\frac{p\left(e-a+d-\frac{\kap de}{2}\right)+q\left(e+c-b-\frac{\kap ce}{2}\right)}{2e\left(1-\frac{\kap e}{4}\right)}\, . \label{eq: eq for s}
    \end{align}
    (The condition~$e\not\in\{0,\frac{4}{\kap}\}$ ensures that the denominators of equations~\eqref{eq: eq for f}-\eqref{eq: eq for s} do not vanish.)
\end{proposition}
\begin{proof} Define~$f,r,s$ as in equations~\eqref{eq: eq for f}, \eqref{eq: eq for r}, and~\eqref{eq: eq for s}.  To~verify that the resulting tuple~$(a,b,c,d,e,f,p,q,r,s)$ is a spherical Heronian diamond we need to verify equations~\eqref{eq:def1sphr}-\eqref{eq: def1sph erminuss}.  (Equations~\eqref{eq:def1sphp}-\eqref{eq:def1sphq} hold by assumption.)

Equations~\eqref{eq:def1sphr} and~\eqref{eq:def1sphs} can be checked with standard software, as follows.  For each of the equations~\eqref{eq:def1sphp}, \eqref{eq:def1sphq}, \eqref{eq: eq for f}, \eqref{eq: eq for r}, and~\eqref{eq: eq for s}, identify a polynomial that vanishes when the equation holds.  For example,
\[
p^2 = -b^2-c^2-e^2+2bc+2be+2ce-\kap bce \quad \text{(equation }\eqref{eq:def1sphp})
\]
corresponds to the polynomial 
\begin{equation}\label{eq: make it a poly}
    p^2-( -b^2-c^2-e^2+2bc+2be+2ce-\kap bce)\, .
\end{equation}
Collect these five polynomials and use them to generate an ideal~$I$. Then, for example, equation~\eqref{eq:def1sphr} is obtained by checking that
\begin{equation}\label{eq: poly for r}
    e^2(\kap e-4)^2(r^2-(-a^2-b^2-f^2+2ab+2af+2bf-\kap abf))\in I.
\end{equation}
We can check for ideal membership by reducing the polynomial~\eqref{eq: poly for r} with respect to a Gr\"{o}bner basis~$B$ (for example with lexicographic monomial order). Note that the polynomial~\eqref{eq: make it a poly} does not lie in I: the factors~$e^2$ and~$(\kap e - 4)^2$ in~\eqref{eq: poly for r} are required for the argument to work.

Equation~\eqref{eq: eq for f} can be slightly rearranged to get equation~\eqref{eq: def1sph bretanalogue}. Equations~\eqref{eq:def1sph p+qanalogue}-\eqref{eq: def1sph erminuss} can be verified manually, by plugging in the definitions of~$r$ and~$s$.

To~show the uniqueness of~$f,r,s$, it suffices to deduce~\eqref{eq: eq for f}-\eqref{eq: eq for s} from the equations of  a spherical Heronian diamond.  Again, equation~\eqref{eq: def1sph bretanalogue} can be slightly rearranged to obtain equation~\eqref{eq: eq for f}.
    To~get~$r$, we use equation~\eqref{eq: ers}. Replacing~$s$ with the expression obtained from~\eqref{eq:def1sph p+qanalogue}, we get
    \begin{equation*}
        er=p(a-d)+q(b-c)+e\left(p+q-\kap\left(\mytfrac{ap+bq-er}{2}\right)\right)\, .
    \end{equation*} Simplifying, we obtain~\eqref{eq: eq for r}.  Formula~\eqref{eq: eq for s} is established analogously.\qedhere
\end{proof}

\begin{remark}\label{rem: def equiv}
By Proposition~\ref{prop: analogue of 2.14}, Definitions~\ref{def:original definition} and~\ref{def: intro def} are equivalent to each other.
\end{remark}

Let~$P$ be a spherical quadrilateral with squared side distances~$a,b,c,d$.  Proposition~\ref{prop: analogue of 2.14} gives formulas to express the squared distance~$f$ and the~$\skap$ measurements~$r,s$ in terms of~$a,b,c,d$, the squared distance~$e$, and the~$\skap$ measurements~$p,q$.

The following example shows an application of Theorem~\ref{prop: analogue of 2.14}.
\begin{figure}[htbp!]
    \centering
    \begin{tikzpicture}[scale=0.8]

\coordinate [label=left: San Francisco] (A_1) at (-2,1);
\circleat{-2}{1};

\coordinate [label=below: Miami](A_4) at (-0.5,0);
\circleat{-0.5}{0};

\coordinate [label=right: Boston] (A_3) at (4,1);
\circleat{4}{1};

\coordinate[label=above: Minneapolis] (A_2) at (3,3);
\circleat{3}{3};

\draw (A_1) to node[above] {4352} (A_3) ;

\draw (A_1) to node[above left] {2557} (A_2) ;
\draw (A_2) to node[above right] {1809} (A_3) ;
\draw (A_3) to node[below right] {2025} (A_4) ;
\draw (A_4) to node[below left] {4160} (A_1) ;

\end{tikzpicture}
    \caption{Spherical distances (in~$\text{km}$) between four American cities~\cite{swartz_great_nodate}.}
    \label{fig: fourcities}
\end{figure}
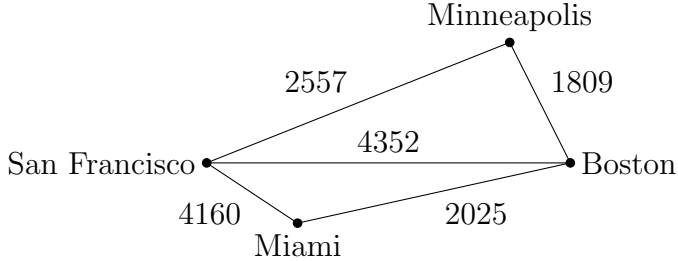
\begin{example}
         Given the five spherical distances shown in Figure~\ref{fig: fourcities}, let us compute the unknown spherical distance from Minneapolis to Miami.   For this calculation we will assume that the surface of the Earth is a sphere whose great circles have length~$40,000 \, \text{km}$. Then 
   ~$R=\frac{4\cdot 10^4}{2\cdot \pi} \, \text{km}$ and therefore~$\kap=\frac{\pi^2}{4\cdot 10^8} \, \text{km}^{-2}$ .

We begin by converting  spherical distances into squared Euclidean distances. If~$|AB|$ is the spherical distance between points~$A,B\in\bfS$, then the squared Euclidean distance between~$A$ and~$B$ is given by
\[
x(A,B)=2R^2\left(1-\cos\left(\mytfrac{|AB|}{R}\right)\right)
\]
(cf.~\eqref{eq: euc dist}).
Applying this formula to the values in Figure~\ref{fig: fourcities}, we get the values of~$a,b,c,d,e$ shown in Figure~\ref{fig: fourcitiesEuclidean}.   Let~$f$ denote the squared Euclidean distance between Minneapolis and Miami.
Set
\begin{align*}
p&=\skap(\text{\small San Francisco, Minneapolis, Boston})\, , \\  q&=\skap(\text{\small San Francisco, Boston, Miami}).
\end{align*}
Then
\[
p^2=\hkap(b,c,e) \ \text{ and } \ q^2=\hkap(a,d,e)
\]
with~$\hkap$ given by equation~\eqref{eq: h kappa def}.
Taking positive square roots, we obtain the quantities~$p,q$ shown in Figure~\ref{fig: fourcitiesEuclidean}.

We can now use equation~\eqref{eq: eq for f} to compute the spherical distance~$F$ from Minneapolis to Miami:
\begin{gather*}
    f=\frac{(p+q)^2+(a-b+c-d)^2-\kap e(a-b)(c-d)}{4e\left(1-\frac{\kap e}{4}\right)}=
57600370 \ \text{km}^2,\\
F=R\arccos\left(1-\mytfrac{57600370 \ \text{km}^2}{2R^2}\right) = 2414 \text{ km}.
\end{gather*}
We note that in spite of making several simplifications and approximations, our calculation achieved a remarkably accurate result: the actual spherical distance from Miami to Minneapolis is 2415 km, see~\cite{swartz_great_nodate}.
    \end{example}
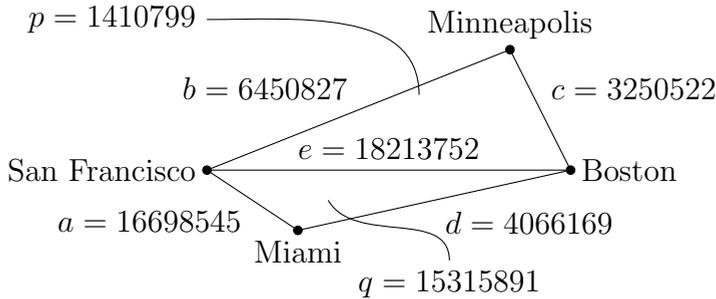
\begin{figure}[htbp!]
    \centering
    \begin{tikzpicture}[scale=0.8]

\coordinate [label=left: San Francisco] (A_1) at (-2,1);
\circleat{-2}{1};

\coordinate [label=below: Miami](A_4) at (-0.5,0);
\circleat{-0.5}{0};

\coordinate [label=right: Boston] (A_3) at (4,1);
\circleat{4}{1};

\coordinate[label=above: Minneapolis] (A_2) at (3,3);
\circleat{3}{3};

\draw (A_1) to node[above] {$e=18213752$} (A_3) ;

\draw (A_1) to node[above left] {$b=6450827$} (A_2) ;
\draw (A_2) to node[above right] {$c=3250522$} (A_3) ;
\draw (A_3) to node[below right] {$d=4066169$} (A_4) ;
\draw (A_4) to node[below left] {$a=16698545$} (A_1) ;

\coordinate  (p) at (-2,3.5);
\node at (p) [left] {$p=1410799$};
\coordinate (p_1) at (1.5,2.25);
\draw (p_1) to [out=90,in=0] (p) ;

\coordinate (q) at (2,-0.5);
\node at (q) [below] {$q=15315891$};
\coordinate (q_1) at (0,0.5);
\draw (q_1) to [out=300,in=90] (q) ;

\end{tikzpicture}
    \caption{Squared Euclidean distances between the four cities (in~$\text{km}^2$).}
    \label{fig: fourcitiesEuclidean}
\end{figure}
    
We next discuss the symmetries of the notion of a spherical Heronian diamond.
\begin{proposition} \label{prop: analogue of 2.15} 
    Let~$(a,b,c,d,e,f,p,q,r,s)$ be a spherical Heronian diamond.  Then
    \begin{enumerate}
        \item~$(c,d,a,b,e,f,q,p,s,r)$ is a spherical Heronian diamond;
        \item if~$e\not \in \left\{0,\frac{4}{\kap}\right\}$ then~$(a,d,c,b,f,e,s,r,q,p)$ is a spherical Heronian diamond.
    \end{enumerate}
\end{proposition}
Proposition~\ref{prop: analogue of 2.15} implies that flipping a spherical Heronian diamond horizontally or vertically produces another spherical Heronian diamond.
\begin{proof}
    To~prove the first claim, we interchange~$a\leftrightarrow c$,~$b\leftrightarrow d$,~$p\leftrightarrow q$, and~$r\leftrightarrow s$. These switches correspond to a vertical reflection of the spherical Heronian diamond across the center horizontal axis.  We need to show equations~\eqref{eq:def1sphp}-\eqref{eq: def1sph erminuss} hold for~$(c,d,a,b,e,f,q,p,s,r)$. The swaps cause equations~\eqref{eq:def1sphp} and~\eqref{eq:def1sphq} and  equations~\eqref{eq:def1sphr} and~\eqref{eq:def1sphs}  to be exchanged, and equations~\eqref{eq: def1sph bretanalogue}-\eqref{eq: def1sph erminuss} are unaffected.  Equation~\eqref{eq:def1sph p+qanalogue} with these values is
    \begin{equation*}
p+q=r+s+\mytfrac{\kap}{2}\left(cq+dp-es\right)\, ,
    \end{equation*}
    but since equation~\eqref{eq: def1sph erminuss} holds we can substitute~$e(s-r)-q(c-b)+ap$ for~$dp$ 
and simplify to achieve    \begin{equation*}
        p+q=r+s+\mytfrac{\kap}{2}\left(ap+bq-er\right)\, ,
    \end{equation*} which holds by assumption.
So~$(c,d,a,b,e,f,q,p,s,r)$ again satisfies the conditions of Proposition~\ref{prop:prop210analoge} and therefore by definition is a spherical Heronian diamond.

    Now interchange~$b\leftrightarrow d$,~$e\leftrightarrow f$,~$p\leftrightarrow s$ and~$q\leftrightarrow r$. This corresponds to a horizontal reflection of the spherical Heronian diamond~$(a,b,c,d,e,f,p,q,r,s)$. Again equations~\eqref{eq:def1sphp}-\eqref{eq:def1sphs} are interchanged.

    Equations~\eqref{eq:def1sph p+qanalogue}, \eqref{eq: def1sph bretanalogue} and~\eqref{eq: def1sph erminuss} become 
    \begin{align}
s+r&=q+p+\mytfrac{\kap}{2}\left(as+dr-fq\right)\label{eq: fpq analogue}\, ,\\
        4ef\left(1-\mytfrac{\kap}{4}f\right)&=(s+r)^2+(a-d+c-b)^2-\kap f(a-d)(c-b)\label{eq: fbret analogue}\, ,\\
        f(q-p)&=s(a-b)+r(d-c) \label{eq: ferminuss analogue}\, ,
    \end{align}
    and we need to show that as long as~$e\not\in\{ 0,\frac{4}{\kap}\}$ and~$(a,b,c,d,e,f,p,q,r,s)$ is a spherical Heronian diamond, these equations hold. We can write each of the equations~\eqref{eq:def1sphp}-\eqref{eq: def1sph erminuss} in the form~$F(a,b,c,d,e,f,p,q,r,s)=0$ where~$F$ is a polynomial.
    We then collect these seven polynomials and form the ideal~$I$ that they generate. We similarly write each of the equations~\eqref{eq: fpq analogue}-\eqref{eq: ferminuss analogue} in the form~$G=0$.   It then suffices to show that the polynomials
    \begin{gather*}
        \left(s+r-q-p-\kap\left(\mytfrac{as+dr-fq}{2}\right)\right)\cdot e(\kap e - 4)\, ,\\
        \left(4ef\left(1-\kap\mytfrac{f}{4}\right)-(s+r)^2-(a-d+c-b)^2+\kap f(a-d)(c-b)\right)\cdot e(\kap e-4)\, ,\\
        \left(f(q-p)-s(a-b)-r(d-c)\right)\cdot e(\kap e-4)
    \end{gather*}
    lie in the ideal~$I$. (Here we rely on the condition~$e\not\in\{0,\frac{4}{\kap}\}$.) This can be confirmed with standard software.
\end{proof}

\begin{remark}\label{rem: prop 215 reverse}
    By the symmetry of swapping, if a tuple $(c,d,a,b,e,f,q,p,s,r)$ is a spherical Heronian diamond,  then by {bullet 1} of Proposition~\ref{prop: analogue of 2.15}, the tuple~$(a,b,c,d,e,f,p,q,r,s)$ is also a spherical Heronian diamond.
\end{remark}

\begin{corollary}\label{cor: 2.16 analogue}
    In a spherical Heronian diamond~$(a,b,c,d,e,f,p,q,r,s)$, once the components~$a,b,c,d$ are fixed, the values~$e,p,q$ determine~$f,r,s$ uniquely (as long as~$e\not\in\{ 0,\frac{4}{\kap}\}$), and vice versa (as long as~$f\not\in\{ 0,\frac{4}{\kap}\}$).
\end{corollary}
\begin{proof}
    Fix~$a,b,c,d$. The fact that~$e,p,q$ with~$e\not\in\{ 0,\frac{4}{\kap}\}$ determines~$f,r,s$ is the content of Proposition~\ref{prop: analogue of 2.14} (as applied to the 7-tuple~$(a,b,c,d,e,p,q)$).  
    
    In reverse, by the second bullet of Proposition~\ref{prop: analogue of 2.15}, if~$(a,b,c,d,e,f,p,q,r,s)$ is a spherical Heronian diamond then~$(a,d,c,b,f,e,s,r,q,p)$ is a spherical Heronian diamond.  Assuming that~$f\not\in\{ 0,\frac{4}{\kap}\}$ and applying Proposition~\ref{prop: analogue of 2.14} we see that~$e,p,q$ are determined uniquely by the 7-tuple~$(a,d,c,b,f,s,r)$.  Since~$a,b,c,d$ are fixed, as long as~$f\not\in\{ 0,\frac{4}{\kap}\}$, the measurements~$f,r,s$ determine~$e,p,q$ uniquely.
\end{proof}

The following lemmas and corollary deal with cases corresponding in the geometric setting to various degenerate quadrilaterals.

\begin{lemma}\label{lem: degenerate 1} 
     Let~$b,c,d,e,f,p,s\in\C$ and~$a=q=r=0$. 
 The corresponding tuple~$(0,b,c,d,e,f,p,0,0,s)$ 
    is a spherical Heronian diamond if and only if
\begin{equation}\label{eq: bdry}
p^2=\hkap(b,c,e)\, , \quad
        d=e \, , \quad
        f=b \, ,\quad \text{and} \quad 
        s=p \,.
    \end{equation}
\end{lemma}
\begin{proof}
    Assume the tuple~$(0,b,c,d,e,f,p,0,0,s)$ is a spherical Heronian diamond. Then equations~\eqref{eq:def1sphp}-\eqref{eq:def1sphq} and~\eqref{eq:def1sph p+qanalogue} give us
    \begin{align*}
        p^2&=\hkap(b,c,e)\, ,\\
        0&=\hkap(0,d,e)=-(d-e)^2\, ,\\
        0&=\hkap(0,f,b)=-(f-b)^2\, ,\\
        p&=s\, ,
    \end{align*}
   implying~\eqref{eq: bdry}. Conversely, assume that~\eqref{eq: bdry} hold.  Then equations~\eqref{eq:def1sphp}-\eqref{eq:def1sphr} are immediate. To~ show~$s^2=\hkap(c,f,d)$ (equation~\eqref{eq:def1sphs}), we note that 
    \[
    s^2=p^2=\hkap(b,c,e)=\hkap(f,c,d)\, .
    \]
    Since~$f=b$, we have 
    \begin{equation*}
        p^2=\hkap(b,c,e)=-b^2-c^2-e^2+2bc+2be+2ce+\kap ebc+4e(f-b)+\kap e^2(b-f)
    \end{equation*}
    which rearranges to equation~\eqref{eq: def1sph bretanalogue}.
    Finally, utilizing~$p=s$ and~$e=d$, we conclude~$-es=-pd$ (equations~\eqref{eq:def1sph p+qanalogue}, \eqref{eq: def1sph erminuss}), proving the claim.
\end{proof}

\begin{lemma}\label{lem: degenerate 2}   Let~$a,b,d,e,f,q,r\in\C$ and~$c=p=s=0$. The corresponding tuple~$(a,b,0,d,e,f,0,q,r,0)$ is a spherical Heronian diamond if and only if
\begin{align}
    q^2&=\hkap(a,d,e)\, ,\\
    b&=e\label{eq: bdry 4}\, ,\\
    f&=d \label{eq: bdry 5}\, ,\\
    r&=q \label{eq: bdry 6}\, .
\end{align}
    
\end{lemma}
\begin{proof}
    Argue similarly to Lemma~\ref{lem: degenerate 1} above, or note~$(a,b,0,d,e,f,0,q,r,0)$ is a spherical Heronian diamond if and only if~$(0,d,a,b,e,f,q,0,0,r)$ is a spherical Heronian diamond (see Remark~\ref{rem: prop 215 reverse}); then apply Lemma~\ref{lem: degenerate 1}. 
\end{proof}

The lemmas above and Corollary~\ref{cor: 2.16 analogue} imply the following corollary.

\begin{corollary} \label{cor: 2.19 analogue}
    Let~$a,b,c,d,e,f,p,q,r,s\in\C$.  If the tuple~$(0,b,c,d,e,f,p,0,0,s)$ is a spherical Heronian diamond, then the values~$e,b,p$ determine~$f,d,s$ uniquely, and vice versa.  If~$(a,b,0,d,e,f,0,q,r,0)$ is a spherical Heronian diamond, then the values~$d,e,q$ determine~$b,f,r$ uniquely, and vice versa.
\end{corollary}

We will now  apply what we know about spherical quadrilaterals to the case of a spherical polygon.

\begin{proposition}\label{prop: 3.1 analogue} 
    Let~${P=A_1\dots A_n}$ be a spherical polygon on~$\bfS$. For any four vertices~$A_i,A_j,A_k,A_l$, we can arrange the corresponding ten measurements~$\xS(P)=(x_{ij})\sqcup (\skap(A_i,A_j,A_k))$ to form a Heronian diamond as pictured on the right in Figure~\ref{fig: sph heronian polygon}.

    In particular, the measurements~$\xS(P)=(x_{ij})\sqcup (\skap(A_i,A_j,A_k))$ satisfy the following identities, where we use the shorthand~$\skapijk=\skap(A_i,A_j,A_k)$:
    \begin{align}
(\skapijk)^2&=\hkap(x_{ij},x_{jk},x_{ik})\label{eq: heron relation}\, ,\\
\skapijk+\skapikl&=\skapijl+\skapjkl+\mytfrac{\kap}{2}\left(x_{il}\skapijk+x_{ij}\skapikl-x_{ik}\skapjkl\right)\nonumber\, ,\\
        4x_{ik}x_{jl}\left(1-\mytfrac{\kap x_{ik}}{4}\right)&=(\skapijk+\skapikl)^2+(x_{il}-x_{ij}+x_{jk}-x_{kl})^2\nonumber\\
        &\quad\quad\quad\quad\quad\quad\quad\quad\quad-\kap x_{ik}(x_{il}-x_{ij})(x_{jk}-x_{kl})\nonumber\, ,\\
        x_{ik}(\skapijl-\skapjkl)&=\skapijl(x_{il}-x_{kl})+\skapijk(x_{ij}-x_{jk})\, .\nonumber
    \end{align}

\end{proposition}
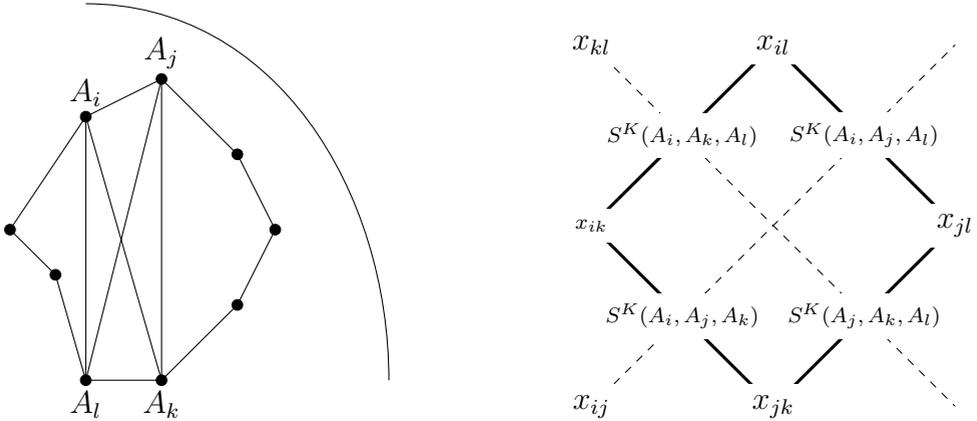
\begin{figure}[htbp!]
        \centering

        \begin{tikzpicture}[scale=1]

        \coordinate  (left) at (-1,5);
        \coordinate  (right) at (3,0);
\draw (left) to [out=0,in=90] (right) ;

\coordinate (v_1) at (1,1);
\filldraw [black] (v_1) circle (2pt) ;
\coordinate (v_2) at (1.5,2);
\filldraw [black] (v_2) circle (2pt) ;
\coordinate (v_3) at (1,3);
\filldraw [black] (v_3) circle (2pt) ;

\coordinate (v_5) at (-2,2);
\filldraw [black] (v_5) circle (2pt) ;
\coordinate (v_6) at (-1.4,1.4);
\filldraw [black] (v_6) circle (2pt) ;

\coordinate [label=above:$A_i$] (i) at (-1,3.5);
\filldraw [black] (i) circle (2pt) ;

\coordinate [label=above:$A_j$](j) at (0,4);
\filldraw [black] (j) circle (2pt);

\coordinate [label=below:$A_l$] (l) at (-1,0);
\filldraw [black] (l) circle (2pt);

\coordinate[label=below:$A_k$] (k) at (0,0);
\filldraw [black] (k) circle (2pt);

\draw (k) to (v_1) to (v_2) to (v_3) to (j) to (i) to (v_5) to (v_6) to (l) to (k);
\draw (i) to (k) to (j) to (l) to (i);
            
        \end{tikzpicture}
        \hspace{2 cm}
                \begin{tikzpicture}
            [scale=0.6,every node/.style={rectangle,fill=white,inner sep=2pt}]

        \begin{scope}[on background layer]
        
        \draw[dashed] (-4,8) node {$x_{kl}$}
        --(4,0);
        \draw[dashed] (-4,0) node {$x_{ij}$}
        --(4,8);
    \end{scope}
        \draw[very thick]
        (0,0) node {$x_{jk}$}
        -- (-2,2) node[xshift=-10pt] {$\skap(A_i,A_j,A_k)$}
        -- (-4,4) node {$x_{ik}$}
        -- (-2,6) node[xshift=-10pt] {$\skap(A_i,A_k,A_l)$}
        -- (0,8) node {$x_{il}$}
        -- (2,6) node[xshift=10pt] {$\skap(A_i,A_j,A_l)$}
        -- (4,4) node {$x_{jl}$}
        -- (2,2) node[xshift=10pt] {$\skap(A_j,A_k,A_l)$}
        -- (0,0);
        \end{tikzpicture}

        \caption{Left: In a spherical polygon, any four vertices~$A_i,A_j,A_k,A_l$ form a spherical quadrilateral. Right: The corresponding spherical Heronian diamond.}
        \label{fig: sph heronian polygon}
    \end{figure}
We next define a spherical Heronian frieze, a tessellation of a planar region by spherical Heronian diamonds that collects and organizes the measurement data for a spherical polygon.  Spherical Heronian friezes generalize to arbitrary values of~$\kap$ the (Euclidean) Heronian friezes introduced by Fomin and Setiabrata~\cite{fomin_heronian_2021}. 
 
 We begin by defining an indexing set for a spherical Heronian frieze.
\begin{definition} \label{def: indexing set}
    For~$n\ge 4$, let
\begin{equation*}
I_n=\left\{\left(i+\mytfrac{1}{2},j\right), \left(i,j+\mytfrac{1}{2}\right), (i,j)\right\}_{\stackrel{0\le j-i\le n}{ i,j\in\Z}}\cup\left\{\left(i+\mytfrac{1}{2},\neline\right)\right\}_{i\in\Z}\cup\left\{\left(\seline,j+\mytfrac{1}{2}\right)\right\}_{j\in\Z}.
\end{equation*}
    Our notation uses the symbols~$\neline$ and~$\seline$, which represent lines headed Northeast and Southeast respectively.  
    
    To~visually represent~$I_n$, rotate the usual coordinate axes in~$\R^2$ by~$\pi/4$ clockwise, so that the points~$(i,j)\in\Z^2$ satisfying the inequalities~$0\le j-i\le n$ would lie inside a horizontal strip.  See a depiction of~$I_6$ in Figure~\ref{fig: sph heron index set}. We call elements of~$I_n$ with two numerical coordinates \emph{nodes}. We classify a node~$(i,j)$ as \emph{interior} if~$1\le j-i\le n-1$, so that~$(i,j)$ lies in the interior of the horizontal strip.
       
\end{definition}

\pagebreak[3]

 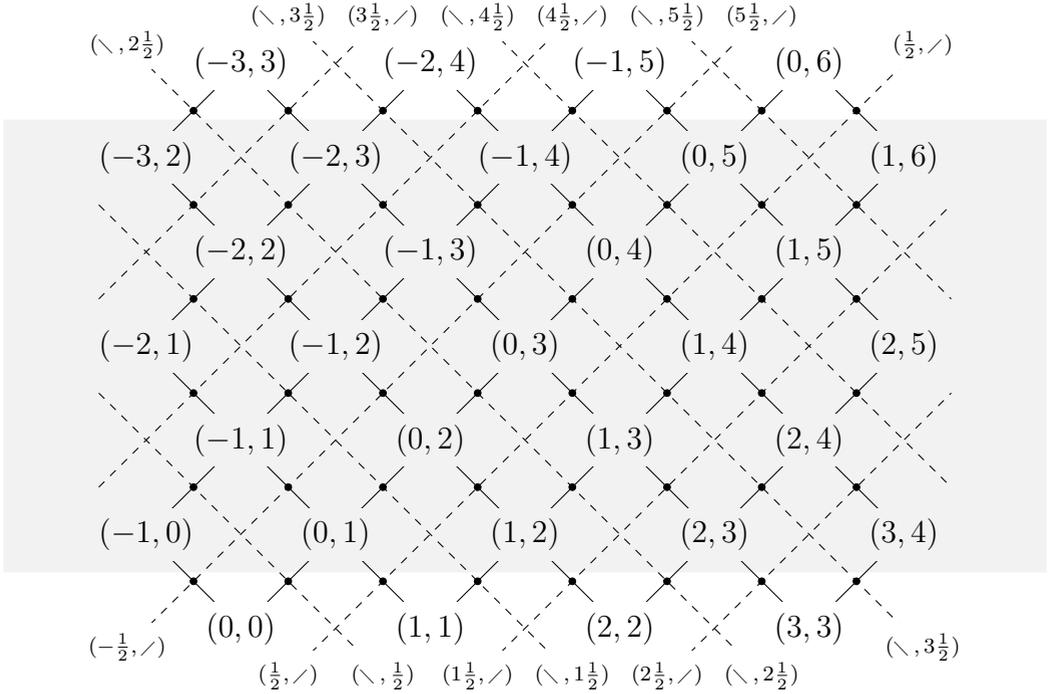
\begin{figure}[htbp!]
        \centering
        \begin{tikzpicture}
        [scale=2.5,intnode/.style={rectangle,fill=black!5,inner sep=1pt}, bdrynode/.style={rectangle,fill=white,inner sep=2pt}]

        \begin{scope}[on background layer] 
\filldraw [black!5] (-1.25,0.35) rectangle (4.25,2.65);
        \draw (-0.5,0.5) -- 
(0,0);

\draw (-0.5,1.5) -- (1,0);

\draw (-0.5,2.5) -- (2,0);

\draw (0,3) -- (3,0);

\draw (1,3) -- (3.5,0.5);

\draw (2,3) -- (3.5,1.5);

\draw (3,3) -- (3.5, 2.5);

\begin{scope}
\draw[dashed] (-0.75,1.75)  
        --(0.75,3.25) node[bdrynode] {\tiny$(3\frac{1}{2},\neline)$}; 

\draw[dashed] (-0.75,.75) 
        --(1.75,3.25) node[bdrynode] {\tiny$(4\frac{1}{2},\neline)$};

        \draw[dashed] (-0.6,-0.1) node[bdrynode] {\tiny$(-\frac{1}{2},\neline)$}
        --(2.75,3.25) node[bdrynode] {\tiny$(5\frac{1}{2},\neline)$};

\draw[dashed] (0.25,-0.25) node[bdrynode] {\tiny$(\frac{1}{2},\neline)$}
        -- (3.6,3.1) node[bdrynode] {\tiny$(\frac{1}{2},\neline)$};

\draw[dashed] (1.25,-0.25) node[bdrynode] {\tiny$(1\frac{1}{2},\neline)$}
        --(3.75,2.25) ;
        
\draw[dashed] (2.25,-0.25) node[bdrynode] {\tiny$(2\frac{1}{2},\neline)$}
        --(3.75,1.25) ;  

\draw[dashed] (-0.75,1.25)  --(0.75,-0.25) node[bdrynode] {\tiny$(\seline,\frac{1}{2})$};

\draw[dashed] (-0.75,2.25)  --(1.75,-0.25) node[bdrynode] {\tiny$(\seline,1\frac{1}{2})$};

\draw[dashed] (-0.6,3.1) node[bdrynode] {\tiny$(\seline,2\frac{1}{2})$} --(2.75,-0.25) node[bdrynode] {\tiny$(\seline,2\frac{1}{2})$};

\draw[dashed] (0.25,3.25) node[bdrynode] {\tiny$(\seline,3\frac{1}{2})$} --(3.6,-0.1) node[bdrynode] {\tiny$(\seline,3\frac{1}{2})$};

\draw[dashed] (1.25,3.25) node[bdrynode] {\tiny$(\seline,4\frac{1}{2})$} --(3.75,0.75);

\draw[dashed] (2.25,3.25) node[bdrynode] {\tiny$(\seline,5\frac{1}{2})$} --(3.75,1.75);

\end{scope}
    \end{scope}
\draw (-0.5,2.5) node[intnode] {$(-3,2)$} -- (0,3) node[bdrynode] {$(-3,3)$};
    
\draw (-0.5,1.5) node[intnode] {$(-2,1)$} -- (0,2) node[intnode] {$(-2,2)$} -- (0.5,2.5) node[intnode] {$(-2,3)$} -- (1,3) node[bdrynode] {$(-2,4)$};

\draw (-0.5,0.5) node[intnode] {$(-1,0)$} -- (0,1) node[intnode] {$(-1,1)$} -- (0.5,1.5) node[intnode] {$(-1,2)$} -- (1,2) node[intnode] {$(-1,3)$} -- (1.5,2.5) node[intnode] {$(-1,4)$} -- (2,3) node[bdrynode] {$(-1,5)$};

        \draw (0,0) node[bdrynode] {$(0,0)$} -- (0.5,0.5) node[intnode] {$(0,1)$} -- (1,1) node[intnode] {$(0,2)$} -- (1.5,1.5) node[intnode] {$(0,3)$} -- (2,2) node[intnode] {$(0,4)$} -- (2.5,2.5) node[intnode] {$(0,5)$} -- (3,3) node[bdrynode] {$(0,6)$};

        \draw (1,0) node[bdrynode] {$(1,1)$} -- (1.5,0.5) node[intnode] {$(1,2)$} -- (2,1) node[intnode] {$(1,3)$} -- (2.5,1.5) node[intnode] {$(1,4)$} -- (3,2) node[intnode] {$(1,5)$} -- (3.5,2.5) node[intnode] {$(1,6)$};
        
        \draw (2,0) node[bdrynode] {$(2,2)$} -- (2.5,0.5) node[intnode] {$(2,3)$} -- (3,1) node[intnode] {$(2,4)$} -- (3.5,1.5) node[intnode] {$(2,5)$};

        \draw (3,0) node[bdrynode] {$(3,3)$} -- (3.5,0.5) node[intnode] {$(3,4)$};

        \filldraw [black] (-0.25,0.25) circle (0.5pt);

        \filldraw [black] (-0.25,0.75) circle (0.5pt);

        \filldraw [black] (-0.25,1.25) circle (0.5pt);

        \filldraw [black] (-0.25,1.75) circle (0.5pt);

        \filldraw [black] (-0.25,2.25) circle (0.5pt);

        \filldraw [black] (-0.25,2.75) circle (0.5pt);

        \filldraw [black] (0.25,0.25) circle (0.5pt);

        \filldraw [black] (0.25,0.75) circle (0.5pt);

        \filldraw [black] (0.25,1.25) circle (0.5pt);

        \filldraw [black] (0.25,1.75) circle (0.5pt);

        \filldraw [black] (0.25,2.25) circle (0.5pt);

        \filldraw [black] (0.25,2.75) circle (0.5pt);

        \filldraw [black] (0.75,0.25) circle (0.5pt);

        \filldraw [black] (0.75,0.75) circle (0.5pt);

        \filldraw [black] (0.75,1.25) circle (0.5pt);

        \filldraw [black] (0.75,1.75) circle (0.5pt);

        \filldraw [black] (0.75,2.25) circle (0.5pt);

        \filldraw [black] (0.75,2.75) circle (0.5pt);
        
        \filldraw [black] (1.25,0.25) circle (0.5pt);

        \filldraw [black] (1.25,0.75) circle (0.5pt);

        \filldraw [black] (1.25,1.25) circle (0.5pt);

        \filldraw [black] (1.25,1.75) circle (0.5pt);

        \filldraw [black] (1.25,2.25) circle (0.5pt);

        \filldraw [black] (1.25,2.75) circle (0.5pt);

        \filldraw [black] (1.75,0.25) circle (0.5pt);

        \filldraw [black] (1.75,0.75) circle (0.5pt);

        \filldraw [black] (1.75,1.25) circle (0.5pt);

        \filldraw [black] (1.75,1.75) circle (0.5pt);

        \filldraw [black] (1.75,2.25) circle (0.5pt);

        \filldraw [black] (1.75,2.75) circle (0.5pt);
        
        \filldraw [black] (2.25,0.25) circle (0.5pt);

        \filldraw [black] (2.25,0.75) circle (0.5pt);

        \filldraw [black] (2.25,1.25) circle (0.5pt);

        \filldraw [black] (2.25,1.75) circle (0.5pt);

        \filldraw [black] (2.25,2.25) circle (0.5pt);

        \filldraw [black] (2.25,2.75) circle (0.5pt);
        
        \filldraw [black] (2.75,0.25) circle (0.5pt);

        \filldraw [black] (2.75,0.75) circle (0.5pt);

        \filldraw [black] (2.75,1.25) circle (0.5pt);

        \filldraw [black] (2.75,1.75) circle (0.5pt);

        \filldraw [black] (2.75,2.25) circle (0.5pt);

        \filldraw [black] (2.75,2.75) circle (0.5pt);
        
        \filldraw [black] (3.25,0.25) circle (0.5pt);

        \filldraw [black] (3.25,0.75) circle (0.5pt);

        \filldraw [black] (3.25,1.25) circle (0.5pt);

        \filldraw [black] (3.25,1.75) circle (0.5pt);

        \filldraw [black] (3.25,2.25) circle (0.5pt);

        \filldraw [black] (3.25,2.75) circle (0.5pt);

        \end{tikzpicture}
        \caption{A  portion of the index set~$I_6$, it extends infinitely to the right and left.  For the points of the form~$\left(i+\frac{1}{2},j\right)$ and~$\left(i,j+\frac{1}{2}\right)$ (indicated with black filled-in circles), the labels are not shown.  Interior nodes lie in the gray shaded region. The dashed lines have labels of the form~$(i+\frac{1}{2},\neline)$ or~$(\seline,j+\frac{1}{2})$.}
        \label{fig: sph heron index set}
    \end{figure}

\begin{definition}
    For~$\kap\in\C$, and~$n\ge 4$, a \emph{spherical Heronian frieze of order~$n$} is an array~${\zfrieze=(z_\alpha)_{\alpha\in I_n}}$ of complex numbers indexed by~$I_n$ (Definition~\ref{def: indexing set}) such that each full diamond of numbers is a spherical Heronian diamond.  Specifically, for an interior node~$(i,j)\in\Z^2$, the tuple
    \[
    (z_{(i,j+1)},z_{(i+\frac{1}{2},\neline)},z_{(i+1,j)},z_{(\seline,j+\frac{1}{2})},z_{(i,j)},z_{(i+1,j+1)},z_{(i+\frac{1}{2},j)},z_{(i,j+\frac{1}{2})},z_{(i+\frac{1}{2},j+1)},z_{(i+1,j+\frac{1}{2})})
    \]
    forms a Heronian diamond of the form shown in Figure~\ref{fig: diamond}.   Additionally, at the boundary of the frieze, we impose the following conditions:
\begin{equation}\label{eq: boundary condition}
    z_{(i,i)}=z_{(i,i+n)}=z_{(i,i+\frac{1}{2})}=z_{(i,i+n-\frac{1}{2})}=0 \quad  (i\in\Z).    
\end{equation}
Figure~\ref{fig: sph heron frieze example} shows an example of a spherical Heronian frieze.

    \begin{remark}
        Let~$\zfrieze$ be a spherical Heronian frieze. Consider the points with first coordinate~0:~$(0,0),  (0,\mytfrac{1}{2}), (0,1),  (0,1+\mytfrac{1}{2}),  (0,2),\dots$. In our visualization they form a straight line of alternating integer and half-integer points traveling in the Northeast direction. Crossing this line perpendicularly at each half integer point are the lines indexed by~${(\seline,j+\frac{1}{2})}$.  The entry~$z_{(0,j+\frac{1}{2})}$ at each half integer point satisfies a Heron relation (equation~\eqref{eq: heron relation}) together with the entries located immediately Northeast and immediately Southwest of it, and the entry associated with the line~$(\seline,j+\frac{1}{2})$.  
    \end{remark}

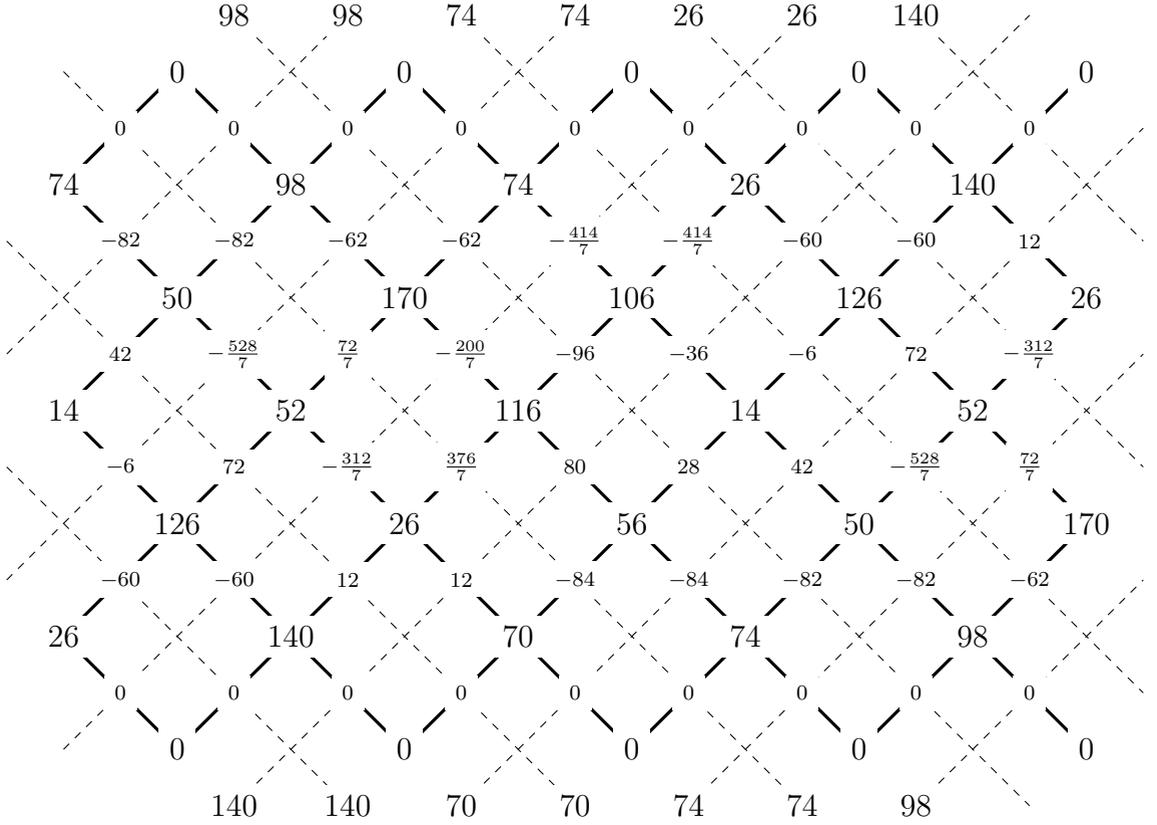
\begin{figure}[htbp!]
        \centering
        \begin{tikzpicture}
        [scale=3,every node/.style={rectangle,fill=white,inner sep=3pt}]

        \begin{scope}[on background layer] 

        \draw[very thick](-0.5,0.5) -- 
(0,0);

\draw[very thick](-0.5,1.5) -- (1,0);

\draw[very thick](-0.5,2.5) -- (2,0);

\draw[very thick](0,3) -- (3,0);

\draw[very thick](1,3) -- (4,0);

\draw[very thick](2,3) -- (4,1);

\draw[very thick](3,3) -- (4, 2);

\begin{scope}
\draw[dashed] (-0.75,1.75) 
        --(0.75,3.25) node {$98$}; 

\draw[dashed] (-0.75,.75) 
        --(1.75,3.25) node {$74$};

        \draw[dashed] (-0.5,0) 
        --(2.75,3.25) node {$26$};

\draw[dashed] (0.25,-0.25)  node {$140$}
        -- (3.75 ,3.25);

\draw[dashed] (1.25,-0.25)  node {$70$}
        --(4.25,2.75);
        
\draw[dashed] (2.25,-0.25)  node {$74$}
        --(4.25,1.75);  

\draw[dashed] (3.25,-0.25) node {$98$} --(4.25,0.75);

\draw[dashed] (-0.75,1.25)  --(0.75,-0.25) node {$140$};

\draw[dashed] (-0.75,2.25)  --(1.75,-0.25) node {$70$};

\draw[dashed] (-0.5,3)   --(2.75,-0.25) node {$74$};

\draw[dashed] (0.25,3.25) node {$98$} --(3.75,-0.25);

\draw[dashed] (1.25,3.25) node {$74$} --(4.25,0.25);

\draw[dashed] (2.25,3.25) node {$26$} --(4.25,1.25);

\draw[dashed] (3.25,3.25) node {$140$} -- (4.25,2.25);

\end{scope}
    \end{scope}
\draw[very thick](-0.5,2.5) node {$74$} -- (0,3) node {$0$};
    
\draw[very thick](-0.5,1.5) node {$14$} -- (0,2) node {$50$} -- (0.5,2.5) node {$98$} -- (1,3) node {$0$};

\draw[very thick](-0.5,0.5) node {$26$} -- (0,1) node {$126$} -- (0.5,1.5) node {$52$} -- (1,2) node {$170$} -- (1.5,2.5) node {$74$} -- (2,3) node {$0$};

        \draw[very thick](0,0) node {$0$} -- (0.5,0.5) node {$140$} -- (1,1) node {$26$} -- (1.5,1.5) node {$116$} -- (2,2) node {$106$} -- (2.5,2.5) node {$26$} -- (3,3) node {$0$};

        \draw[very thick](1,0) node {$0$} -- (1.5,0.5) node {$70$} -- (2,1) node {$56$} -- (2.5,1.5) node {$14$} -- (3,2) node {$126$} -- (3.5,2.5) node {$140$} -- (4,3) node {$0$};
        
        \draw[very thick](2,0) node {$0$} -- (2.5,0.5) node {$74$} -- (3,1) node {$50$} -- (3.5,1.5) node {$52$} -- (4,2) node {$26$};

        \draw[very thick](3,0) node {$0$} -- (3.5,0.5) node {$98$} -- (4,1) node {$170$};

        \draw[very thick](4,0) node {$0$};

        \filldraw [black] (-0.25,0.25) circle (0.5pt) node {\tiny$0$};

        \filldraw [black] (-0.25,0.75) circle (0.5pt) node {\tiny$-60$};

        \filldraw [black] (-0.25,1.25) circle (0.5pt) node {\tiny$-6$};

        \filldraw [black] (-0.25,1.75) circle (0.5pt) node {\tiny$42$};

        \filldraw [black] (-0.25,2.25) circle (0.5pt) node {\tiny$-82$};

        \filldraw [black] (-0.25,2.75) circle (0.5pt) node {\tiny$0$};

        \filldraw [black] (0.25,0.25) circle (0.5pt) node {\tiny \tiny$0$};

        \filldraw [black] (0.25,0.75) circle (0.5pt) node {\tiny$-60$};

        \filldraw [black] (0.25,1.25) circle (0.5pt) node {\tiny$72$};

        \filldraw [black] (0.25,1.75) circle (0.5pt) node {\tiny$-\frac{528}{7}$};

        \filldraw [black] (0.25,2.25) circle (0.5pt) node {\tiny$-82$};

        \filldraw [black] (0.25,2.75) circle (0.5pt) node {\tiny$0$};

        \filldraw [black] (0.75,0.25) circle (0.5pt) node {\tiny \tiny$0$};

        \filldraw [black] (0.75,0.75) circle (0.5pt) node {\tiny$12$};

        \filldraw [black] (0.75,1.25) circle (0.5pt) node {\tiny$-\frac{312}{7}$};

        \filldraw [black] (0.75,1.75) circle (0.5pt) node {\tiny$\frac{72}{7}$};

        \filldraw [black] (0.75,2.25) circle (0.5pt)  node {\tiny$-62$};

        \filldraw [black] (0.75,2.75) circle (0.5pt)  node {\tiny$0$};
        
        \filldraw [black] (1.25,0.25) circle (0.5pt)  node {\tiny$0$};

        \filldraw [black] (1.25,0.75) circle (0.5pt)  node {\tiny$12$};

        \filldraw [black] (1.25,1.25) circle (0.5pt)  node {\tiny$\frac{376}{7}$};

        \filldraw [black] (1.25,1.75) circle (0.5pt) node {\tiny$-\frac{200}{7}$};

        \filldraw [black] (1.25,2.25) circle (0.5pt)  node {\tiny$-62$};

        \filldraw [black] (1.25,2.75) circle (0.5pt)  node {\tiny$0$};

        \filldraw [black] (1.75,0.25) circle (0.5pt)  node {\tiny$0$};

        \filldraw [black] (1.75,0.75) circle (0.5pt)  node {\tiny$-84$};

        \filldraw [black] (1.75,1.25) circle (0.5pt) node {\tiny$80$};

        \filldraw [black] (1.75,1.75) circle (0.5pt) node {\tiny$-96$};

        \filldraw [black] (1.75,2.25) circle (0.5pt)  node {\tiny$-\frac{414}{7}$};

        \filldraw [black] (1.75,2.75) circle (0.5pt)  node {\tiny$0$};
        
        \filldraw [black] (2.25,0.25) circle (0.5pt)  node {\tiny$0$};

        \filldraw [black] (2.25,0.75) circle (0.5pt) node {\tiny$-84$};

        \filldraw [black] (2.25,1.25) circle (0.5pt) node {\tiny$28$};

        \filldraw [black] (2.25,1.75) circle (0.5pt) node {\tiny$-36$};

        \filldraw [black] (2.25,2.25) circle (0.5pt) node {\tiny$-\frac{414}{7}$};

        \filldraw [black] (2.25,2.75) circle (0.5pt) node {\tiny$0$};
        
        \filldraw [black] (2.75,0.25) circle (0.5pt) node {\tiny$0$};

        \filldraw [black] (2.75,0.75) circle (0.5pt) node {\tiny$-82$};

        \filldraw [black] (2.75,1.25) circle (0.5pt) node {\tiny$42$};

        \filldraw [black] (2.75,1.75) circle (0.5pt) node {\tiny$-6$};

        \filldraw [black] (2.75,2.25) circle (0.5pt) node {\tiny$-60$};

        \filldraw [black] (2.75,2.75) circle (0.5pt) node {\tiny$0$};
        
        \filldraw [black] (3.25,0.25) circle (0.5pt) node {\tiny$0$};

        \filldraw [black] (3.25,0.75) circle (0.5pt) node {\tiny$-82$};

        \filldraw [black] (3.25,1.25) circle (0.5pt) node {\tiny$-\frac{528}{7}$};

        \filldraw [black] (3.25,1.75) circle (0.5pt) node {\tiny$72$};

        \filldraw [black] (3.25,2.25) circle (0.5pt) node {\tiny$-60$};

        \filldraw [black] (3.25,2.75) circle (0.5pt) node {\tiny$0$};

        \filldraw [black] (3.75,0.25) circle (0.5pt) node {\tiny$0$};

        \filldraw [black] (3.75,.75) circle (0.5pt) node {\tiny$-62$};

        \filldraw [black] (3.75,1.25) circle (0.5pt) node {\tiny$\frac{72}{7}$};

        \filldraw [black] (3.75,1.75) circle (0.5pt) node {\tiny$-\frac{312}{7}$};

        \filldraw [black] (3.75,2.25) circle (0.5pt) node {\tiny$12$};

        \filldraw [black] (3.75,2.75) circle (0.5pt) node {\tiny$0$};

        \end{tikzpicture}
        \caption{A spherical Heronian frieze of order~$6$.  This frieze corresponds to a hexagon on a sphere with radius~$7$.}
        \label{fig: sph heron frieze example}
    \end{figure}
  
\end{definition}

\begin{definition}\label{def: polygon to sph her frieze}
    Let~$P=A_1\dots A_n$ be a spherical polygon on~$\bfS$.  Define~$\zfrieze=\zfrieze(P)$, a spherical Heronian frieze of order~$n$, by setting
    \begin{align*}
        z_{(i,j)}&=x_{\lr{i}\lr{j}}\, ,\\
        z_{(i+\frac{1}{2},j)}&=\skap_{\lr{i}\lr{i+1}\lr{j}}\, ,\\
        z_{(i,j+\frac{1}{2})}&=\skap_{\lr{i}\lr{j}\lr{j+1}}\, ,\\
        z_{(i+\frac{1}{2},\seline)}&=x_{\lr{i}\lr{i+1}},\\
        z_{(\neline,j+\frac{1}{2})}&=x_{\lr{j}\lr{j+1}}\, ,
    \end{align*}
    where~$\lr{m}$ denotes the unique integer~$k$ in~$\{1,\dots,n\}$ satisfying~$m\equiv k \, (\text{mod } n)$.  Condition~\eqref{eq: boundary condition} on the boundary holds because for any~$i\in\Z$ we have \[x_{ii}=\skap_{\langle i\rangle\langle i\rangle\langle i+1\rangle}=\skap_{\langle i\rangle\langle i+1\rangle\langle i+1\rangle}=0.\]

    In words, we put the~$\skap$ measurements of spherical triangles at each half integer point, e.g.,~$(i+\frac{1}{2},j)$.  At the  integer points directly above and below, we surround a particular triangle measurement with the squared side lengths  of the triangle.  There is one perpendicular line passing through the half integer point, which we  label with the third side of the triangle.  As we move up and to the right along a line we alternate between integer and half-integer points, ergo side length and triangle measurements.  Each step along the line represents a small movement, from a side to an adjacent triangle, to another side of the triangle, to another adjacent triangle, slowly walking through the polygon.
    By design, the only triangle measurements that appear in a frieze come from triangles that share a side with the polygon.
\end{definition}

\begin{example}
    The frieze in Figure~\ref{fig: sph heron frieze example} corresponds to a hexagon on the sphere of radius~$7$ centered at the origin.  It is depicted in Figure~\ref{fig: poly example}.
\end{example}
\begin{figure}[htbp!]
    \centering
\begin{tikzpicture}
    \node (background) at (0,0) {\includegraphics[width=0.4\textwidth]{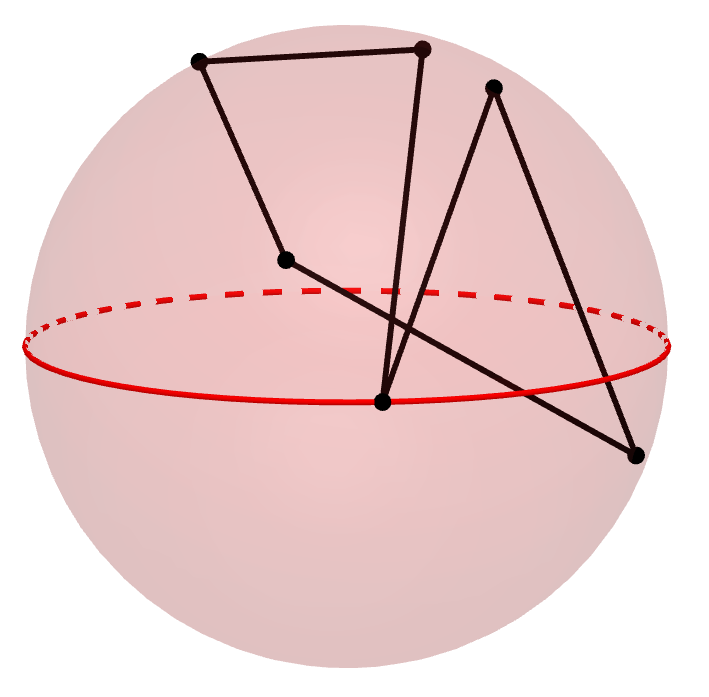}};

    \node[inner sep=1pt,label=below:{$A_1=(7,0,0)$}] (a1) at (0,-.5) {};
    \node[inner sep=1pt,label=above:{$A_2=(2,3,6)$}] (a2) at (2.5,2.2) {};
    \node[inner sep=1pt,label=below right:{$A_3=(3,6,-2)$}] (a3) at (2.5,-.8) {};
    \node[inner sep=1pt,label=left:{$A_4=(6,-2,3)$}] (a4) at (-1,.8) {};
    \node[inner sep=1pt,label=left:{$A_5=(-2,-3,6)$}] (a5) at (-1.5,2.8) {};
    \node[inner sep=1pt,label=above:{$A_6=(-3,2,6)$}] (a6) at (0,2.7) {};
\end{tikzpicture}
    \caption{The vertices of the spherical hexagon~$A_1A_2A_3A_4A_5A_6$ lie on the sphere of radius~$7$ centered at the origin in~$\R^3$.}
    \label{fig: poly example}
\end{figure}
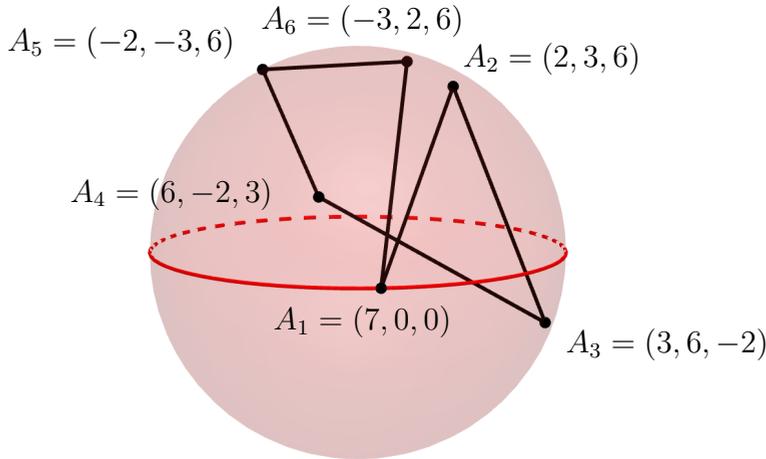

For a spherical Heronian frieze arising from a spherical polygon, the interior nodes~$z_{(i,i+1)}, \, z_{(i,i+n-1)}$ along the top and bottom of the frieze correspond to the squared distances~$x_{\lr{i}\lr{i+1}},\, x_{\lr{i}\lr{i+n-1}}$ respectively---the squared side lengths of the polygon.  These quantities also appear as the dashed lines that run through multiple spherical Heronian diamonds.  We don't need this geometric interpretation, however, to get the following purely algebraic result.

\begin{proposition} Let~$\zfrieze=(z_\alpha)_{\alpha\in I_n}$ be a spherical Heronian frieze of order~$n$.  Then
\[\begin{array}{r@{\ }l@{\ }l@{\ }l}
z_{(i,i+1)} &= z_{(i+\frac{1}{2},\neline)} &= z_{(\seline,i+\frac{1}{2})} & \quad (i\in\Z)\, ,\\
z_{(i,i+n-1)} &= z_{(i-\frac{1}{2},\neline)} &= z_{(\seline,i+n-\frac{1}{2})} & \quad (i\in\Z)\, .
\end{array}
\]    
\end{proposition}
\begin{proof}
    Looking at the spherical Heronian diamond with~$z_{(i,i+1)}$ on the far left (in the position~$e$ in Figure~\ref{fig: diamond}), we see condition~\eqref{eq: boundary condition} ensures the bottom entries are all zero.  In other words,
    \begin{equation*}
        z_{(i+1,i+1)}=z_{(i+\frac{1}{2},i+1)}=z_{(i+1,i+\frac{3}{2})}=0\, .
    \end{equation*}
    This means that Lemma~\ref{lem: degenerate 2} applies and we get~$z_{(i,i+1)}=z_{(i+\frac{1}{2},\neline)}$ (by~\eqref{eq: bdry 4}).  The other results can be proved similarly.\end{proof}

\begin{remark}\label{rem: glide symmetry}
    
A spherical polygon~$P$ with~$n$ vertices has a finite number of sides, diagonals, and triangulations. As in the Euclidean case (see~\cite[(3.15)-(3.17)]{fomin_heronian_2021}), a spherical Heronian frieze~$\zfrieze(P)$ of order~$n$ arising from~$P$ is therefore periodic.  Specifically, the frieze~$\zfrieze(P)$ must satisfy
\begin{gather*}
    z_{(i,j)}=x_{\lr{i}\lr{j}}=z_{(i+n,j+n)} \, ,\\
    z_{(i+\frac{1}{2},\neline)}=x_{\lr{i}\lr{i+1}}=z_{(i+\frac{1}{2}+n,\neline)}\, ,\\
    z_{(\seline,j+\frac{1}{2})}=x_{\lr{j}\lr{j+1}}=z_{(\seline,j+\frac{1}{2}+n)}\, .
\end{gather*}
In particular, because of the symmetry of the measurements (see Definition~\ref{def: squared distance} and Definition~\ref{def: skappa def}), for~$i,j\in\Z$ we must have \emph{glide symmetry} (cf.~Figure~\ref{fig: sph heron glide symm}):
\begin{gather*}
z_{(i,j)}=x_{\lr{i}\lr{j}}=x_{\lr{j}\lr{i}}=z_{(j,i+n)}\\
z_{(i+\frac{1}{2},j)}=\skap_{\lr{i}\lr{i+1}\lr{j}}=\skap_{\lr{j}\lr{i}\lr{i+1}}=z_{(j,i+\frac{1}{2}+n)}\\
z_{(i,j+\frac{1}{2})}=\skap_{\lr{i}\lr{j}\lr{j+1}}=\skap_{\lr{j}\lr{j+1}\lr{i}}=z_{(j+\frac{1}{2},i+n)}.
\end{gather*}

\end{remark}
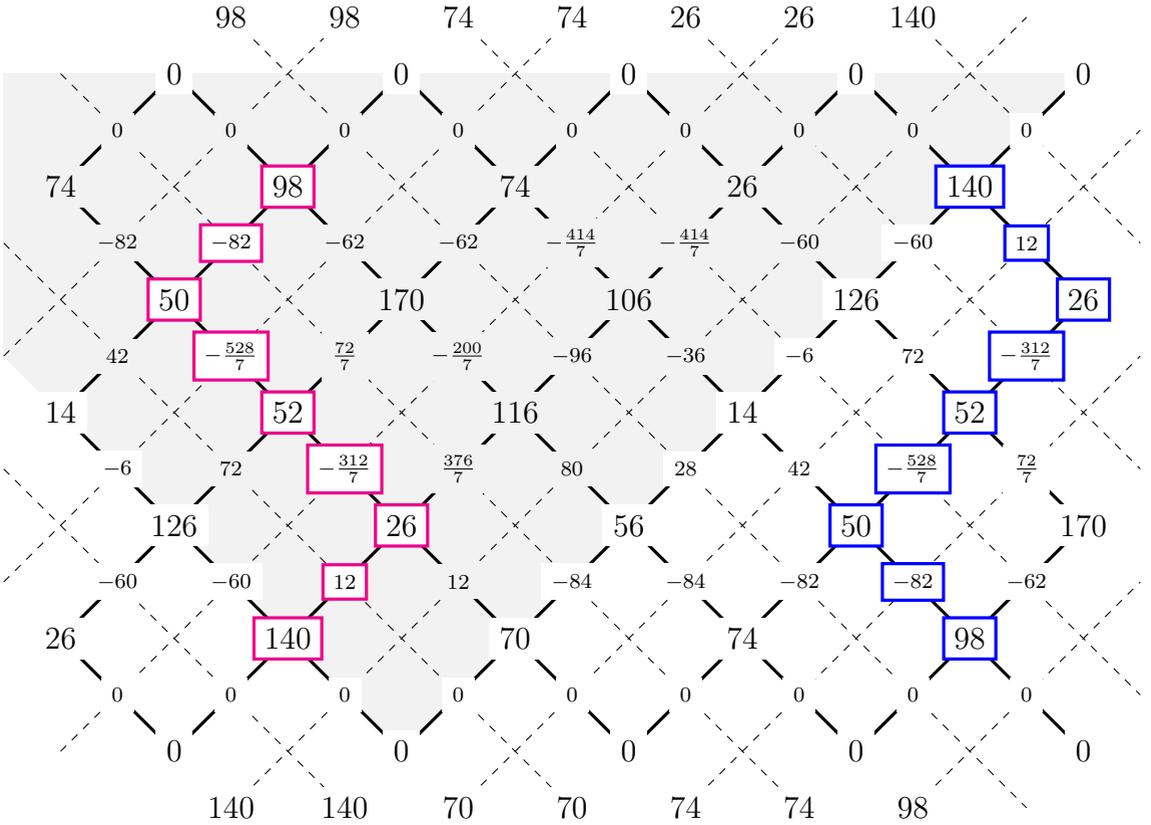
\begin{figure}[htbp!]
        \centering
        \begin{tikzpicture}
        [scale=3, greynode/.style={rectangle,fill=black!5,inner sep=2pt}, whitenode/.style={rectangle,fill=white, inner sep=2pt}, bordergreynode/.style={rectangle,draw=magenta,very thick,fill=white!5,inner sep=3pt}, borderwhitenode/.style={rectangle,draw=blue,fill=white,very thick,inner sep=3pt}] 

        \begin{scope}[on background layer] 
\filldraw[black!5] (-0.75,2) -- (-0.75,1.75) -- (1,0) -- (4,3) -- (-0.75,3) -- (-0.75,2);
        \draw[very thick](-0.5,0.5) -- 
(0,0);

\draw[very thick](-0.5,1.5) -- (1,0);

\draw[very thick](-0.5,2.5) -- (2,0);

\draw[very thick](0,3) -- (3,0);

\draw[very thick](1,3) -- (4,0);

\draw[very thick](2,3) -- (4,1);

\draw[very thick](3,3) -- (4, 2);

\begin{scope}
\draw[dashed] (-0.75,1.75) 
        --(0.75,3.25) node[whitenode] {$98$}; 

\draw[dashed] (-0.75,.75) 
        --(1.75,3.25) node[whitenode] {$74$};

        \draw[dashed] (-0.5,0) 
        --(2.75,3.25) node[whitenode] {$26$};

\draw[dashed] (0.25,-0.25)  node[whitenode] {$140$}
        -- (3.75 ,3.25);

\draw[dashed] (1.25,-0.25)  node[whitenode] {$70$}
        --(4.25,2.75);
        
\draw[dashed] (2.25,-0.25)  node[whitenode] {$74$}
        --(4.25,1.75);  

\draw[dashed] (3.25,-0.25) node[whitenode] {$98$} --(4.25,0.75);

\draw[dashed] (-0.75,1.25)  --(0.75,-0.25) node[whitenode] {$140$};

\draw[dashed] (-0.75,2.25)  --(1.75,-0.25) node[whitenode] {$70$};

\draw[dashed] (-0.5,3)   --(2.75,-0.25) node[whitenode] {$74$};

\draw[dashed] (0.25,3.25) node[whitenode] {$98$} --(3.75,-0.25);

\draw[dashed] (1.25,3.25) node[whitenode] {$74$} --(4.25,0.25);

\draw[dashed] (2.25,3.25) node[whitenode] {$26$} --(4.25,1.25);

\draw[dashed] (3.25,3.25) node[whitenode] {$140$} -- (4.25,2.25);

\end{scope}
    \end{scope}
\draw[very thick](-0.5,2.5) node[greynode] {$74$} -- (0,3) node[whitenode] {$0$};
    
\draw[very thick](-0.5,1.5) node[whitenode] {$14$} -- (0,2) node[bordergreynode] {$50$} -- (0.5,2.5) node[bordergreynode] {$98$} -- (1,3) node[whitenode] {$0$};

\draw[very thick](-0.5,0.5) node[whitenode] {$26$} -- (0,1) node[whitenode] {$126$} -- (0.5,1.5) node[bordergreynode] {$52$} -- (1,2) node[greynode] {$170$} -- (1.5,2.5) node[greynode] {$74$} -- (2,3) node[whitenode] {$0$};

        \draw[very thick](0,0) node[whitenode] {$0$} -- (0.5,0.5) node[bordergreynode] {$140$} -- (1,1) node[bordergreynode] {$26$} -- (1.5,1.5) node[greynode] {$116$} -- (2,2) node[greynode] {$106$} -- (2.5,2.5) node[greynode] {$26$} -- (3,3) node[whitenode] {$0$};

        \draw[very thick](1,0) node[whitenode] {$0$} -- (1.5,0.5) node[whitenode] {$70$} -- (2,1) node[whitenode] {$56$} -- (2.5,1.5) node[whitenode] {$14$} -- (3,2) node[whitenode] {$126$} -- (3.5,2.5) node[borderwhitenode] {$140$} -- (4,3) node[whitenode] {$0$};
        
        \draw[very thick](2,0) node[whitenode] {$0$} -- (2.5,0.5) node[whitenode] {$74$} -- (3,1) node[borderwhitenode] {$50$} -- (3.5,1.5) node[borderwhitenode] {$52$} -- (4,2) node[borderwhitenode] {$26$};

        \draw[very thick](3,0) node[whitenode] {$0$} -- (3.5,0.5) node[borderwhitenode] {$98$} -- (4,1) node[whitenode] {$170$};

        \draw[very thick](4,0) node[whitenode] {$0$};

        \filldraw [black] (-0.25,0.25) circle (0.5pt) node[whitenode] {\tiny$0$};

        \filldraw [black] (-0.25,0.75) circle (0.5pt) node[whitenode] {\tiny$-60$};

        \filldraw [black] (-0.25,1.25) circle (0.5pt) node[whitenode] {\tiny$-6$};

        \filldraw [black] (-0.25,1.75) circle (0.5pt) node[greynode] {\tiny$42$};

        \filldraw [black] (-0.25,2.25) circle (0.5pt) node[greynode] {\tiny$-82$};

        \filldraw [black] (-0.25,2.75) circle (0.5pt) node[greynode] {\tiny$0$};

        \filldraw [black] (0.25,0.25) circle (0.5pt) node[whitenode] {\tiny \tiny$0$};

        \filldraw [black] (0.25,0.75) circle (0.5pt) node[whitenode] {\tiny$-60$};

        \filldraw [black] (0.25,1.25) circle (0.5pt) node[greynode] {\tiny$72$};

        \filldraw [black] (0.25,1.75) circle (0.5pt) node[bordergreynode] {\tiny$-\frac{528}{7}$};

        \filldraw [black] (0.25,2.25) circle (0.5pt) node[bordergreynode] {\tiny$-82$};

        \filldraw [black] (0.25,2.75) circle (0.5pt) node[greynode] {\tiny$0$};

        \filldraw [black] (0.75,0.25) circle (0.5pt) node[whitenode] {\tiny \tiny$0$};

        \filldraw [black] (0.75,0.75) circle (0.5pt) node[bordergreynode] {\tiny$12$};

        \filldraw [black] (0.75,1.25) circle (0.5pt) node[bordergreynode] {\tiny$-\frac{312}{7}$};

        \filldraw [black] (0.75,1.75) circle (0.5pt) node[greynode] {\tiny$\frac{72}{7}$};

        \filldraw [black] (0.75,2.25) circle (0.5pt)  node[greynode] {\tiny$-62$};

        \filldraw [black] (0.75,2.75) circle (0.5pt)  node[greynode] {\tiny$0$};
        
        \filldraw [black] (1.25,0.25) circle (0.5pt)  node[whitenode] {\tiny$0$};

        \filldraw [black] (1.25,0.75) circle (0.5pt)  node[greynode] {\tiny$12$};

        \filldraw [black] (1.25,1.25) circle (0.5pt)  node[greynode] {\tiny$\frac{376}{7}$};

        \filldraw [black] (1.25,1.75) circle (0.5pt) node[greynode] {\tiny$-\frac{200}{7}$};

        \filldraw [black] (1.25,2.25) circle (0.5pt)  node[greynode] {\tiny$-62$};

        \filldraw [black] (1.25,2.75) circle (0.5pt)  node[greynode] {\tiny$0$};

        \filldraw [black] (1.75,0.25) circle (0.5pt)  node[whitenode] {\tiny$0$};

        \filldraw [black] (1.75,0.75) circle (0.5pt)  node[whitenode] {\tiny$-84$};

        \filldraw [black] (1.75,1.25) circle (0.5pt) node[greynode] {\tiny$80$};

        \filldraw [black] (1.75,1.75) circle (0.5pt) node[greynode] {\tiny$-96$};

        \filldraw [black] (1.75,2.25) circle (0.5pt)  node[greynode] {\tiny$-\frac{414}{7}$};

        \filldraw [black] (1.75,2.75) circle (0.5pt)  node[greynode] {\tiny$0$};
        
        \filldraw [black] (2.25,0.25) circle (0.5pt)  node[whitenode] {\tiny$0$};

        \filldraw [black] (2.25,0.75) circle (0.5pt) node[whitenode] {\tiny$-84$};

        \filldraw [black] (2.25,1.25) circle (0.5pt) node[whitenode] {\tiny$28$};

        \filldraw [black] (2.25,1.75) circle (0.5pt) node[greynode] {\tiny$-36$};

        \filldraw [black] (2.25,2.25) circle (0.5pt) node[greynode] {\tiny$-\frac{414}{7}$};

        \filldraw [black] (2.25,2.75) circle (0.5pt) node[greynode] {\tiny$0$};
        
        \filldraw [black] (2.75,0.25) circle (0.5pt) node[whitenode] {\tiny$0$};

        \filldraw [black] (2.75,0.75) circle (0.5pt) node[whitenode] {\tiny$-82$};

        \filldraw [black] (2.75,1.25) circle (0.5pt) node[whitenode] {\tiny$42$};

        \filldraw [black] (2.75,1.75) circle (0.5pt) node[whitenode] {\tiny$-6$};

        \filldraw [black] (2.75,2.25) circle (0.5pt) node[greynode] {\tiny$-60$};

        \filldraw [black] (2.75,2.75) circle (0.5pt) node[greynode] {\tiny$0$};
        
        \filldraw [black] (3.25,0.25) circle (0.5pt) node[whitenode] {\tiny$0$};

        \filldraw [black] (3.25,0.75) circle (0.5pt) node[borderwhitenode] {\tiny$-82$};

        \filldraw [black] (3.25,1.25) circle (0.5pt) node[borderwhitenode] {\tiny$-\frac{528}{7}$};

        \filldraw [black] (3.25,1.75) circle (0.5pt) node[whitenode] {\tiny$72$};

        \filldraw [black] (3.25,2.25) circle (0.5pt) node[whitenode] {\tiny$-60$};

        \filldraw [black] (3.25,2.75) circle (0.5pt) node[greynode] {\tiny$0$};

        \filldraw [black] (3.75,0.25) circle (0.5pt) node[whitenode] {\tiny$0$};

        \filldraw [black] (3.75,.75) circle (0.5pt) node[whitenode] {\tiny$-62$};

        \filldraw [black] (3.75,1.25) circle (0.5pt) node[whitenode] {\tiny$\frac{72}{7}$};

        \filldraw [black] (3.75,1.75) circle (0.5pt) node[borderwhitenode] {\tiny$-\frac{312}{7}$};

        \filldraw [black] (3.75,2.25) circle (0.5pt) node[borderwhitenode] {\tiny$12$};

        \filldraw [black] (3.75,2.75) circle (0.5pt) node[whitenode] {\tiny$0$};

        \end{tikzpicture}
        \caption{The same spherical Heronian frieze as in Figure~\ref{fig: sph heron frieze example}, this time illustrating the glide symmetry. To~see it, flip the gray shaded region upside down and paste it to the right of its original location.
        The blue and magenta highlighted entries form two separate traversing paths.  Thanks to the glide symmetry, their entries form identical sets.  Both paths correspond to the triangulation shown in Figure~\ref{fig: poly triangulated}.}
        \label{fig: sph heron glide symm}
    \end{figure}

We will next show that, like in the Euclidean case, a small amount of data determines the entire spherical Heronian frieze through a propagation recurrence. This propagation requires careful selection of initial data. One suitable choice is the set of entries along a \emph{traversing path}.  The following statements mirror those in~\cite{fomin_heronian_2021}.

\begin{definition}\label{def: traversing path heronian}
    A \emph{traversing path} in a spherical Heronian frieze of order~$n$ is an ordered collection
    \[
    \pi=((i_1,j_1),\dots,(i_{2n-3},j_{2n-3}),l_1,\dots,l_{n-2})
    \]
    of~$3n-5$ indices in~$I_n$ that can be viewed as stepping stones on a shortest path from the lower to the upper boundary of the frieze (see Figure~\ref{fig: sph heron glide symm}).  We require that the first point is of the form~$(i_1,j_1)=(k,k+1)$ for some~$k\in\Z$ (a point almost on the bottom boundary); similarly~$(i_{2n-3},j_{2n-3})=(m,m+n-1)$ for some~$m\in\Z$ (just below the top boundary).
 In particular, all~$(i_k,j_k)$ are  interior nodes. The~$l_1,\dots,l_{n-2}$ are the dashed lines that intersect this shortest path. 
\end{definition}

\begin{remark}\label{rem: 3.9 analogue}
    If a spherical Heronian frieze~$\zfrieze(P)$ arises from a spherical polygon~$P$, then the path~$\pi$ corresponds to a triangulation of~$P$. To~see this, start at~$(i_1,j_1)$ with value~$(k,k+1)$ at which we have~$z_{(k,k+1)}=x_{\lr{k}\lr{k+1}}$. 
    
    Imagine two ants sitting at the boundary of polygon~$P$: ant Alice at vertex~$A_k$ and ant Bob at vertex~$A_{k+1}$. As we move up along~$\pi$ to~$(i_2,j_2)$, we get to~$(k,k+1+\frac{1}{2})$ (for example) with~$z_{(k,k+1+\frac{1}{2})}$ corresponding to~$\skap_{\lr{k}\lr{k+1}\lr{k+2}}$ and we include the associated triangle~$A_kA_{k+1}A_{k+2}$ in our triangulation.  Then we move to~$z_{(k,k+2)}$.  Bob moves to vertex~$A_{k+2}$ of~$P$.  
    
    Proceed similarly.  As we continue on the path~$\pi$, Alice moves to the vertex corresponding to the first coordinate, and Bob to the vertex corresponding to the second coordinate.  Note that since only one number changes at each step in an alternating fashion, the triangles form a zigzagging triangulation of~$P$ and the only triangles we include have one side along the boundary of~$P$.

    As a side note, the lines~$l_1,\dots,l_{n-2}$ are not needed to pass from a traversing path to a triangulation of~$P$.  They will be needed in the algebraic result (Corollary~\ref{cor: 3.10 analogue}) below.

    Since the traversing path gives  a triangulation of the polygon, by Proposition~\ref{prop: cor 2.9 analogue} we can recover all the measurements of the polygon, hence the entire frieze.

\end{remark}
\begin{figure}[htbp!]
    \centering
\begin{tikzpicture}
[scale=1.6, tnode/.style={blue},regnode/.style={black}]
\node[inner sep=1pt,label=below:{$A_1=(7,0,0)$},draw,circle,fill=black] (a1) at (0.2,-.4) {};
    \node[inner sep=1pt,label=below right:{$A_2=(2,3,6)$},draw,circle,fill=black] (a2) at (1.2,2.2) {};
    \node[inner sep=1pt,label=below right:{$A_3=(3,6,-2)$},draw,circle,fill=black] (a3) at (2.4,-.9) {};
    \node[inner sep=1pt,label=left:{$A_4=(6,-2,3)$},draw,circle,fill=black] (a4) at (-.6,.75) {};
    \node[inner sep=1pt,label=left:{$A_5=(-2,-3,6)$},draw,circle,fill=black] (a5) at (-1.3,2.5) {};
    \node[inner sep=1pt,label=above:{$A_6=(-3,2,6)$},draw,circle,fill=black] (a6) at (.6,2.6) {};
   
    \draw (a1) -- (a2) -- (a3) -- (a4) -- (a5) -- (a6) -- (a1);

\draw[dotted] (a2) -- node[tnode,above right, inner sep=2pt] {\footnotesize${x_{26}=26}$} (a6);

\draw[dotted] (a2) -- node[tnode,below left] {\footnotesize${x_{25}=52}$} (a5);

\draw[ dotted] (a2) -- node[tnode, left, inner sep=5pt] {\footnotesize${x_{42}=50}$} (a4);

\draw[dotted] (a3) -- node[tnode, above right ] {\footnotesize${x_{43}=98}$} (a4);

\draw (0.7,1.2) to [out=10,in=160] (2.2,1.2) node[right,blue] {$\skap_{612}=12$};

\draw (-0.2,2.5) to [out=100,in=350] (-1,3.2) node[left,blue] {$\skap_{562}=-\frac{312}{7}$};

\draw (-0.8,1.8) to [out=170,in=10] (-1.8,1.8) node[left,blue] {$\skap_{452}=-\frac{528}{7}$};

\draw (1.4,0.7) to [out=0,in=200] (2.5,0.5) node[right,blue] {$\skap_{423}=-82$};

    \end{tikzpicture}

    \caption{Frieze entries for the hexagon~$A_1A_2A_3A_4A_5A_6$ from Figure~\ref{fig: poly example}. These frieze entries lie on the traversing path shown in Figure~\ref{fig: sph heron glide symm}.}
    \label{fig: poly triangulated}
\end{figure}
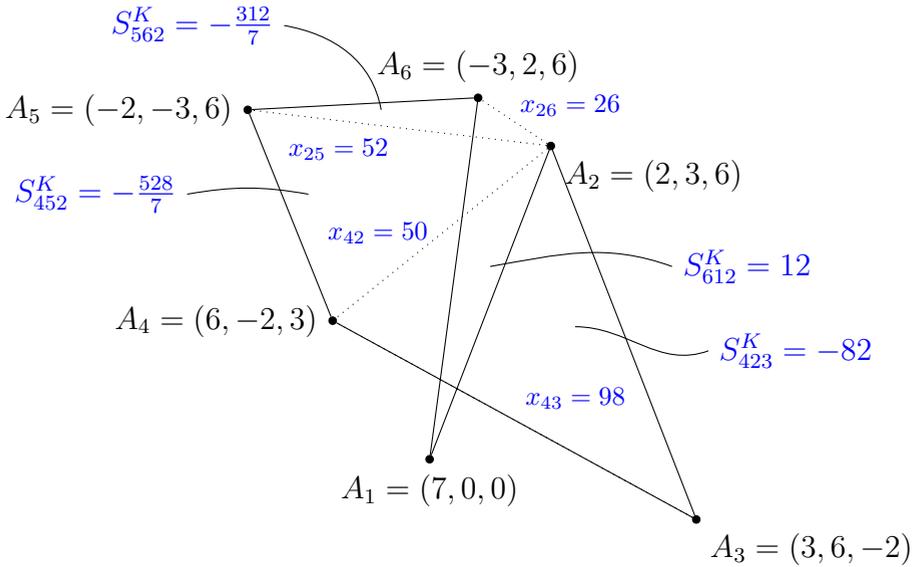

\begin{corollary}\label{cor: 3.10 analogue} Let~$(z_\alpha)_{\alpha\in I_n}$ be a spherical Heronian frieze of order~$n$ such that
\begin{equation}\label{eq: non-vanishing condition}
    z_{(i,j)}\not\in \left\{0, \mytfrac{4}{\kap}\right\} \text{ for any } (i,j)\in\Z^2 \text{ with } 2\le j-i \le n-2\, .
\end{equation}
  Then starting from a single traversing path~$\pi$, we can recover the entire frieze.  
\end{corollary}
\begin{proof}
    Starting with the entries adjacent to~$\pi$, apply Corollary~\ref{cor: 2.16 analogue} and Corollary~\ref{cor: 2.19 analogue} filling in all remaining values uniquely.  Recall that we need to know that all the diamonds in the frieze are spherical Heronian in order to apply these corollaries.
\end{proof}

While Corollary~\ref{cor: 3.10 analogue} is a purely algebraic result, it has a geometric underpinning:

\begin{corollary}\label{cor: 3.12 analogue} 
    Let~$\zfrieze$ be a spherical Heronian frieze of order~$n$ satisfying condition~\eqref{eq: non-vanishing condition}. Then there exists a unique~$n$-gon~$P$ such that~$\zfrieze=\zfrieze(P)$.
\end{corollary}
\begin{proof}
    Select a traversing path~$\pi$.  The data of~$\pi$ satisfies the requirements to apply Corollary~\ref{prop: cor 2.9 analogue} and construct a unique triangulated spherical polygon~$P$ whose spherical Heronian frieze~$\zfrieze(P)$ agrees with~$\zfrieze$ along~$\pi$.  By Corollary~\ref{cor: 3.10 analogue}, the rest of~$\zfrieze(P)$ frieze is determined uniquely by~$\pi$ and so must agree with~$\zfrieze$.
\end{proof}

\begin{proposition}\label{prop: analogue of 4.10 first part}
    Let $\pi$ be a traversing path.  Denote by~$\zfrieze_-$ a collection of complex numbers assigned to the indices in~$\pi$ which satisfy the appropriate spherical Heron equations. If the values in~$\zfrieze_-$ at the integer nodes of~$\pi$ lie in~${\C-\{0,\frac{4}{\kap}\}}$ then there exists a unique spherical Heronian frieze that agrees with~$\pi$.
\end{proposition} 

\begin{proof}
 Using the ideas in Remark~\ref{rem: 3.9 analogue}, construct a triangulated cycle~$G$ matching~$\pi$.  By Proposition~\ref{prop: cor 2.9 analogue} there exists a unique~$n$-gon~$P$ such that the measurements along~$\pi$ are~$\xS_G(P)$. Then construct the frieze~$\zfrieze(P)$.  It will agree with~$\pi$ and by Corollary~\ref{cor: 3.10 analogue} above, the frieze is unique.
\end{proof}

\begin{remark}
Proposition~\ref{prop: analogue of 4.10 first part} provides formulas for solving the original question in spherical distance geometry, namely the propagation formulas (equations~\eqref{eq: eq for f}-\eqref{eq: eq for s}) applied iteratively.  We will say more about these formulas in Section~\ref{sec: denominators}.
\end{remark}

The geometric connection in Corollary~\ref{cor: 3.12 analogue} and Remark~\ref{rem: glide symmetry} combine to give us this algebraic conclusion:

\begin{theorem}\label{thm: glide symm}
    Let~$\zfrieze$ be a spherical Heronian frieze satisfying condition~\eqref{eq: non-vanishing condition}. Then~$\zfrieze$ is periodic and possesses glide symmetry.
\end{theorem}

\newpage

\section{Rational expressions for spherical Heronian friezes}\label{sec: denominators}

To~propagate outwards from an initial data set in a frieze we apply the propagation formulas iteratively. We will next show that the resulting expressions for the entries of the frieze are structurally simpler than one might expect: they satisfy a form of Laurentness, with predictable denominators. 

Let~$\zpi$ be the initial data associated with a traversing path~$\pi$ in a spherical Heronian frieze~$\zfrieze$. 
    Define
    \begin{equation*}
        D(\pi)=\Set{
        x_{ij}\left(1-\mytfrac{\kap x_{ij}}{4}\right) | (i,j) \text{ an integer node of }\pi \text{ with } 2\le i-j\le n-2
        }.
    \end{equation*}

\begin{proposition}\label{prop: analogue of cor 4.10}
    If all elements of~$D(\pi)$ are nonzero, then each entry of~$\zfrieze$ can be written as a rational function in terms of~$\zpi$ with denominator a monomial in the elements of~$D(\pi)$.
\end{proposition}

The rest of this section is devoted to the proof of Proposition~\ref{prop: analogue of cor 4.10}.

\begin{remark}
    Let~$\zfrieze(P)$ be a spherical Heronian frieze arising from a spherical polygon~$P$. 
    The integer nodes~$(i,j)$ of~$\pi$ with~$2\le i-j\le n-2$ index the squared distances of the diagonals of~$P$.
    An element~$x_{ij}\left(1-\frac{\kap x_{ij}}{4}\right)$ vanishes precisely when the vertices~$A_i$,~$A_j$ of~$P$ are either coincident or antipodal (cf.~Remark~\ref{rem: why conditions}).
\end{remark}

We will need additional terminology, some of which also appeared in~\cite{fomin_heronian_2021}. 

\begin{definition}
Let~$n\ge 4$. Number the vertices of an~$n$-cycle in order~$1,2,\dots$. Consider two distinct diagonals~$\{i,j\}$ and~$\{k,l\}$, with~$i<j$ and~$k<l$. We say that these diagonals \emph{cross} if~$i<k<j<l$ or~$k<i<l<j$.   We say that a diagonal~$\{i,j\}$ crosses a triangle~$(k,l,m)$ if it crosses at least one diagonal side of the triangle.
\end{definition}

\begin{definition}
Fix a triangulation~$G$ of an~$n$-cycle.  For an edge~$\{i,j\}$ or a triple~$(i,j,k)$, let~$\gamma(G,i,j)$ (resp.~$\gamma(G,i,j,k)$) denote the \emph{trimming} of~$G$ with respect to~$\{i,j\}$ (resp.~$(i,j,k)$), defined as the induced subgraph of~$G$ with vertex set
\[
\left\{i,j\right\}\bigcup \Set{l,k | \text{diagonal } \{l,k\} \text{ crosses edge } \{i,j\} \text{ (resp. triangle } (i,j,k))}.
\]
If~$G=\tau(G,i,j)$ (resp.,~$G=\tau(G,i,j,k)$) we say that~$G$ is \emph{trimmed} with respect to~$\{i,j\}$ (resp.~$(i,j,k)$).  See Figure~\ref{fig: trimming example}.
\end{definition}

\begin{remark}\label{rem: trimming}
     Trimming a triangulation always returns a subgraph that is itself a triangulation, typically with fewer edges and vertices than the original one. Loosely speaking, trimming snips off a few extraneous triangles. For example, the trimming of a triangulation~$G$ with respect to a diagonal or triangle that appears in~$G$ is just the diagonal or triangle itself. A  triangulated~$n$-gon trimmed with respect to~$\{i,j\}$ has no diagonals incident to~$i$ or~$j$.  As a tool, trimming is most impactful with respect to a pair of vertices or a triangle that is not a subgraph of the triangulation. 
    
   In particular, trimming a triangulation of an~$n$-cycle produces a triangulation of an~$(n-k)$-cycle ($k\ge 0$) by \emph{thin} triangles, each of which shares a side with the~${(n-k)}$-cycle. Performing such a trimming can be used strategically to remove distracting pieces of the polygon whose measurements are irrelevant for computation.

    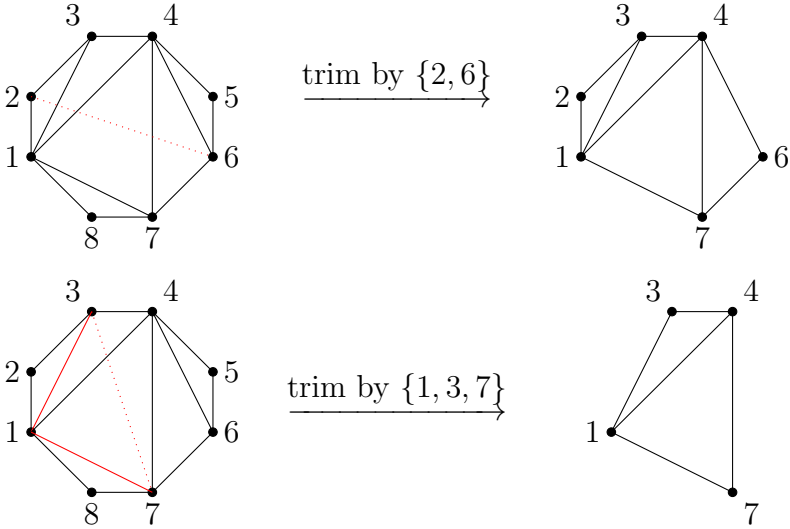
\begin{figure}[htbp!]
        \centering
        \begin{tabular}[c]{ccc}
         \begin{tikzpicture}[scale=0.8]
\coordinate[label=left:$1$] (1) at (-1.5,1);
\filldraw [black] (1) circle (2pt) ;
\coordinate[label=left:$2$] (2) at (-1.5,2);
\filldraw [black] (2) circle (2pt) ;
\coordinate [label=above left:$3$](3) at (-.5,3);
\filldraw [black] (3) circle (2pt) ;
\coordinate [label=above right:$4$](4) at (.5,3);
\filldraw [black] (4) circle (2pt) ;
\coordinate [label=right:$5$](5) at (1.5,2);
\filldraw [black] (5) circle (2pt) ;
\coordinate [label=right:$6$](6) at (1.5,1);
\filldraw [black] (6) circle (2pt) ;
\coordinate [label=below:$7$](7) at (.5,0);
\filldraw [black] (7) circle (2pt) ;
\coordinate [label=below:$8$](8) at (-.5,0);
\filldraw [black] (8) circle (2pt) ;

\draw (1) to (2) to (3) to (4) to (5) to (6) to (7) to (8) to (1);
\draw (1) to (3);
\draw (1) to (4);
\draw (4) to (7);
\draw (4) to (6);
\draw (1) to (7);
\draw[dotted, red] (2) to (6);
        \end{tikzpicture} & \begin{tikzpicture}
        \coordinate [label=below:{}] (nothing) at (0,0);
            \coordinate [label=below:$\underrightarrow{\text{trim by }\{2,6\}}$] (arrow) at (0,2.5);
        \end{tikzpicture} & \begin{tikzpicture}[scale=0.8]
\coordinate [label=left:$1$] (1) at (-1.5,1);
\filldraw [black] (1) circle (2pt) ;

\coordinate [label=left:$2$](2) at (-1.5,2);
\filldraw [black] (2) circle (2pt) ;
\coordinate [label=above left:$3$](3) at (-.5,3);
\filldraw [black] (3) circle (2pt) ;
\coordinate [label=above right:$4$](4) at (.5,3);
\filldraw [black] (4) circle (2pt) ;

\coordinate [label=right:$6$](6) at (1.5,1);
\filldraw [black] (6) circle (2pt) ;
\coordinate [label=below:$7$](7) at (.5,0);
\filldraw [black] (7) circle (2pt) ;

\draw (1) to (2) to (3) to (4) to (6) to (7) to (1);
\draw (1) to (3);
\draw (1) to (4);
\draw (4) to (7);
        \end{tikzpicture} \\
         \begin{tikzpicture}[scale=0.8]
\coordinate[label=left:$1$] (1) at (-1.5,1);
\filldraw [black] (1) circle (2pt) ;
\coordinate[label=left:$2$] (2) at (-1.5,2);
\filldraw [black] (2) circle (2pt) ;
\coordinate [label=above left:$3$](3) at (-.5,3);
\filldraw [black] (3) circle (2pt) ;
\coordinate [label=above right:$4$](4) at (.5,3);
\filldraw [black] (4) circle (2pt) ;
\coordinate [label=right:$5$](5) at (1.5,2);
\filldraw [black] (5) circle (2pt) ;
\coordinate [label=right:$6$](6) at (1.5,1);
\filldraw [black] (6) circle (2pt) ;
\coordinate [label=below:$7$](7) at (.5,0);
\filldraw [black] (7) circle (2pt) ;
\coordinate [label=below:$8$](8) at (-.5,0);
\filldraw [black] (8) circle (2pt) ;

\draw (1) to (2) to (3) to (4) to (5) to (6) to (7) to (8) to (1);
\draw[red] (1) to (3);
\draw (1) to (4);
\draw (4) to (7);
\draw (4) to (6);
\draw[red] (1) to (7);
\draw[dotted,red] (3) to (7);
        \end{tikzpicture} & \begin{tikzpicture}
        \coordinate [label=below:{}] (nothing) at (0,0);
            \coordinate [label=below:$\underrightarrow{\text{trim by }\{1,3,7\}}$] (arrow) at (0,2);
        \end{tikzpicture}  &\begin{tikzpicture}[scale=0.8]
\coordinate[label=left:$1$] (1) at (-1.5,1);
\filldraw [black] (1) circle (2pt) ;

\coordinate[label=above left:$3$] (3) at (-.5,3);
\filldraw [black] (3) circle (2pt) ;
\coordinate [label=above right:$4$](4) at (.5,3);
\filldraw [black] (4) circle (2pt) ;

\coordinate [label=below right:$7$](7) at (.5,0);
\filldraw [black] (7) circle (2pt) ;

\draw (1) to  (3) to (4)  to (7) to (1);
\draw (1) to (4);
        \end{tikzpicture}
  
        \end{tabular}

        \caption{The trimmings of triangulation~$G$ with respect to diagonal~$\{2,6\}$ and triangle~$\{1,3,7\}$ respectively.}
        \label{fig: trimming example}
    \end{figure}
\end{remark}

  \begin{definition}
        Let~$G$ be a triangulated~$n$-cycle. We call~$G$ a \emph{thin} triangulation if each of its triangles shares a side with the~$n$-cycle (we will refer to such triangles as \emph{thin}).   Notably, any traversing path in a spherical Heronian frieze corresponds to a triangulation by thin triangles.
    \end{definition}

    \begin{definition}
        Let~$\{i,j\}$ be a diagonal in a triangulation~$G$.  There is a unique diagonal~$\{k,l\}$  such that 
        \[
        G'=\mu_{ij}(G)=G\cup\{k,l\}\setminus\{i,j\}
        \] is still a triangulation. We refer to this operation as \emph{flipping}~$\{i,j\}$ to~$\{k,l\}$.
    \end{definition}

\begin{definition}
    For~$(P,G)$  a triangulated polygon, let~$D(G)$ denote the set of diagonals appearing in~$G$. Define 
    \begin{equation*}
        \xx_{D(G)}(P)=\left\{x_{ij}\left(1-\mytfrac{\kap x_{ij}}{4}\right) \ : \ \{i,j\}\in D(G)\right\}.
    \end{equation*}  
    Note that the elements of~$\xx_{D(G)}$ are the type of factors that appear in the denominators of the equations for propagation in the frieze (cf.~Proposition~\ref{prop: analogue of 2.14}).
\end{definition}

Towards a proof of Proposition~\ref{prop: analogue of cor 4.10}, we establish the following lemma. 
    \begin{lemma}\label{lem: analogue of lemma 4.7} Let~$G$ be a triangulated cycle and~$e\in D(G)$. Flip~$e$ to get edge~$f$ and~$G'=\mu_{e}(G)$.
    Pick~$\{i,j\}\not\in D(G)$ so that~$G$ is trimmed with respect to~$\{i,j\}$, but~$G'$ is not.  Trim~$G'$ with respect to~$\{i,j\}$ to get~$G''=\tau(G',i,j)$.
 Assume that the measurement~$x_{ij}$ can be written as a rational function in~$\xS_{G''}(P)$ whose denominator is a monomial in the elements of~$\xx_{D(G'')}(P).$  Then~$x_{ij}$ can be written as a rational function in~$\xS_G(P)$ whose denominator is a monomial in the elements of~$\xx_{D(G)}(P)$.  The same statement holds for~$(i,j,k)$ and~$\skapijk$ in place of~$\{i,j\}$ and~$x_{ij}$.
    \end{lemma}
    \begin{proof}
        For both results we can apply the same proof method as~\cite{fomin_heronian_2021}.  Using the assumption, write~${x_{ij}=Q/S}$ with numerator~$Q$ a polynomial in~$\xS_{G''}(P)$ and denominator~$S$ a monomial in~$\xx_{D(G'')}(P)$. 
        
        We are already close to the goal, since~$D(G'')\subset D(G)$ and~$\xS_{G''}(P)$ only contains a few measurements not present in~$\xS_G(P)$.  In particular,~$\xS_{G''}(P)\setminus \xS_G(P)$ contains~$x_f$ and two measurements of the form~$\skap_{fgh}$, for two triangles of~$G$ with~$f$ as a side.  Use equations~\eqref{eq: eq for f}-\eqref{eq: eq for s} to write these three measurements as rational functions in~$\xS_G(P)$ each with denominator a scalar multiple of~$x_e\left(1-\frac{\kap x_{e}}{4}\right)\in \xx_{D(G)}(P)$.  
        
        Substitute the new values back into~$Q$ (as needed) to get~$Q'$.  The expression~$Q'$ is a rational function in~$\xS_G(P)$ with denominator a monomial in~$x_e\left(1-\frac{\kap x_e}{4}\right)$.  Substituting back into~$x_{ij}$ we have now written~$x_{ij}=Q'/S$, a rational function in~$\xS_G(P)$ whose denominator is a monomial in
        \[\left\{x_e\left(1-\mytfrac{\kap x_e}{4}\right)\right\}\cup \xx_{D(G)''}(P)=\xx_{D(G)}(P).
        \]
        The result for~$(i,j,k)$ and~$\skapijk$ follows similarly.
    \end{proof}

    \begin{proposition}\label{prop: analogue of props 4.8}  Let~$G$ be a triangulated~$n$-gon trimmed with respect to diagonal~$\{i,j\}$.  Then~$x_{ij}$ can be expressed as a rational function in~$\xS_G(P)$ whose denominator is a monomial in the elements of~$\xx_{D(G)}(P)$.
    \end{proposition}
    \begin{proof}
        As in the analogous proof in~\cite{fomin_heronian_2021} we induct on~$n$.  When~${n=4}$, use our spherical analogue of Bretschneider's formula~\eqref{eq: sphericalbret}. Let~$n>4$. There are no diagonals in~$G$ incident to~$i$ (see Remark~\ref{rem: trimming}), so  edge~${e=\{i-1,i+1\}}$ (working modulo~$n$) must be a diagonal of~$G$. Flip the diagonal~$e$ to get a new triangulation~$G'$. Since~$G'$ contains a diagonal incident to~$i$, it is not trimmed with respect to~$\{i,j\}$.  Compute~$G''$, the trimming of~$G'$ with respect to~$e$. Since~$G''$ has fewer vertices than~$G$, apply the induction hypothesis. We get the conditions of Lemma~\ref{lem: analogue of lemma 4.7}, which then proves the claim.
    \end{proof}

    \begin{proposition}\label{prop: analogue of prop 4.9}
    Let~$G$ be a triangulated~$n$-gon that is trimmed with respect to a triple~$(i,j,k)$.  Then~$\skapijk$ can be expressed as a rational function in~$\xS_G(P)$ whose denominator is a monomial in the elements of~$\xx_{D(G)}(P)$.
        
    \end{proposition}

    \begin{proof}
        Again, we induct on~$n$, following the general outline of the proof of the analogous result in \cite{fomin_heronian_2021}. When~$n=4$, the result is due to equation~\eqref{eq: eq for r} or~\eqref{eq: eq for s} (both of whom return rational functions in~$\xS_G(P)$ with denominator a monomial in~$\xx_D(G)(P)$).

        Let~$n=5$.  All triangulations are the same up to cyclic renumbering of vertices, so without loss of generality assume triangulation~$G$ has diagonals~$\{1,3\}$ and~$\{1,4\}$.  These diagonals must cut across the triangle~$(i,j,k)$, so~$2$ and~$5$ must be two of the components of the triple.  Using Lemma~\ref{lem: analogue of lemma 4.7} with~$e=\{1,3\}$ and~$(i,j,k)=(1,2,5)$ (resp.,~$(2,4,5)$), we can see that~$\skap_{125}$ (resp.~$\skap_{245}$) can be written as a rational function in~$\xS_G(P)$ with denominator a monomial in~$\xx_{D(G)}(P)$.  The case for~$\skap_{235}$ is similar.

        Consider the case where~$n=6$ and~$G$ is the triangulation with diagonals~$\{1,3\}$, $\{3,5\}$, and~$\{1,5\}$ (see Figure~\ref{fig: hexagon}).

    \begin{figure}[htbp!]
            \centering
                    \begin{tabular}[c]{ccccccc}

\begin{tikzpicture}
[scale=0.4]
    \draw (-1,4) node[above left] (1){$1$} -- (1,4) node[above right](2) {$2$} -- (2,2) node[right](3) {$3$} -- (1,0) node[below right] (4){$4$} -- (-1,0) node[below left](5) {$5$} -- (-2,2) node[left] (6){$6$}-- (-1,4);
    \draw (-1,4)--(2,2)--(-1,0)--(-1,4);

    \circleat{-1}{4};
    \circleat{1}{4};
    \circleat{2}{2};
    \circleat{1}{0};
    \circleat{-1}{0};
    \circleat{-2}{2};
\end{tikzpicture} 
&
\begin{tikzpicture}
        \coordinate [label=below:{}] (nothing) at (0,0);
            \coordinate [label=below:$\rightarrow$] (arrow) at (0,1.2);
        \end{tikzpicture}
        &
\begin{tikzpicture}
[scale=0.4]
    \draw (-1,4) node[above left] (1){$1$} -- (1,4) node[above right](2) {$2$} -- (2,2) node[right](3) {$3$} -- (1,0) node[below right] (4){$4$} -- (-1,0) node[below left](5) {$5$} -- (-2,2) node[left] (6){$6$}-- (-1,4);
    \draw (-1,4)--(2,2)--(-1,0);
    \draw (2,2) -- (-2,2);

    \circleat{-1}{4};
    \circleat{1}{4};
    \circleat{2}{2};
    \circleat{1}{0};
    \circleat{-1}{0};
    \circleat{-2}{2};
\end{tikzpicture}
&
\begin{tikzpicture}
        \coordinate [label=below:{}] (nothing) at (0,0);
            \coordinate [label=below:$\rightarrow$] (arrow) at (0,1.2);
        \end{tikzpicture}
        &
\begin{tikzpicture}
[scale=0.4]
    \draw (-1,4) node[above left] (1){$1$} -- (1,4) node[above right](2) {$2$} -- (2,2) node[right](3) {$3$} -- (1,0) node[below right] (4){$4$} -- (-1,0) node[below left](5) {$5$} -- (-2,2) node[left] (6){$6$}-- (-1,4);
    \draw (1,4) -- (-2,2) -- (1,0);
    \draw (2,2) -- (-2,2);

    \circleat{-1}{4};
    \circleat{1}{4};
    \circleat{2}{2};
    \circleat{1}{0};
    \circleat{-1}{0};
    \circleat{-2}{2};
\end{tikzpicture} 
&
\begin{tikzpicture}
        \coordinate [label=below:{}] (nothing) at (0,0);
            \coordinate [label=below:$\rightarrow$] (arrow) at (0,1.2);
        \end{tikzpicture}
        &
\begin{tikzpicture}
[scale=0.4]
    \draw (-1,4) node[above left] (1){$1$} -- (1,4) node[above right](2) {$2$} -- (2,2) node[right](3) {$3$} -- (1,0) node[below right] (4){$4$} -- (-1,0) node[below left](5) {$5$} -- (-2,2) node[left] (6){$6$}-- (-1,4);
    \draw (1,4) -- (-2,2) -- (1,0);
    \draw (1,4)--(1,0);

    \circleat{-1}{4};
    \circleat{1}{4};
    \circleat{2}{2};
    \circleat{1}{0};
    \circleat{-1}{0};
    \circleat{-2}{2};
\end{tikzpicture}

\end{tabular}
\caption{Using flips to derive an identity for~$\skap_{246}$.}
            \label{fig: hexagon}
        \end{figure}
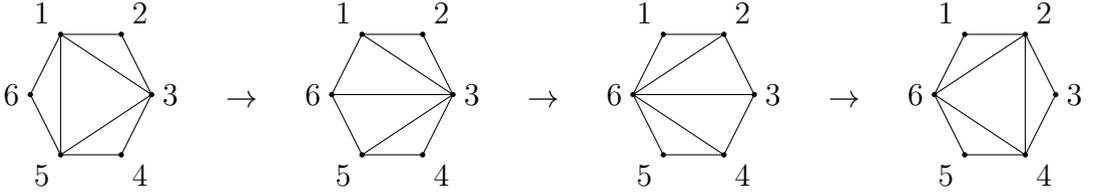

  \pagebreak[3]
  
         While~$G$ is not trimmed with respect to any of the  triangles
        \begin{equation*}
            (2,3,4), (4,5,6), (1,2,6), (3,4,6), (2,3,6), (1,3,6), (3,5,6),
        \end{equation*}
        each is inside of a sub-quadrilateral or sub-pentagon of~$G$ that is triangulated by diagonals of~$G$, and so by induction the claim holds. For example,~$\{1,2,3,4,5\}$ is triangulated by the diagonals of~$G$ and is trimmed with respect to~$(2,3,4)$.  Each triangle is shown below with the sub-polygon that it is trimmed with respect to:
        \begin{center}

        \begin{tabular}{c|c}
        triangle & trimmed with respect to\\
        \hline
            $(1,3,6)$, $(3,5,6)$ & $\{1,3,5,6\}$\\
            \hline
            $(2,3,4)$ & $\{1,2,3,4,5\}$ \\
            \hline
            $(4,5,6)$, $(3,4,6)$ & $\{1,3,4,5,6\}$\\
            \hline
            $(1,2,6)$, $(2,3,6)$ & $\{1,2,3,5,6\}$
        \end{tabular}
                    
        \end{center}
        Additionally, note that since the sub-quadrilateral~$\{1,3,5,6\}$ is trimmed with respect to~$\{3,6\}$ by Proposition~\ref{prop: analogue of props 4.8}, the side length~$x_{36}$ can also be written in the desired form.

        We have now handled all triangles that~$G$ is trimmed with respect to except the triangle~$(2,4,6)$.  In this case, starting with~$G$, we flip edges until we arrive at a triangulation that contains the triangle~$(2,4,6)$ (see the sequence of flips in Figure~\ref{fig: hexagon}). With each flip we use equation~\eqref{eq:def1sph p+qanalogue} to write a new identity. Flipping edge~$\{1,5\}$ we get
        \begin{equation}\label{eq: id 1}
            \skap_{135}+\skap_{156}=\skap_{136}+\skap_{356}+\mytfrac{x_{35}\skap_{156}+x^\kap_{56}\skap_{135}-x_{15}\skap_{356}}{2}\, .
        \end{equation}
        Flipping edge~$\{1,3\}$ we get
        \begin{equation}\label{eq: id 2}
            \skap_{123}+\skap_{136}=\skap_{126}+\skap_{236}+\mytfrac{x_{16}\skap_{123}+x_{12}\skap_{136}-x_{13}\skap_{126}}{2}\, .
        \end{equation}
        Flipping edge~$\{3,5\}$ we get
        \begin{equation}\label{eq: id 3}
            \skap_{356}+\skap_{345}=\skap_{456}+\skap_{346}+\mytfrac{x_{45}\skap_{356}+x_{56}\skap_{345}-x_{35}\skap_{456}}{2}\, .
        \end{equation}
        Finally, flipping edge~$\{3,6\}$ we get
        \begin{equation}\label{eq: id 4}
            \skap_{346}+\skap_{236}=\skap_{234}+\skap_{246}+\mytfrac{x_{23}\skap_{346}+x_{34}\skap_{236}-x_{35}\skap_{456}}{2}\, .
        \end{equation}
        Combining equations~\eqref{eq: id 1}-\eqref{eq: id 4} we achieve the identity
\begin{multline*}\label{eq: id final}
    \skap_{156}+\skap_{123}+\skap_{345}+\skap_{135}=\\
    \skap_{246}+\skap_{234}+\skap_{126}+\skap_{456}+\mytfrac{x_{35}(\skap_{156}-\skap_{456})+x_{56}(\skap_{135}+\skap_{345})+(x_{45}-x_{15})\skap_{356}}{2}\\
+\mytfrac{x_{16}\skap_{123}+x_{12}\skap_{136}-x_{13}\skap_{126}+x_{23}\skap_{346}+x_{34}\skap_{236}-x_{36}\skap_{234}}{2}\, .
\end{multline*}
        Since we have shown that every component of this identity (with the exception of~$\skap_{246}$) can be written as a rational function in~$\xS_G(P)$ with denominator a monomial in~$\xx_{D(G)}(P)$, this must be the case for~$\skap_{246}$ as well. (A much cleaner identity holds in the case~$\kap=0$.)

        The remaining inductive cases will be handled generally.  Suppose triangulation~$G$ has~$e\in D(G)$ incident to vertex~$i$.  Since~$G$ is trimmed with respect~$(i,j,k)$, we can assume~$e$ crosses~$\{j,k\}$.  Similarly, were~$G$ to have a diagonal incident to~$j$, it would have to cross~$\{i,k\}$.  But such a diagonal would also cross~$e$.  Therefore~$G$ has no diagonal incident to~$j$ and~$f=\{j-1,j+1\}\in D(G).$  The triangulation~$G'=\mu_f(G)$ is not trimmed with respect to~$(i,j,k)$ so by Lemma~\ref{lem: analogue of lemma 4.7} and the induction hypothesis,~$\skap_{ijk}$ can be written as a rational function in~$\xS_G(P)$ with denominator a monomial in~$\xx_{D(G)}(P)$.

        Consider the case where triangulation~$G$ has no diagonals incident to~$i$,~$j$, or~$k$.  Then~$D(G)$ must include diagonals~$\{i-1,i+1\}$,~$\{j-1,j+1\}$, and~$\{k-1,k+1\}$ (modulo~$n$).  Let~$G'$ be the triangulation achieved by flipping~$\{i-1,i+1\}$ to~$\{i,v\}$ for the appropriate~$v\in\{1,\dots,n\}$.  If~$\{i,v\}$ does not cross~$\{j,k\}$, then Lemma~\ref{lem: analogue of lemma 4.7} and the induction hypothesis apply.  
        
        If~$\{i,v\}$ crosses~$\{j,k\}$, then so do~$\{i-1,v\},\{i+1,v\}$ which are edges of the original triangulation.  First assume that
        \begin{equation}\label{eq: hex case} 
        \{\{i-1,v\},\{i+1,v\}\}\ne \{\{j-1,j+1\},\{k-1,k+1\}\}\, .
                \end{equation}
        Then flipping one of~$\{j-1,j+1\},\{k-1,k+1\}$ will result in an untrimmed triangulation (with respect to~$(i,j,k)$), trim it down and then apply Lemma~\ref{lem: analogue of lemma 4.7} and the induction hypothesis to the smaller polygon.  In the case that~\eqref{eq: hex case} holds, we must have~$n=6$, with the triangulation given by~$\{3,4,5\}$ and~$(i,j,k)=(2,4,6)$, which we handled above.
    \end{proof}

    \begin{proof}[Proof of Proposition~\ref{prop: analogue of cor 4.10}]
Proposition~\ref{prop: analogue of 4.10 first part} proves that~$\zpi$, the collection of numbers associated with a traversing path~$\pi$, can be uniquely extended to a spherical Heronian frieze~$\zfrieze$ via the construction of a triangulated polygon~$(P,G)$ associated to~$\pi$.  Propositions~4.8 and~4.9 above show that the elements of~$\zfrieze$ are all rational functions with denominators coming from~$D(\pi)$.  
    \end{proof}

\newpage

\section{Spherical Cayley-Menger friezes}\label{sec: cayley menger}

We now present our second solution to the original spherical distance geometry question. In this solution, our initial data set consists entirely of squared distances~$x_{ij}$. It involves the squared distances associated to a particular triangulation of a polygon, along with a few extra squared distances (the ``bracing'' edges). The measurements~$\skap$ are not used at all.

Proposition~\ref{prop: gencaymeng} implies a coplanarity test for~$n+1$ points in~$n$-dimensional Euclidean space:  if the points~$A_1,\dots,A_{n+1}\in\R^n$ lie on a hyperplane, then the Cayley-Menger determinant~$M_{n+1}^0(x_{12},\dots,x_{n,n+1})$ vanishes.  We will need a similar test for cosphericity, this time utilizing a \emph{spherical Cayley-Menger determinant}. The papers  \cite{audet_determinants_2011} and \cite{blumenthal_distribution_1943} give different developments of the spherical Cayley-Menger determinant. We most closely follow the development in \cite{tao_spherical_2019}, cf.~Theorem~\ref{thm: scm}.

Lemma~\ref{lem: cosine} below will allow us to express scalar products in terms of the squared distances.
We will use the following shorthand: for a point~$A\in\bfS$, we will write~$\overrightarrow{A}$ instead of~$\overrightarrow{OA}$. 
\begin{lemma}\label{lem: cosine}
    For~$A, B\in\bfS$, denote~$x_{A,B}=x(A,B)$. Then
\begin{equation}\label{eq: dot product}
        \langle \overrightarrow{A},\overrightarrow{B}\rangle =\mytfrac{1}{\kap}\left(1-\mytfrac{\kap}{2} x_{A,B}\right).
    \end{equation}
\end{lemma}
\begin{proof}
    Using the vector form of the law of cosines, we have
    \begin{equation}\label{eq: lawofcosines}
        \langle \overrightarrow{B}-\overrightarrow{A},\overrightarrow{B}-\overrightarrow{A}\rangle =\langle\overrightarrow{B},\overrightarrow{B}\rangle+\langle\overrightarrow{A},\overrightarrow{A}\rangle-2\langle\overrightarrow{A},\overrightarrow{B}\rangle.
    \end{equation}
    Since~$A,B\in\bfS$, we have 
    \begin{align*}
        \langle\overrightarrow{A},\overrightarrow{A}\rangle=\langle\overrightarrow{B},\overrightarrow{B}\rangle&=R^2=\mytfrac{1}{\kap}\, ,\\
        \langle \overrightarrow{B}-\overrightarrow{A},\overrightarrow{B}-\overrightarrow{A}\rangle&=x_{A,B}\, .
    \end{align*}
    So equation~\eqref{eq: lawofcosines} becomes~$
        x_{A,B}=\mytfrac{2}{\kap}-2\langle \overrightarrow{A},\overrightarrow{B}\rangle,$
    which rearranges to~\eqref{eq: dot product}.
\end{proof}

\begin{theorem}[cf.~{\cite{audet_determinants_2011, blumenthal_distribution_1943, tao_spherical_2019}}]\label{thm: scm} 
Let~$m\ge 4$ and~$A_1,\dots, A_m\in\bfS\subset\C^3$. Denote~$x_{ij}=x(A_i,A_j)$.
Then the spherical Cayley-Menger determinant vanishes:
\begin{equation*}
\scmm(x_{12},\dots,x_{m-1,m})=\begin{bmatrix}
    \frac{\kap}{2} & 1  & \cdots & 1\\
    1 & 0   & \cdots & x_{1m} \\
    \vdots & \vdots  & \ddots & \vdots\\
    1 & x_{m1}  & \cdots & 0
    \end{bmatrix}=0\, .
\end{equation*}

\end{theorem}
\begin{remark}\label{rem: explaining zero} The notation~$\scmm$ reflects that this is a computation associated to a collection of~$m$ points on a sphere with curvature~$\kap$. 
\end{remark}

\begin{remark}\label{rem: dim three cay meng} The spherical Cayley-Menger determinant~$\scmthree$ associated to a triple of points on~$\bfS$ is closely related to the spherical Heron formula (cf.~Proposition~\ref{prop: sph heron}):
\begin{multline*}
    -\scmthree (x_{12},x_{13},x_{23})  \\= -x_{12}^2-x_{13}^2-x_{23}^2+2x_{12}x_{13}+2x_{12}x_{23}+2x_{13}x_{23}-\kap x_{12}x_{13}x_{23}\\  =\hkap(x_{12},x_{13},x_{23}) =(\skap(A_1,A_2,A_3))^2\, .
\end{multline*}
\end{remark}

\begin{proof}[Proof of Theorem~\ref{thm: scm}]
Let~$M$ be the matrix with column vectors~$A_i$. Its  rank is at most~3. The product~$M^TM$ is the~$m\times m$ \emph{Gram matrix} (cf.~Remark~\ref{rem: gram det})
\begin{equation}\label{eq: gram mat}
M^TM= \begin{bmatrix}
    \langle A_1,A_1\rangle & \langle A_1,A_2\rangle & \cdots & \langle A_1, A_m\rangle\\
    \langle A_2, A_1\rangle & \langle A_2, A_2\rangle & \cdots & \langle A_2, A_m\rangle\\
    \vdots & \cdots & \ddots & \vdots\\
    \langle A_m, A_1\rangle & \langle A_m,A_2\rangle & \cdots & \langle A_m,A_m\rangle
\end{bmatrix}\, ,
\end{equation}
which has rank at most~3. Since~$m\ge 4$, the determinant of~$M^TM$ must vanish.  

We now use the condition that our points lie on a sphere with curvature~$\kap$. Applying the identity~\eqref{eq: dot product} to~\eqref{eq: gram mat}, we get:
\begingroup
\addtolength{\jot}{.5em}
\renewcommand*{\arraystretch}{1.25}
\begin{align*}\label{eq: transformed gram}
    0&=\det\begin{bmatrix}
        \frac{1}{\kap} & \frac{1}{\kap}\left(1-\frac{\kap}{2}x_{12}\right) & \cdots & \frac{1}{\kap}\left(1-\frac{\kap}{2}x_{1m}\right)\\
        \frac{1}{\kap}\left(1-\frac{\kap}{2}x_{12}\right) & \frac{1}{\kap} & \cdots & \frac{1}{\kap}\left(1-\frac{\kap}{2}x_{2m}\right)\\
        \vdots & \cdots & \ddots & \vdots\\
        \frac{1}{\kap}\left(1-\frac{\kap}{2}x_{1m}\right) & \frac{1}{\kap}\left(1-\frac{\kap}{2}x_{2m}\right) & \cdots & \frac{1}{\kap}
    \end{bmatrix}\, ,\\
&=
    \left(-\mytfrac{1}{\kap}\right)^m\det\begin{bmatrix}
        -1 & \frac{\kap}{2}x_{12}-1 & \cdots & \frac{\kap}{2}x_{1m}-1\\
        \frac{\kap}{2}x_{12}-1 & -1 & \cdots & \frac{\kap}{2}x_{2m}-1\\
        \vdots & \cdots & \ddots & \vdots\\
        \frac{\kap}{2}x_{1m}-1 & \frac{\kap}{2}x_{2m}-1 & \cdots & -1
    \end{bmatrix}\, ,\\
    &=\left(-\mytfrac{1}{2}\right)^m\left(\mytfrac{\kap}{2}\right)^{-(m+1)}\det\begin{bmatrix}
    \mytfrac{\kap}{2} & 1 & 1 & \cdots & 1\\
      0 &  -1 & \frac{\kap}{2}x_{12}-1 & \cdots & \frac{\kap}{2}x_{1m}-1\\
      0 &  \frac{\kap}{2}x_{12}-1 & -1 & \cdots & \frac{\kap}{2}x_{2m}-1\\
       \vdots & \vdots & \cdots & \ddots & \vdots\\
       0 & \frac{\kap}{2}x_{1m}-1  &\frac{\kap}{2}x_{2m}-1 & \cdots & -1
    \end{bmatrix}\, ,\\
    &=\left(-\mytfrac{1}{2}\right)^m\left(\mytfrac{\kap}{2}\right)^{-(m+1)}\det\begin{bmatrix}
    \frac{\kap}{2} & 1 & 1 & \cdots & 1\\
      \frac{\kap}{2} &  0 & \frac{\kap}{2}x_{12} & \cdots & \frac{\kap}{2}x_{1m}\\
      \frac{\kap}{2} &  \frac{\kap}{2}x_{12} & 0 & \cdots & \frac{\kap}{2}x_{2m}\\
       \vdots & \vdots & \cdots & \ddots & \vdots\\
       \frac{\kap}{2} & \frac{\kap}{2}x_{1m}  &\frac{\kap}{2}x_{2m} & \cdots & 0
    \end{bmatrix}\,  ,\\
    &=\left(-\mytfrac{1}{2}\right)^m\left(\mytfrac{2}{\kap}\right)\det\begin{bmatrix}
    \frac{\kap}{2} & 1 & 1 & \cdots & 1\\
      1 &  0 & x_{12} & \cdots & x_{1m}\\
      1 &  x_{12} & 0 & \cdots & x_{2m}\\
       \vdots & \vdots & \cdots & \ddots & \vdots\\
       1 & x_{1m}  &x_{2m} & \cdots & 0
    \end{bmatrix}\, ,
\end{align*}
\endgroup
which implies the claim.
\end{proof}

\begin{remark}\label{rem: sph heron implies scm is zero}
    Let~$\kap\in\C$. Let~$(a,b,c,d,e,f,p,q,r,s)\in\C^{10}$ be a spherical Heronian diamond (cf.~Definition~\ref{def:original definition}) such that~$e,f\not\in\left\{0,\frac{4}{\kap}\right\}$.
    There exists a unique quadrilateral~$A_1A_2A_3A_4$ on~$\bfS$ with the measurements~$(a,b,c,d,e,f,p,q,r,s)$ by Proposition~\ref{prop: cor 2.9 analogue}.
    Ergo, the spherical Cayley-Menger determinant~$
    \scmfour (a,b,c,d,e,f)$ should vanish.
    
    To~verify this directly,  write each of the seven equations~\eqref{eq:def1sphp}-\eqref{eq: def1sph erminuss} in the form $F(a,b,c,d,e,f,p,q,r,s)=0$ where~$F$ is a polynomial.
    Let~$I$ be the ideal generated by these polynomials.  One can check, with standard software, that
    \[
    \scmfour (a,b,c,d,e,f)e(e\kap-4)\]
    lies in the ideal~$I$.  Since~$e\not\in\{ 0,\frac{4}{\kap}\}$ we conclude that~$\scmfour (a,b,c,d,e,f)=0$.
\end{remark}

\begin{remark}
    In the limit~$\kap\to 0$ (moving towards Euclidean space), the entry in the top left corner of the spherical Cayley-Menger determinant goes to 0, recovering the Euclidean Cayley-Menger determinant.
\end{remark}

Some of the following definitions generalize those found in \cite[Section~5]{fomin_heronian_2021}.

\begin{definition}\label{def: sph cay meng diamond}
    A \emph{spherical Cayley-Menger diamond} is a 6-tuple~$(a,b,c,d,e,f)$ of complex numbers such that
    \[
    \scmfour (a,b,c,d,e,f)=0.\]
    We organize these numbers in a diamond shape as shown below.
    \begin{figure}[htbp!]
        \centering
        \begin{tikzpicture}
            [scale=0.7,every node/.style={rectangle,fill=white}]

        \begin{scope}[on background layer]
        
        \draw[dashed] (-2,4) node {$d$}
        --(2,0);
        \draw[dashed] (-2,0) node {$b$}
        --(2,4);
    \end{scope}
        \draw[very thick]
        (0,0) node {$c$}
        -- (-2,2) node {$e$}
        -- (0,4) node {$a$}
        -- (2,2) node {$f$}
        -- (0,0);
        \end{tikzpicture}
        \label{fig: sph cayley-menger diamond}
    \end{figure}
\end{definition}

We note that every spherical quadrilateral gives rise to a spherical Cayley-Menger diamond.

\begin{remark} 
    The definition of~$\scmfour (a,b,c,d,e,f)$ implies that we can swap rows and columns to see that if~$(a,b,c,d,e,f)$ is a spherical Cayley-Menger diamond, then the tuples~$(c,d,a,b,e,f)$, $(a,d,c,b,f,e)$, $(f,d,e,b,a,c)$, and so on, are spherical Cayley-Menger diamonds as well.
\end{remark}

\begin{definition} \label{def: sph cayley-menger frieze index set}
    
    For~$n\ge 4$, define the indexing set
    \begin{equation*}
        I_n^{SCM}=
        \left\{(i,j)\in\Z^2 \right\}_{0\le j-i\le n}\cup \left\{\left(i+\mytfrac{1}{2},\neline\right) \right\}_{i\in\Z}\cup\left\{\left(\seline,j+\mytfrac{1}{2}\right) \right\}_{j\in\Z}\, ,
    \end{equation*}
    a subset of the indexing set for a spherical Heronian frieze (cf.~Definition~\ref{def: indexing set}).
    As before, the coordinates that use symbols~$\neline$ or~$\seline$ represent entire lines. We continue to call elements with two numerical coordinates \emph{nodes}.  When~$1\le i-j\le n-1$, the node~$(i,j)$ is \emph{interior}. We visually represent the set~$I_n^{SCM}$ on the usual~$\R^2$ coordinate plane rotated~$\pi/4$ clockwise (see Figure~\ref{fig: cay-meng index set}).
    \end{definition}

    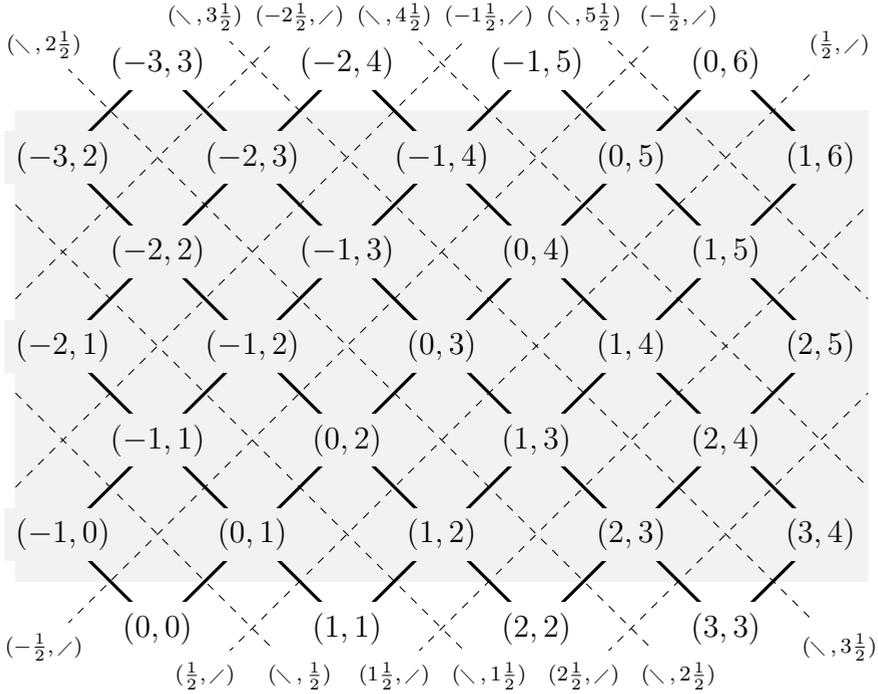
\begin{figure}[htbp!]
        \centering
        \begin{tikzpicture}
        [scale=2.5,gnode/.style={rectangle,fill=black!5,inner sep=1pt}, wnode/.style={rectangle,fill=white,inner sep=1pt}]

        \begin{scope}[on background layer] 
        
\filldraw [black!5] (-0.75,0.25) rectangle (3.75,2.75);

\draw[dashed] (-0.75,1.75) 
        --(0.75,3.25) node[wnode] {\tiny$(-2\frac{1}{2},\neline)$}; 

\draw[dashed] (-0.75,.75) 
        --(1.75,3.25) node[wnode] {\tiny$(-1\frac{1}{2},\neline)$};

        \draw[dashed] (-0.6,-0.1) node[wnode] {\tiny$(-\frac{1}{2},\neline)$}
        --(2.75,3.25) node[wnode] {\tiny$(-\frac{1}{2},\neline)$};

\draw[dashed] (0.25,-0.25) node[wnode] {\tiny$(\frac{1}{2},\neline)$}
        --(3.6,3.1) node[wnode] {\tiny$(\frac{1}{2},\neline)$};

\draw[dashed] (1.25,-0.25) node[wnode] {\tiny$(1\frac{1}{2},\neline)$}
        --(3.75,2.25) ;
        
\draw[dashed] (2.25,-0.25) node[wnode] {\tiny$(2\frac{1}{2},\neline)$}
        --(3.75,1.25) ;  

\draw[dashed] (-0.75,1.25)  --(0.75,-0.25) node[wnode] {\tiny$(\seline,\frac{1}{2})$};

\draw[dashed] (-0.75,2.25)  --(1.75,-0.25) node[wnode] {\tiny$(\seline,1\frac{1}{2})$};

\draw[dashed] (-0.6,3.1) node[wnode] {\tiny$(\seline,2\frac{1}{2})$} --(2.75,-0.25) node[wnode] {\tiny$(\seline,2\frac{1}{2})$};

\draw[dashed] (0.25,3.25) node[wnode] {\tiny$(\seline,3\frac{1}{2})$} --(3.6,-0.1) node[wnode] {\tiny$(\seline,3\frac{1}{2})$};

\draw[dashed] (1.25,3.25) node[wnode] {\tiny$(\seline,4\frac{1}{2})$} --(3.75,0.75);

\draw[dashed] (2.25,3.25) node[wnode] {\tiny$(\seline,5\frac{1}{2})$} --(3.75,1.75);

\draw[very thick](-0.5,0.5) --(0,0);

\draw[very thick](-0.5,1.5) -- (1,0);

\draw[very thick](-0.5,2.5) -- (2,0);

\draw[very thick](0,3) -- (3,0);

\draw[very thick](1,3) -- (3.5,0.5);

\draw[very thick](2,3) -- (3.5,1.5);

\draw[very thick](3,3) -- (3.5, 2.5);
        
    \end{scope}
\draw[very thick](-0.5,2.5) node[gnode] {$(-3,2)$} -- (0,3) node[wnode] {$(-3,3)$};
    
\draw[very thick](-0.5,1.5) node[gnode] {$(-2,1)$} -- (0,2) node[gnode] {$(-2,2)$} -- (0.5,2.5) node[gnode] {$(-2,3)$} -- (1,3) node[wnode] {$(-2,4)$};

\draw[very thick](-0.5,0.5) node[gnode] {$(-1,0)$} -- (0,1) node[gnode] {$(-1,1)$} -- (0.5,1.5) node[gnode] {$(-1,2)$} -- (1,2) node[gnode] {$(-1,3)$} -- (1.5,2.5) node[gnode] {$(-1,4)$} -- (2,3) node[wnode] {$(-1,5)$};

        \draw[very thick](0,0) node[wnode] {$(0,0)$} -- (0.5,0.5) node[gnode] {$(0,1)$} -- (1,1) node[gnode] {$(0,2)$} -- (1.5,1.5) node[gnode] {$(0,3)$} -- (2,2) node[gnode] {$(0,4)$} -- (2.5,2.5) node[gnode] {$(0,5)$} -- (3,3) node[wnode] {$(0,6)$};

        \draw[very thick](1,0) node[wnode] {$(1,1)$} -- (1.5,0.5) node[gnode] {$(1,2)$} -- (2,1) node[gnode] {$(1,3)$} -- (2.5,1.5) node[gnode] {$(1,4)$} -- (3,2) node[gnode] {$(1,5)$} -- (3.5,2.5) node[gnode] {$(1,6)$};
        
        \draw[very thick](2,0) node[wnode] {$(2,2)$} -- (2.5,0.5) node[gnode] {$(2,3)$} -- (3,1) node[gnode] {$(2,4)$} -- (3.5,1.5) node[gnode] {$(2,5)$};

        \draw[very thick](3,0) node[wnode] {$(3,3)$} -- (3.5,0.5) node[gnode] {$(3,4)$};

        \end{tikzpicture}
        \caption{A visualization of a portion of the index set~$I_6^{SCM}$. The set extends to the right and left. All nodes in the grey shaded region are interior.}
        \label{fig: cay-meng index set}
    \end{figure}
    
    \begin{definition}\label{def: sph cay meng frieze}
        
    For~$n\ge 4$, a \emph{spherical Cayley-Menger frieze of order}~$n$ is an array~$\zscm=(z_\alpha)_{\alpha\in I_n^{SCM}}$  of complex numbers indexed by~$I_n^{SCM}$ (cf.~Definition~\ref{def: sph cayley-menger frieze index set}) such that each diamond of numbers is a spherical Cayley-Menger diamond. Figure~\ref{fig: cayley-menger frieze example} shows an example of a spherical Cayley-Menger frieze.
    
    More concretely, for an interior node~$(i,j)\in\Z^2$, the tuple denoted
    \begin{equation*}
        \xdiamond_{i,j}(\zscm)=(z_{(i,j+1)},z_{(i+\frac{1}{2},\neline)},z_{(i+1,j)},z_{(\seline,j+\frac{1}{2})},z_{(i,j)},z_{(i+1,j+1)}).
    \end{equation*} forms the spherical Cayley-Menger diamond
    \begin{figure}[htbp!]
        \centering
        \begin{tikzpicture}
            [scale=0.6,every node/.style={rectangle,fill=white}]

        \begin{scope}[on background layer]
        
        \draw[dashed] (-2,4) node[left] {$(\seline,j+\frac{1}{2})$}
        --(2,0);
        \draw[dashed] (-2,0) node[left] {$(i+\frac{1}{2},\neline)$}
        --(2,4);
    \end{scope}
        \draw[very thick]
        (0,0) node {$(i+1,j)$}
        -- (-2,2) node {$(i,j)$}
        -- (0,4) node {$(i,j+1)$}
        -- (2,2) node {$(i+1,j+1)$}
        -- (0,0);
        \end{tikzpicture}
    \end{figure} 
    
    \noindent with~$
        \scmfour \left(\xdiamond_{i,j}(\zscm)\right)=0$. Along the boundary of the frieze, for~$i\in\Z$ we require that
    \begin{align}
        z_{(i,i)}&=z_{(i,i+n)}=0 \quad \text{(bottom and top rows)}\label{eq: boundary condition top and bottom row CM}\, ,\\
        z_{(i,i+1)}&=z_{(i+\frac{1}{2},\neline)}=z_{(\seline,i+\frac{1}{2})}\label{eq: boundary condition lines CM}\, ,\\
        z_{(i,i+n-1)}&=z_{(i-\frac{1}{2},\neline)}=z_{(\seline,i+n-\frac{1}{2})} \label{eq: boundary condition lines 2 CM}\, .
    \end{align}
    
\end{definition}

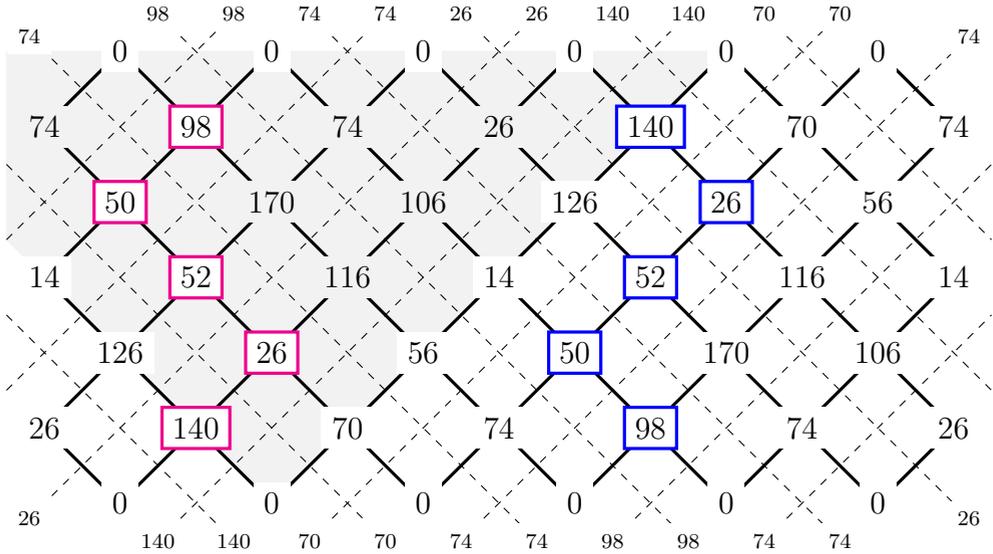
\begin{figure}[htbp!]
        \centering
        \begin{tikzpicture}
        [scale=2,greynode/.style={rectangle,fill=black!5,inner sep=2pt}, whitenode/.style={rectangle,fill=white, inner sep=2pt}, bordergreynode/.style={rectangle,draw=magenta,very thick,fill=white!5, inner sep=3pt}, borderwhitenode/.style={rectangle,draw=blue,fill=white,very thick,inner sep=3pt}]

        \begin{scope}[on background layer] 
\filldraw[black!5] (-0.75,2) -- (-0.75,1.75) -- (1,0) -- (4,3) -- (-0.75,3) -- (-0.75,2);
        \draw[very thick](-0.5,0.5) -- 
(0,0);

\draw[very thick](-0.5,1.5) -- (1,0);

\draw[very thick](-0.5,2.5) -- (2,0);

\draw[very thick](0,3) -- (3,0);

\draw[very thick](1,3) -- (4,0);

\draw[very thick](2,3) -- (5,0);

\draw[very thick](3,3) -- (5.5,0.5);

\draw[very thick](4,3)-- (5.5,1.5);

\draw[very thick](5,3) -- (5.5,2.5);

\begin{scope}
\draw[dashed] (-0.75,1.75) 
        --(0.75,3.25) node[whitenode] {\tiny$98$}; 

\draw[dashed] (-0.75,.75) 
        --(1.75,3.25) node[whitenode] {\tiny$74$};

        \draw[dashed] (-0.6,-0.1) node[whitenode] {\tiny$26$}
        --(2.75,3.25) node[whitenode] {\tiny$26$};

\draw[dashed] (0.25,-0.25)  node[whitenode] {\tiny$140$}
        -- (3.75 ,3.25) node[whitenode] {\tiny$140$};

\draw[dashed] (1.25,-0.25)  node[whitenode] {\tiny$70$}
        --(4.75,3.25) node[whitenode] {\tiny$70$};
        
\draw[dashed] (2.25,-0.25)  node[whitenode] {\tiny$74$}
        --(5.6,3.1) node[whitenode] {\tiny$74$};  

\draw[dashed] (3.25,-0.25) node[whitenode] {\tiny$98$} --(5.75,2.25);

\draw[dashed] (4.25,-0.25) node[whitenode] {\tiny$74$} --(5.75,1.25);

\draw[dashed] (-0.75,1.25)  --(0.75,-0.25) node[whitenode] {\tiny$140$};

\draw[dashed] (-0.75,2.25)  --(1.75,-0.25) node[whitenode] {\tiny$70$};

\draw[dashed] (-0.6,3.1)  node[whitenode] {\tiny$74$} --(2.75,-0.25) node[whitenode] {\tiny$74$};

\draw[dashed] (0.25,3.25) node[whitenode] {\tiny$98$} --(3.75,-0.25) node[whitenode] {\tiny$98$};

\draw[dashed] (1.25,3.25) node[whitenode] {\tiny$74$} --(4.75,-0.25) node[whitenode] {\tiny$74$};

\draw[dashed] (2.25,3.25) node[whitenode] {\tiny$26$} --(5.6,-0.1) node[whitenode] {\tiny$26$};

\draw[dashed] (3.25,3.25) node[whitenode] {\tiny$140$} -- (5.75,0.75);

\draw[dashed] (4.25,3.25) node[whitenode] {\tiny$70$} -- (5.75,1.75);

\end{scope}
    \end{scope}
\draw[very thick](-0.5,2.5) node[greynode] {$74$} -- (0,3) node[whitenode] {$0$};
    
\draw[very thick](-0.5,1.5) node[whitenode] {$14$} -- (0,2) node[bordergreynode] {$50$} -- (0.5,2.5) node[bordergreynode] {$98$} -- (1,3) node[whitenode] {$0$};

\draw[very thick](-0.5,0.5) node[whitenode] {$26$} -- (0,1) node[whitenode] {$126$} -- (0.5,1.5) node[bordergreynode] {$52$} -- (1,2) node[greynode] {$170$} -- (1.5,2.5) node[greynode] {$74$} -- (2,3) node[whitenode] {$0$};

        \draw[very thick](0,0) node[whitenode] {$0$} -- (0.5,0.5) node[bordergreynode] {$140$} -- (1,1) node[bordergreynode] {$26$} -- (1.5,1.5) node[greynode] {$116$} -- (2,2) node[greynode] {$106$} -- (2.5,2.5) node[greynode] {$26$} -- (3,3) node[whitenode] {$0$};

        \draw[very thick](1,0) node[whitenode] {$0$} -- (1.5,0.5) node[whitenode] {$70$} -- (2,1) node[whitenode] {$56$} -- (2.5,1.5) node[whitenode] {$14$} -- (3,2) node[whitenode] {$126$} -- (3.5,2.5) node[borderwhitenode] {$140$} -- (4,3) node[whitenode] {$0$};
        
        \draw[very thick](2,0) node[whitenode] {$0$} -- (2.5,0.5) node[whitenode] {$74$} -- (3,1) node[borderwhitenode] {$50$} -- (3.5,1.5) node[borderwhitenode] {$52$} -- (4,2) node[borderwhitenode] {$26$} -- (4.5,2.5) node[whitenode] {$70$} -- (5,3) node[whitenode] {$0$};

        \draw[very thick](3,0) node[whitenode] {$0$} -- (3.5,0.5) node[borderwhitenode] {$98$} -- (4,1) node[whitenode] {$170$} -- (4.5,1.5) node[whitenode] {$116$} -- (5,2) node[whitenode] {$56$} -- (5.5,2.5) node[whitenode] {$74$};

        \draw[very thick](4,0) node[whitenode] {$0$} -- (4.5,0.5) node[whitenode] {$74$} -- (5,1) node[whitenode] {$106$} -- (5.5,1.5) node[whitenode] {$14$};

        \draw[very thick](5,0) node[whitenode] {$0$} -- (5.5,0.5) node[whitenode] {$26$};

        \end{tikzpicture}
        \caption{A portion of a spherical Cayley-Menger frieze~$\zfrieze$ of order~$6$. This frieze corresponds to the hexagon on a sphere with radius~$7$ in Figure~\ref{fig: poly example} (the same polygon that gives rise to the spherical Heronian frieze in Figure~\ref{fig: sph heron frieze example}).
        To~visualize the glide symmetry, flip the gray shaded region upside down and paste it to the right of its original location. The blue and magenta entries form two separate traversing paths, their entries form identical sets. }
        \label{fig: cayley-menger frieze example}
    \end{figure}

\begin{definition}
    A \emph{traversing path} in a spherical Cayley-Menger frieze (cf.~Definition~\ref{def: traversing path heronian}) is an ordered collection
    \[
    \pi=((i_1,j_1),\dots,(i_{n-1},j_{n-1}),l_1,\dots,l_{n-2})
    \]
    of~$2n-3$ indices in~$I_n^{SCM}$ that consists of elements on a shortest path connecting the lower and upper boundaries of the frieze.  To~be more precise, we require that the first point is of the form~${(i_1,j_1)=(k,k+1)}$ for some~$k\in\Z$ (a point almost on the bottom boundary); similarly~${(i_{n-1},j_{n-1})=(m,m+n-1)}$ for some~$m\in\Z$ (just below the top boundary).  In particular, all~$(i_k,j_k)$ are interior nodes. The~$l_1,\dots,l_{n-2}$ are the dashed lines that intersect this shortest path.
\end{definition}

    \begin{remark} A spherical polygon~$P$ gives rise to a spherical Cayley-Menger frieze~$\zscm=\zscm(P)$.
    To~visualize the layout of the frieze, again start with two ants standing together at vertex~$k$ of the polygon. Correspondingly, the entry at vertex~$(k,k)$ in the bottom row is~$0$, indicating the squared distance between the ants. 
    
    As we move along a traversing path upwards in the frieze (now without entries at the half-integer points), each new node corresponds to one of the ants moving to an adjacent vertex of the polygon according to the index on the node.  The value at each node always shows the squared distance between the two ants.  The squared distance along the side of the polygon that an ant traverses to get to a new vertex is the label on the line~$(\seline, j+\frac{1}{2})$ or~$(i+\frac{1}{2},\neline)$ that perpendicularly intersects the traversing path on its way to the next node.
\end{remark}

There is an impediment to recovering the entire spherical Cayley-Menger frieze from solely the data associated to a traversing path.
When viewed as a polynomial in a single variable~$x_{ij}$, the spherical Cayley-Menger determinant has degree~$2$.
Hence, given~$a,b,c,d,e\in\C$, there are generically two values of~$f$, satisfying~$\scmfour (a,b,c,d,e,f)=0$. 

To~enable propagation, we are going to derive a \emph{coherence condition}, an additional relation amongst the entries of a spherical Cayley-Menger frieze that holds whenever a frieze arises from a spherical polygon.

\begin{definition}
We use notation for the partial derivatives of the spherical Cayley-Menger polynomial~$\scmfour (a,b,c,d,e,f)$ with respect to its~$6$ variables as shown in Figure~\ref{fig: notation for partials}.  The arrows in the notation reflect the location of the arguments in a spherical Cayley-Menger diamond. 
For example,
\begin{multline*}
\partial_\rightarrow\scmfour(a,b,c,d,e,f)=
    \frac{\partial \scmfour }{\partial f}(a,b,c,d,e,f)=\\
    2\left((a-d)(c-b)+e\left(a+b+c+d-e-2f-\mytfrac{\kap}{2} \left(ac-bd+ef\right)\right)\right)\, .
    \end{multline*}
The other five partial derivatives of~$\scmfour(a,b,c,d,e,f)$ can be obtained by evaluating the expression on the right hand side for various permutations of~$a,b,c,d,e,f$.  Note, for example, that a vertical reflection of the spherical Cayley-Menger diamond (swapping~$a\leftrightarrow c$ and~$b\leftrightarrow d$) leaves the expression unchanged.
\end{definition}
\vspace{-5pt}
\begin{figure}[htbp!]
    \centering
        \begin{tikzpicture}
            [scale=0.8,every node/.style={rectangle,fill=white, inner sep = 1pt}]
        \begin{scope}[on background layer]
        \draw[dashed] (-2,4) node[left=3pt] {$\partial_{\>\smallseline} \scmfour =\frac{\partial \scmfour }{\partial d}$}
        --(2,0);
        \draw[dashed] (-2,0) node[left=3pt] {$\partial_{\!\smallneline} \scmfour =\frac{\partial \scmfour }{\partial b}$}
        --(2,4);
    \end{scope}
        \draw[very thick]
        (0,0) node[below=-4pt] {$\partial_\downarrow \scmfour =\frac{\partial M}{\partial c}$}
        -- (-2,2) node[left=-10pt] {$\partial_\leftarrow \scmfour =\frac{\partial \scmfour }{\partial e}$}
        -- (0,4) node[above=-3pt] {$\partial_\uparrow \scmfour =\frac{\partial \scmfour }{\partial a}$}
        -- (2,2) node[right=-15pt] {$\partial_\rightarrow \scmfour =\frac{\partial M}{\partial f}$}
        -- (0,0);
        \end{tikzpicture}
    \caption{Notation for partial derivatives of~$\scmfour (a,b,c,d,e,f)$.}
    \label{fig: notation for partials}
\end{figure}
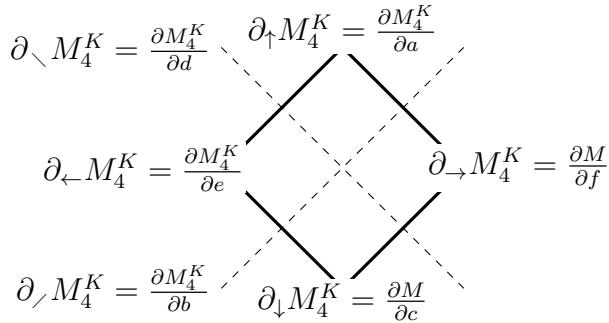
\vspace{-5pt}
\begin{lemma}\label{lem: partial H_kap}
For~$(a,b,c,d,e,f)$, a spherical Cayley-Menger diamond, we have:
\begin{align}
    (\partial_\rightarrow \scmfour (a,b,c,d,e,f))^2&=4\hkap(b,c,e)\hkap(a,d,e)\, ,\label{eq: partial f identity}\\
    (\partial_\uparrow \scmfour (a,b,c,d,e,f))^2&=4\hkap(b,c,e)\hkap(c,d,f)\, ,\nonumber\\
    (\partial_\downarrow \scmfour (a,b,c,d,e,f))^2&=4\hkap(a,b,f)\hkap(a,d,e)\, ,\nonumber\\
    (\partial_\leftarrow \scmfour (a,b,c,d,e,f))^2&=4\hkap(a,b,f)\hkap(c,d,f)\, ,\nonumber\\
    (\partial_{\>\smallseline} \scmfour (a,b,c,d,e,f))^2&=4\hkap(b,c,e)\hkap(a,b,f)\, ,\nonumber\\
    (\partial_{\!\smallneline} \scmfour (a,b,c,d,e,f))^2&=4\hkap(c,d,f)\hkap(a,d,e)\, .\nonumber
\end{align}
\end{lemma}
\begin{proof}
    One can manually verify the following polynomial identity, which holds for any values of~$a,b,c,d,e,f\in\C$:
    \vspace{-5pt}
    \begin{multline}\label{eq: id for partial kap}
        (\partial_\rightarrow \scmfour (a,b,c,d,e,f))^2=\\-8e\left(1-\mytfrac{\kap e}{4}\right)\scmfour (a,b,c,d,e,f)+4\hkap(b,c,e)\hkap(a,d,e) \, .
    \end{multline}
    Since~$\scmfour (a,b,c,e,d,e,f)=0$, we get equation~\eqref{eq: partial f identity}.  The other equations can be derived in a similar way.
\end{proof}

When a spherical Cayley-Menger diamond arises from a quadrilateral on a sphere, Lemma~\ref{lem: partial H_kap} can be strengthened by assigning a geometric meaning to each evaluation of a partial derivative of~$\scmfour$.

\begin{proposition} \label{prop: deriv geom meaning} Consider a spherical quadrilateral~${P=A_1A_2A_3A_4}$.  Let
     \[
     \xx=(x_{14},x_{12},x_{23},x_{34},x_{13},x_{24})=(a,b,c,d,e,f)
     \]
     denote the corresponding spherical Cayley-Menger diamond.
     Then
     \begin{align}
         \partial_\leftarrow\scmfour(\xx)&=-2\skap_{124}\skap_{234} \, , & \partial_\rightarrow\scmfour(\xx)&=-2\skap_{123}\skap_{134} \, , \label{eq: deriv 1}\\
         \partial_{\!\smallneline}\scmfour(\xx)&=2\skap_{134}\skap_{234} \, , & \partial_{\>\smallseline}\scmfour(\xx)&=2\skap_{123}\skap_{124} \, , \label{eq: deriv 2}\\
         \partial_\uparrow\scmfour(\xx)&=2\skap_{123}\skap_{234} \, , & \partial_\downarrow\scmfour(\xx)&=2\skap_{124}\skap_{134} \label{eq: deriv 3}\, .
     \end{align}
\end{proposition}
     \begin{proof}  Denoting~$p=\skap_{123}$ and~$q=\skap_{134}$,  the relation~$\partial_\rightarrow\scmfour(\xx)=-2\skap_{123}\skap_{134}$ is the equation~\eqref{eq: eq for pq bret}.
     The other relations~\eqref{eq: deriv 1}-\eqref{eq: deriv 3} can be proved similarly, by applying Lemma~\ref{lem: vector mult} to appropriate collections of vectors.
\end{proof}

\begin{remark} Specializing~$\kap=0$ in the equations~\eqref{eq: deriv 1}-\eqref{eq: deriv 3} recovers their analogues obtained in \cite{fomin_heronian_2021}.  For example,~\eqref{eq: pq in euclidean case} is obtained by setting~$\kap=0$ in the relation~$\partial_\rightarrow\scmfour(\xx)=-2\skap_{123}\skap_{134}$. For proof, Fomin and Setiabrata cite~\cite{khimshiashvili_point_2017}, which references~\cite[40]{dziobek_ueber_1900} and~\cite{hampton_concave_2002}. The proof method in  \cite{khimshiashvili_point_2017} is to apply the partial derivative directly to the Cayley-Menger matrix, perform row and column operations, reinterpret as a cross product, and then use the Binet-Cauchy formula.

Neither our proof of Proposition~\ref{prop: deriv geom meaning}, nor the proof in \cite{khimshiashvili_point_2017}, give a geometric explanation for the equations~\eqref{eq: deriv 1}-\eqref{eq: deriv 3}, in particular for the appearance of partial derivatives of~$\scmfour(\xx)$. We next attempt to provide such an explanation.

The polynomial~$\scmfour(\xx)$ is quadratic in~$f$: we have $\scmfour(\xx)=Af^2+Bf+C$ for some~$A,B,C\in\C$.  Therefore 
\vspace{-5pt}
    \begin{equation*}
        \partial_\rightarrow\scmfour(\xx)=\mytfrac{\partial}{\partial f}\scmfour(\xx)=2Af+B\, .
    \end{equation*}
    Denote~$D=B^2-4AC$, the discriminant of~$\scmfour(\xx)$.  Rewriting
    \begin{equation*}
        \scmfour(\xx)= \mytfrac{(2Af+B)^2-(B^2-4AC)}{4A}\, ,
    \end{equation*}
    and recalling that~$\scmfour(\xx)=0$, we conclude that~$
\left(\mytfrac{\partial}{\partial f}\scmfour(\xx)\right)^2=D\, .$  By Proposition~\ref{prop: deriv geom meaning}, it follows that
\vspace{-5pt}
    \begin{equation*}
        D=\left(\mytfrac{\partial}{\partial f}\scmfour(\xx)\right)^2=4(S_{123}^\kap)^2 (S_{134}^\kap)^2\, .
    \end{equation*}
Thus~$D=0$ when either~$\skap_{123}=0$ or~$\skap_{134}=0$. 

This can be interpreted geometrically as follows. The discriminant~$D$ vanishes when~$\scmfour(\xx)=0$ has a a double root, i.e., when both solutions for~$f$ are the same. 
Fixing~$a,b,c,d,e$, we get two solutions for~$f$ related to each other by reflecting the triangle~$A_1A_2A_3$ across the line~$A_1A_3$ (or by reflecting~$A_1A_3A_4$ across~$A_1A_3$):

    \begin{figure}[htbp!]
        \centering
        \vspace{-5pt}
\begin{tikzpicture}[scale=0.3]
    \draw(0,0) node[left] {$A_4$} -- (-2,3) node[left] {$A_1$} -- (-1, 5) node[left] {$A_2$} -- (2,2) node[right] {$A_3$} -- (0,0);
    \draw(-2,3) -- (2,2);
    \draw[dotted] (0,0) -- (-1,5);
        \draw(-.4,3.5) to [out=0, in=180] (2,4) node[right] {option 1 for~$f$};
\end{tikzpicture}     
\hspace{1 cm}
\begin{tikzpicture}[scale=0.3]
    \draw (0,0) node[right] {$A_4$} -- (-2,3) node[left] {$A_1$} -- (-1.5, 1) node[left] {$A_2$} -- (2,2) node[right] {$A_3$} -- (0,0);
    \draw (-2,3) -- (2,2);
    \draw[dotted] (0,0) --  (-1.5,1);
    \draw (-1,0.5) to [out=200,in=-30] (-3,-.5) node[left] {option 2 for~$f$};
\end{tikzpicture}        
\vspace{-5pt}
        \label{fig: two solutions for f}
    \end{figure}
    
\noindent The two options for~$f$ coincide with each other if and only if either~$A_2$ or~$A_4$ lie on a great circle of~$\bfS$ passing through~$A_1$ and~$A_3$. This would make either~$A_1A_2A_3$ or~$A_1A_3A_4$ a ``degenerate'' triangle, with~$S_{123}^\kap$ or~$\skap_{134}$ vanishing:
    \begin{figure}[htbp!]
        \centering
        \vspace{-5pt}
\begin{tikzpicture}[scale=0.3]
    \draw (0,0) node[below] {$A_4$} -- (-2,3) node[left] {$A_1$} -- (-1, 2.8) node[above] {$A_2$} -- (2,2) node[right] {$A_3$} -- (0,0);
    \draw (-2,3) -- (2,2);
    \draw[dotted] (0,0) -- (-1,2.8);
        \draw (-.2,1.5) to [out=0, in=180] (2,0) node[right] {unique~$f$};
\end{tikzpicture}     
\hspace{1 cm}
\begin{tikzpicture}[scale=0.3]
    \draw (6,0) node[right] {$A_4$} -- (-2,2) node[left] {$A_1$} -- (-1,4) node[above] {$A_2$} -- (2,1) node[below] {$A_3$};
    
    \draw[dotted] (6,0) -- node[above right] {unique~$f$}  (-1,4);
    
\end{tikzpicture}    
        \label{fig: degenerate f unique}
          \vspace{-10pt}
    \end{figure}
\end{remark}

\begin{proposition} \label{prop: 5.9 analogue}   Consider four interlocking spherical Cayley-Menger diamonds as depicted in Figure~\ref{fig: four cayley menger diamonds}:
\begin{align*}
    \xx_1&=(x_{15},x_{12},x_{24},x_{45},x_{14},x_{25})\, ,\\
    \xx_2&=(x_{16},x_{12},x_{25},x_{56},x_{15},x_{26})\, ,\\
    \xx_3&=(x_{25},x_{23},x_{34},x_{45},x_{24},x_{35})\, ,\\
    \xx_4&=(x_{26},x_{23},x_{35},x_{56},x_{25},x_{36})\, .
\end{align*}
Then 
\begin{equation*}
    \left(\partial_\leftarrow \scmfour (\xx_1)\partial_\rightarrow \scmfour (\xx_4)\right)^2=\left(\partial_\uparrow \scmfour (\xx_2)\partial_\downarrow \scmfour (\xx_3)\right)^2.
\end{equation*}
If the four Cayley-Menger diamonds correspond to a hexagon~$P=A_1A_2\cdots A_6$, then 
\begin{equation}\label{eq: hexagon coherence condition}
    \partial_\leftarrow \scmfour (\xx_1)\partial_\rightarrow \scmfour (\xx_4)=\partial_\uparrow \scmfour (\xx_2)\partial_\downarrow \scmfour (\xx_3)\, .
\end{equation}
\end{proposition}
\begin{figure}[htbp!]
    \centering
    \begin{tikzpicture}[scale=0.5,every node/.style={rectangle,fill=white,inner sep=3pt}]

        \begin{scope}[on background layer]
\draw[dashed] (-2,8) node {$x_{56}$}--(4,2);
\draw[dashed] (-4,6) node {$x_{45}$}--(2,0);
\draw[dashed] (-4,2) node {$x_{12}$}--(2,8);
\draw[dashed] (-2,0) node {$x_{23}$}--(4,6);
\draw[very thick] (-2,6) -- (0,4) -- (2,2);
    \end{scope}
    \draw[very thick] (-2,2) -- (0,4) node {$x_{25}$}
        -- (2,6);
        \draw[very thick]
        (0,0) node {$x_{34}$}
        -- (-2,2) node {$x_{24}$}
        -- (-4,4) node {$x_{14}$}
        -- (-2,6) node {$x_{15}$}
        -- (0,8) node {$x_{16}$}
        -- (2,6) node {$x_{26}$}
        -- (4,4) node {$x_{36}$}
        -- (2,2) node {$x_{35}$}
        -- (0,0);
        
        \draw[very thick] (-2,4) node[above] {$\xx_1$};
        \draw[very thick] (0,6) node[above] {$\xx_2$};
        \draw[very thick] (2,4) node[above] {$\xx_4$};
        \draw[very thick] (0,2) node[above] {$\xx_3$};
    \end{tikzpicture}
    \caption{Four interlocking Cayley-Menger diamonds.}
    \label{fig: four cayley menger diamonds}
\end{figure}
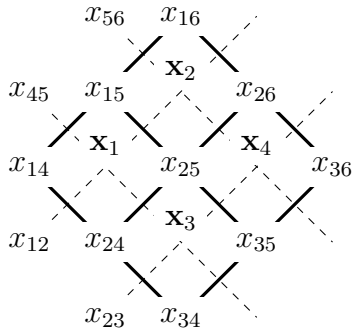

\begin{proof}
    We use the shorthand~$\hkap_{ijk}=\hkap(x_{ij},x_{ik},x_{jk})$. Using Lemma~\ref{lem: partial H_kap} with each diamond gives us
    \begin{align*}
        (\partial_\leftarrow\scmfour(\xx_1))^2&=4H^\kap_{125} H^\kap_{245}\, ,\\
        (\partial_\rightarrow\scmfour(\xx_4))^2&=4H^\kap_{235}H^\kap_{256}\, ,\\
        (\partial_\uparrow\scmfour(\xx_2))^2&=4H^\kap_{125}H^\kap_{256}\, ,\\
        (\partial_\downarrow\scmfour(\xx_3))^2&=4H^\kap_{235}H^\kap_{245}\, .
    \end{align*}
    So
    \begin{equation*}
        \left(\partial_\leftarrow \scmfour(\xx_1)\partial_\rightarrow\scmfour(\xx_4)\right)^2=16H^\kap_{125}H^\kap_{245}H^\kap_{256}H^\kap_{235}=\left(\partial_\uparrow\scmfour(\xx_2)\partial_\uparrow\scmfour(\xx_3)\right)^2.
    \end{equation*}
    If the diamonds correspond to a hexagon~$P=A_1A_2\cdots A_6$ then by Proposition~\ref{prop: deriv geom meaning} 
    \begin{align*}
        \partial_\leftarrow\scmfour(\xx_1)&=-2\skap_{125} \skap_{245}\, ,\\
        \partial_\rightarrow\scmfour(\xx_4)&=-2\skap_{235}\skap_{256}\, ,\\
        \partial_\uparrow\scmfour(\xx_2)&=2\skap_{125}\skap_{256}\, ,\\
        \partial_\downarrow\scmfour(\xx_3)&=2\skap_{235}\skap_{245}\, ,
    \end{align*}
    implying
    \begin{equation*}
        \partial_\leftarrow \scmfour(\xx_1)\partial_\rightarrow\scmfour(\xx_4)=4\skap_{125}\skap_{245}\skap_{256}\skap_{235}=\partial_\uparrow\scmfour(\xx_2)\partial_\downarrow\scmfour(\xx_3)\, .\qedhere
    \end{equation*}
\end{proof}

\begin{remark}
    By Proposition~\ref{prop: 5.9 analogue}, for any four interlocking Cayley-Menger diamonds in a frieze~$\zscm$, e.g. for~$(i,j)\in\Z^2$ with~$2\le j-i\le n-2$, we get
    \begin{equation*}
\left(\partial_\leftarrow \scmfour \left(\xdiamond_{i-1,j-1}(\zscm)\right)\partial_\rightarrow \scmfour \left(\xdiamond_{i,j}(\zscm)\right)\right)^2= \left(\partial_\uparrow \scmfour \left(\xdiamond_{i-1,j}(\zscm)\right)\partial_\downarrow \scmfour \left(\xdiamond_{i,j-1}(\zscm)\right)\right)^2,
    \end{equation*}
    and therefore
\begin{equation*}
        \partial_\leftarrow \scmfour \left(\xdiamond_{i-1,j-1}(\zscm)\right)\partial_\rightarrow \scmfour \left(\xdiamond_{i,j}(\zscm)\right)=\pm\partial_\uparrow \scmfour \left(\xdiamond_{i-1,j}(\zscm)\right)\partial_\downarrow \scmfour \left(\xdiamond_{i,j-1}(\zscm)\right).
    \end{equation*}

\end{remark}

\begin{definition}
    We call a spherical Cayley-Menger frieze \emph{coherent} if for all~$i,j\in\Z^2$ with~$2\le j-i\le n-2$ the following \emph{coherence condition} is satisfied:
\begin{equation}\label{eq: coherence condition}
        \partial_\leftarrow \scmfour \left(\xdiamond_{i-1,j-1}(\zscm)\right)\partial_\rightarrow \scmfour \left(\xdiamond_{i,j}(\zscm)\right)=\partial_\uparrow \scmfour \left(\xdiamond_{i-1,j}(\zscm)\right)\partial_\downarrow \scmfour \left(\xdiamond_{i,j-1}(\zscm)\right).
    \end{equation}

\end{definition}

\begin{theorem}\label{thm: 5.12 analogue}
    For any polygon~$P=A_1A_2\cdots A_n$ on~$\bfS$, the spherical Cayley-Menger frieze~$\zscm(P)$ is coherent.
\end{theorem}
\begin{proof}
    This is a direct corollary of Proposition~\ref{prop: 5.9 analogue}, see identity~\eqref{eq: hexagon coherence condition}.
\end{proof}

\begin{lemma}\label{lem: is a sph cay meng diamond}
    Let
    \begin{align*}
        \xx_2&=(x_{16},x_{12},x_{25},x_{56},x_{15},x_{26})\, ,\\
        \xx_3&=
        (x_{25},x_{23},x_{34},x_{45},x_{24},x_{35})\, ,\\
        \xx_4&=(x_{26},x_{23},x_{35},x_{56},x_{25},x_{36})
            \end{align*}
            be three spherical Cayley-Menger diamonds
            arranged as in Figure~\ref{fig: four cayley menger diamonds}. Assume that 
            \begin{equation}\label{eq: stuff not zero}
                x_{25}
    \left(x_{25} - \mytfrac{4}{\kap}\right)H^\kap(x_{24},x_{25},x_{45})H^\kap(x_{12},x_{15},x_{25})\ne 0\, .
            \end{equation} Let~$x_{14}\in\C$ be such that for~$
         \xx_1=(x_{15},x_{12},x_{24},x_{45},x_{14},x_{25})$, the coherence condition holds:
            \begin{equation}\label{eq: coherence condition for x1}
                \partial_\leftarrow\scmfour(\xx_1)\partial_\rightarrow\scmfour(\xx_4)=\partial_\uparrow\scmfour(\xx_2)\partial_\downarrow\scmfour(\xx_3)\, .
            \end{equation}
            Then the tuple~$\xx_1$ is also a spherical Cayley-Menger diamond, i.e.~$\scmfour(\xx_1)=0$.
\end{lemma}
\begin{proof}
    We use the shorthand~$\hkap_{ij}=\hkap(x_{ij},x_{ik},x_{jk})$. By Lemma~\ref{lem: partial H_kap}, we get
\begin{equation}\label{eq: eq 1}
    \left(\partial_\uparrow \scmfour (\xx_2)\partial_\downarrow \scmfour (\xx_3)\right)^2=16H^\kap_{245}H^\kap_{125}H^\kap_{256}H^\kap_{235}\, . 
    \end{equation}
From equation~\eqref{eq: id for partial kap}, we get    \begin{multline}\label{eq: eq 2}\left(\partial_\leftarrow \scmfour (\xx_1)\partial_\rightarrow \scmfour (\xx_4)\right)^2\\=\left(4H^\kap_{245}H^\kap_{125}-8x_{25}\scmfour (\xx_1)\left(1-\mytfrac{\kap x_{25}}{4}\right)\right)4H^\kap_{256}H^\kap_{235}\\
=16H^\kap_{245}H^\kap_{125}H^\kap_{256}H^\kap_{235}+8\kap x_{25}\left(x_{25}-\mytfrac{4}{\kap}\right)H^\kap_{256}H^\kap_{235}\scmfour (\xx_1)\, .
\end{multline}
Combining~\eqref{eq: coherence condition for x1},~\eqref{eq: eq 1}, and~\eqref{eq: eq 2}, we conclude that 
\begin{equation*}
    8\kap x_{25}\left(x_{25}-\mytfrac{4}{\kap}\right)H^\kap_{256}H^\kap_{235}\scmfour (\xx_1)=0\, .
\end{equation*}
By assumption~\eqref{eq: stuff not zero}, this implies that we must have~$\scmfour (\xx_1)=0$, so~$\xx_1$ is a spherical Cayley-Menger diamond.
\end{proof}

\begin{proposition} \label{prop: 5.15 analogue}
    Start with three interlocking spherical Cayley-Menger diamonds
    \begin{align*}
        \xx_1&=(x_{15},x_{12},x_{24},x_{45},x_{14},x_{25})\, ,\\
        \xx_2&=(x_{16},x_{12},x_{25},x_{56},x_{15},x_{26})\, ,\\
        \xx_3&=
        (x_{25},x_{23},x_{34},x_{45},x_{24},x_{35})
        \end{align*}
        arranged as in Figure~\ref{fig: four cayley menger diamonds}.
    If 
    \[x_{25}
    \left(x_{25} - \mytfrac{4}{\kap}\right)H^\kap(x_{24},x_{25},x_{45})H^\kap(x_{12},x_{15},x_{25})\ne 0\, ,
    \] then there exists a unique~$x_{36}\in\C$ such that the tuple~$
        \xx_4=(x_{26},x_{23},x_{35},x_{56},x_{25},x_{36})$
     together with~$\xx_1,\xx_2,\xx_3$ satisfy the coherence condition~\eqref{eq: coherence condition}.  As a consequence,~$\xx_4$ turns out to be a spherical Cayley-Menger diamond, i.e.~$\scmfour(\xx_4)=0$.

    \noindent Similarly, let~$\xx_2,\xx_3,\xx_4$ be three interlocking spherical Cayley-Menger diamonds.  If 
    \begin{equation*}
        x_{25}\left(x_{25}-\mytfrac{4}{\kap}\right)H^\kap(x_{25},x_{26},x_{56})H^\kap(x_{23},x_{25},x_{35})\ne 0\, ,
    \end{equation*}
    then there exists a unique~$x_{14}\in\C$ such that the tuples~
    $\xx_1,\xx_2,\xx_3,\xx_4$ satisfy the coherence condition. The tuple~$\xx_1$ is therefore a spherical Cayley-Menger diamond and satisfies~$\scmfour(\xx_1)=0$.
\end{proposition}
\begin{proof}
    Consider the second case, where we know the values for~$\xx_2,\xx_3,\xx_4$ and want to solve for~$x_{14}$ to complete~$\xx_1$. 
 We want the resulting collection of four spherical Cayley-Menger diamonds to satisfy the coherence condition~\eqref{eq: coherence condition}
 \begin{multline*}
     \mytfrac{\partial}{\partial x_{14}}\scmfour(x_{15},x_{12},x_{24},x_{45},x_{14},x_{25})\mytfrac{\partial}{\partial x_{36}}\scmfour(x_{26},x_{23},x_{35},x_{56},x_{25},x_{36})=\\
     \mytfrac{\partial}{\partial x_{16}}\scmfour(x_{16},x_{12},x_{25},x_{56},x_{15},x_{26})\mytfrac{\partial}{\partial x_{34}}\scmfour(x_{25},x_{23},x_{34},x_{45},x_{24},x_{35})\, .
 \end{multline*}
 This equation
 is linear in the variable~$x_{14}$.  The coefficient of~$x_{14}$ is
    \begin{equation*}
    \kap x_{25}\left(x_{25} - \mytfrac{4}{\kap}\right)\partial_\rightarrow \scmfour (\xx_4)\, .
    \end{equation*}
    By assumption and Lemma~\ref{lem: partial H_kap}, we have
    \begin{equation*}
    \left(\kap x_{25}\left(x_{25} - \mytfrac{4}{\kap}\right)\right)^2\left(\partial_\rightarrow \scmfour (\xx_4)\right)^2=\left(\kap x_{25}\left(x_{25} - \mytfrac{4}{\kap}\right)\right)^2\cdot 4H^\kap_{256}H^\kap_{235}\ne 0
    \end{equation*}
and so we get a unique solution for~$x_{14}$ such that the coherence condition holds for the four diamonds~$\xx_1,\xx_2,\xx_3,\xx_4$.  

In the first case we can similarly solve for~$x_{36}$.  In both cases, to see that the resulting tuple is a spherical Cayley-Menger diamond, apply Lemma~\ref{lem: is a sph cay meng diamond}.
\end{proof}

\begin{remark}
    If all the input data are real and the diamonds arise from a polygon~$A_1\cdots A_n$ on a sphere, the condition
    \[
    x_{25}
    \left(x_{25} - \mytfrac{4}{\kap}\right)H^\kap(x_{24},x_{25},x_{45})H^\kap(x_{12},x_{15},x_{25})\ne 0
    \]
    is equivalent to requiring that neither of the triples~$\{A_2, A_4, A_5\}$ or~$\{A_1, A_2,A_5\}$  lie on a great circle of~$\bfS$.
\end{remark}

We will now show how to recover the entries of a spherical Cayley-Menger frieze from a small initial data set.

\begin{definition}\label{def: thickening}
    Let~$n\ge 4$. Let~$
\pi=((i_1,j_1),\dots,(i_{n-1},j_{n-1}),l_1,\dots,l_{n-2})$
    be a traversing path in a spherical Cayley-Menger frieze of order~$n$.

Define a \emph{thickening} of~$\pi$, denoted~$\overline{\pi}$, to be the original path~$\pi$ along with each of the interior nodes in the frieze shifted by~$(1,1)$ (to the right) from~$\pi$.  Formally, 
\begin{equation*}
    \overline{\pi}=((i_1,j_1),\dots,(i_{n-1},j_{n-1}),(i_1+1,j_1+1),\dots,(i_{n-1}+1,j_{n-1}+1),l_1,\dots,l_{n-2}).
\end{equation*}
  A thickened path~$\overline{\pi}$ has~$3n-4$ indices, see Figure~\ref{fig: thick path}.

  \begin{figure}[htbp!]
        \centering
        \begin{tikzpicture}
        [scale=2,greynode/.style={rectangle,fill=black!5,inner sep=2pt}, whitenode/.style={rectangle,fill=white,inner sep=2pt}, bordergreynode/.style={rectangle,draw=magenta,very thick,fill=black!5,inner sep=3pt}, borderwhitenode/.style={rectangle,draw=magenta,fill=white,very thick,inner sep=3pt}]

        \begin{scope}[on background layer] 

        \draw[very thick] (-0.5,0.5) -- 
(0,0);

\draw[very thick] (-0.5,1.5) -- (1,0);

\draw[very thick] (-0.5,2.5) -- (2,0);

\draw[very thick] (0,3) -- (3,0);

\draw[very thick] (1,3) -- (4,0);

\draw[very thick] (2,3) -- (5,0);

\draw[very thick] (3,3) -- (5.5,0.5);

\draw[very thick] (4,3)-- (5.5,1.5);

\draw[very thick] (5,3) -- (5.5,2.5);

\draw[very thick][magenta, line width=3pt] (.5,0.5) -- (1,1) --(0.5,1.5) -- (0,2) -- (0.5,2.5);
 
\begin{scope}
\draw[dashed] (-0.75,1.75) 
        --(0.75,3.25) node[whitenode] {\tiny$98$}; 

\draw[dashed] (-0.75,.75) 
        --(1.75,3.25) node[whitenode] {\tiny$74$};

        \draw[dashed] (-0.6,-0.1) node[whitenode] {\tiny$26$}
        --(2.75,3.25) node[whitenode] {\tiny$26$};

\draw[dashed] (0.25,-0.25)  node[whitenode] {\tiny$140$}
        -- (3.75 ,3.25) node[whitenode] {\tiny$140$};

\draw[dashed] (1.25,-0.25)  node[whitenode] {\tiny$70$}
        --(4.75,3.25) node[whitenode] {\tiny$70$};
        
\draw[dashed] (2.25,-0.25)  node[whitenode] {\tiny$74$}
        --(5.6,3.1) node[whitenode] {\tiny$74$};  

\draw[dashed] (3.25,-0.25) node[whitenode] {\tiny$98$} --(5.75,2.25);

\draw[dashed] (4.25,-0.25) node[whitenode] {\tiny$74$} --(5.75,1.25);

\draw[dashed] (-0.75,1.25)  --(0.75,-0.25) node[whitenode] {\tiny$140$};

\draw[dashed] (-0.75,2.25)  --(1.75,-0.25) node[whitenode] {\tiny$70$};

\draw[dashed] (-0.6,3.1)  node[whitenode] {\tiny$74$} --(2.75,-0.25) node[whitenode] {\tiny$74$};

\draw[dashed] (0.25,3.25) node[whitenode] {\tiny$98$} --(3.75,-0.25) node[whitenode] {\tiny$98$};

\draw[dashed] (1.25,3.25) node[whitenode] {\tiny$74$} --(4.75,-0.25) node[whitenode] {\tiny$74$};

\draw[dashed] (2.25,3.25) node[whitenode] {\tiny$26$} --(5.6,-0.1) node[whitenode] {\tiny$26$};

\draw[dashed] (3.25,3.25) node[whitenode] {\tiny$140$} -- (5.75,0.75);

\draw[dashed] (4.25,3.25) node[whitenode] {\tiny$70$} -- (5.75,1.75);

\end{scope}
    \end{scope}
\draw[very thick](-0.5,2.5) node[whitenode] {$74$} -- (0,3) node[whitenode] {$0$};
    
\draw[very thick](-0.5,1.5) node[whitenode] {$14$} -- (0,2) node[borderwhitenode] {$50$} ;

\draw[very thick](0.5,2.5) node[borderwhitenode] {$98$} -- (1,3) node[whitenode] {$0$};

\draw[very thick](-0.5,0.5) node[whitenode] {$26$} -- (0,1) node[whitenode] {$126$} -- (0.5,1.5) node[borderwhitenode] {$52$} -- (1,2) node[borderwhitenode] {$170$} -- (1.5,2.5) node[borderwhitenode] {$74$} -- (2,3) node[whitenode] {$0$};

        \draw[very thick](0,0) node[whitenode] {$0$}--(0.5,0.5) node[borderwhitenode] {$140$};
        
        \draw[very thick](1,1) node[borderwhitenode] {$26$} -- (1.5,1.5) node[borderwhitenode] {$116$} -- (2,2) node[whitenode] {$106$} -- (2.5,2.5) node[whitenode] {$26$} -- (3,3) node[whitenode] {$0$};

        \draw[very thick](1,0) node[whitenode] {$0$} -- (1.5,0.5) node[borderwhitenode] {$70$} -- (2,1) node[borderwhitenode] {$56$} -- (2.5,1.5) node[whitenode] {$14$} -- (3,2) node[whitenode] {$126$} -- (3.5,2.5) node[whitenode] {$140$} -- (4,3) node[whitenode] {$0$};
        
        \draw[very thick](2,0) node[whitenode] {$0$} -- (2.5,0.5) node[whitenode] {$74$} -- (3,1) node[whitenode] {$50$} -- (3.5,1.5) node[whitenode] {$52$} -- (4,2) node[whitenode] {$26$} -- (4.5,2.5) node[whitenode] {$70$} -- (5,3) node[whitenode] {$0$};

        \draw[very thick](3,0) node[whitenode] {$0$} -- (3.5,0.5) node[whitenode] {$98$} -- (4,1) node[whitenode] {$170$} -- (4.5,1.5) node[whitenode] {$116$} -- (5,2) node[whitenode] {$56$} -- (5.5,2.5) node[whitenode] {$74$};

        \draw[very thick](4,0) node[whitenode] {$0$} -- (4.5,0.5) node[whitenode] {$74$} -- (5,1) node[whitenode] {$106$} -- (5.5,1.5) node[whitenode] {$14$};

        \draw[very thick](5,0) node[whitenode] {$0$} -- (5.5,0.5) node[whitenode] {$26$};

        \end{tikzpicture}
        \caption{The same Cayley-Menger frieze as Figure~\ref{fig: cayley-menger frieze example}.  The entries highlighted in magenta are the thickening of the traversing path on the left indicated by magenta lines.}
        \label{fig: thick path}
    \end{figure}
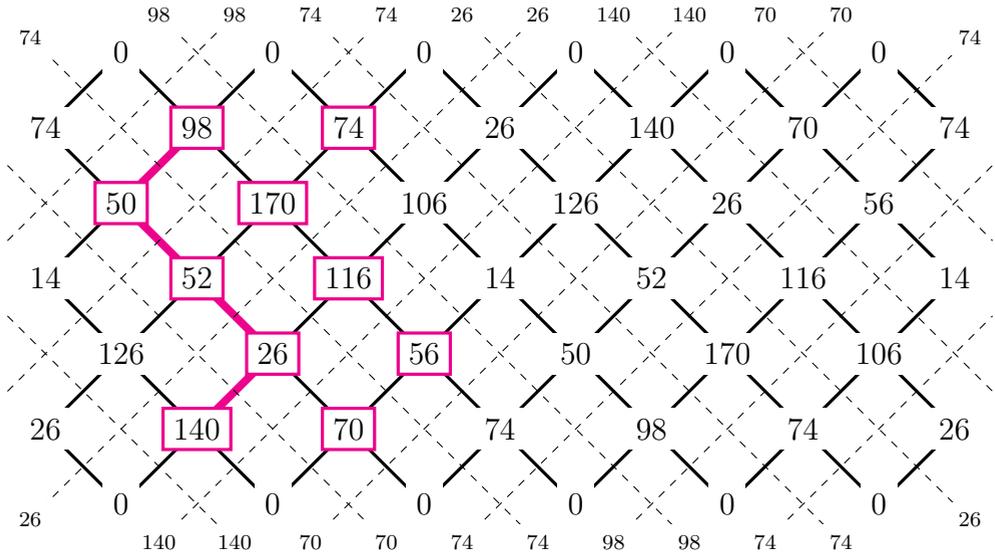

\end{definition}

\begin{theorem}\label{thm: 5.18 analogue} 
    Let~$\zscm=(z_\alpha)_{\alpha\in I_n^{SCM}}$ be a coherent spherical Cayley-Menger frieze of order~$n$ such that
    \begin{align}
      &  \quad z_{(i,j)}\not\in\left\{0,\mytfrac{4}{\kap}\right\} \text{ for any } (i,j)\in\Z^2 \text{ with } 2\le j-i\le n-2, \text{ and}\label{eq: condition interior not zero}\\
& \begin{cases} \hkap(z_{(i,j)},z_{(i+1,j)},z_{(i+\frac12,\smallneline)}) \neq 0\, , \\
\hkap(z_{(i,j-1)},z_{(i,j)},z_{(\smallseline,j-\frac12)}) \neq 0 \label{eq: condition triangles not zero}
\end{cases}
\text{for~$(i,j)\in \Z^2,\  2\le j-i\le n-1$\, .}
        \end{align}
        Then~$\zscm$ is uniquely determined by its entries in the thickening~$\overline{\pi}$ of any traversing path~$\pi$.
\end{theorem}
\begin{proof}
    Fix a traversing path~$\pi$ and its thickening~$\overline{\pi}$.  The boundary condition~\eqref{eq: boundary condition top and bottom row CM} ensures that all entries indexed by~$(i,i), \, (i,i+n)$ for~$i\in\Z$ (the top and bottom rows)  are zero.  
    
    For the interior nodes, we can propagate outwards from~$\overline{\pi}$ by applying Proposition~\ref{prop: 5.15 analogue} to ensure coherence of the resulting frieze.  The conditions~\eqref{eq: condition interior not zero}-\eqref{eq: condition triangles not zero} provide the required non-vanishing conditions.  Because Proposition~\ref{prop: 5.15 analogue} guarantees uniqueness of the resulting entries, they must agree with those in~$\zscm$. 
    
    The boundary conditions~\eqref{eq: boundary condition lines CM} and~\eqref{eq: boundary condition lines 2 CM} give us the lines intersecting just above and just below the thickened path~$\overline{\pi}$.  
\end{proof}

\begin{remark}  We can  uniquely recover a spherical Heronian frieze from a traversing path (cf.~Corollary~\ref{cor: 3.10 analogue}), and a spherical Cayley-Menger frieze from a thickening of a traversing path (cf.~Theorem~\ref{thm: 5.18 analogue}).  The conditions for doing so are slightly different.  While both propagation algorithms require the non-vanishing of the entries at the interior integer nodes (cf.~conditions~\eqref{eq: non-vanishing condition} and~\eqref{eq: condition interior not zero}), recovering a spherical Cayley-Menger frieze additionally requires the non-vanishing of Heron polynomials (cf.~condition~\eqref{eq: condition triangles not zero}).

\end{remark}

By Corollary~\ref{cor: 3.12 analogue},  any sufficiently generic spherical Heronian frieze arises from a spherical polygon.  The following result shows that the same is true for sufficiently generic coherent spherical Cayley-Menger friezes.

\begin{theorem}\label{thm: 5.20 analogue}
    Let~$\zscm=(z_\alpha)_{\alpha\in I_n^{SCM}}$ be a coherent spherical Cayley-Menger frieze of order~$n$ satisfying the conditions in Theorem~\ref{thm: 5.18 analogue}.  Then there exists an~$n$-gon~$P$ on~$\bfS$ such that~$\zscm=\zscm(P)$.  Consequently~$\zscm$ has glide symmetry.
\end{theorem}

We will prove Theorem~\ref{thm: 5.20 analogue} at the end of Section~\ref{sec: frieze equiv}.

\newpage

\section{Spherical Cayley-Menger friezes vs. spherical Heronian friezes}\label{sec: frieze equiv}

In this section, we show when and how a spherical Heronian frieze gives rise to a spherical Cayley-Menger frieze and vice versa. This generalizes the results in~\cite[Section 6]{fomin_heronian_2021} to the case of arbitrary~$\kap$.  

We begin by establishing genericity conditions under which a spherical Heronian diamond restricts to a spherical Cayley-Menger diamond.

\begin{lemma}  \label{lem: 6.1 analogue}
Let~$\xS=(a,b,c,d,e,f,p,q,r,s)\in\C^{10}$ be a spherical Heronian diamond satisfying either
\begin{align}
    &a=q=r=0\, , \quad \text{or} \label{eq: condition for H to be CM 1}\\
    &c=p=s=0\, , \quad \text{or} \label{eq: condition for H to be CM 2}\\
    &\left(e\left(e-\mytfrac{4}{\kap}\right),f\left(f-\mytfrac{4}{\kap}\right)\right)\ne (0,0) \label{eq: condition for H to be CM 3}\, .
\end{align}
Then
\begin{alignat}{3}
\label{eq: partial-xe,xf}
-2rs &= \partial_\leftarrow \scmfour (a,b,c,d,e,f)\, , \qquad & -2pq &= \partial_\rightarrow \scmfour (a,b,c,d,e,f)\, ,\\
\label{eq: partial-xb,xd}
2qs &= \partial_{\!\smallneline} \scmfour (a,b,c,d,e,f)\, , \qquad &2pr &= \partial_{\>\smallseline} \scmfour (a,b,c,d,e,f)\, ,\\
\label{eq: partial-xa,xc}
2ps &= \partial_\uparrow \scmfour (a,b,c,d,e,f)\, , \qquad &2rq &= \partial_\downarrow \scmfour (a,b,c,d,e,f)\, .
\end{alignat}

\end{lemma}
\begin{proof}
    First assume condition~\eqref{eq: condition for H to be CM 1}. By Lemma~\ref{lem: degenerate 1}, we see that
    \begin{align*}
        d&=e\, ,\\
        f&=b\, ,\\
ps=p^2&=H^\kap(b,c,e)\, .
        \end{align*}
        Substituting these values into~$\partial_\uparrow \scmfour (a,b,c,d,e,f)$, we can now manually verify equation~\eqref{eq: partial-xa,xc}.  The other equations and the case of condition~\eqref{eq: condition for H to be CM 2} can be similarly checked by hand.

    Consider the case where condition~\eqref{eq: condition for H to be CM 3} holds.  Assume without loss of generality that~$e\not\in\left\{0,\frac{4}{\kap}\right\}$ (the other case is similar).  Triangulate a four-cycle~$G$ by connecting the first and third vertices, and assign the elements of~$\xS$ to the edges in~$G$ in accordance with Example~\ref{ex: main tet example} (you will not use~$f,r,s$).  By Proposition~\ref{prop: cor 2.9 analogue} there is a spherical~$4$-gon~$P$ on~$\bfS$ with~$\xS=\xS_G(P)$. By Proposition~\ref{prop: analogue of 2.14} the remaining measurements for~$P$ will be~$f,r,s$.  Then by Proposition~\ref{prop: deriv geom meaning}, we get equations~\eqref{eq: partial-xe,xf}-\eqref{eq: partial-xa,xc} by recalling that~$\skap_{123}=p$,~$\skap_{134}=q$, etc.
\end{proof}

\begin{proposition}\label{prop: 6.2 analogue}
    Let~$(a,b,c,d,e,f,p,q,r,s)\in\C^{10}$ be a spherical Heronian diamond satisfying one of the conditions~\eqref{eq: condition for H to be CM 1}-\eqref{eq: condition for H to be CM 3}.  Then~$(a,b,c,d,e,f)$ is a spherical Cayley-Menger diamond.

\end{proposition}
\begin{proof}
    Remark~\ref{rem: sph heron implies scm is zero} addresses the case of condition~\eqref{eq: condition for H to be CM 3}.  In the case of condition~\eqref{eq: condition for H to be CM 1}, we have~$d=e$ and~$f=b$ (see Lemma~\ref{lem: degenerate 1}) and we can manually check that~$\scmfour(a,b,c,d,e,f)=0$.  The case of condition~\eqref{eq: condition for H to be CM 2} is similar.
\end{proof}

When we restrict a spherical Heronian frieze to form a spherical Cayley-Menger frieze, it automatically satisfies the coherence condition:

\begin{proposition} \label{prop: 6.3 analogue}
Let~$\xx_1,\xx_2,\xx_3,\xx_4$ be four interlocking spherical Heronian diamonds arranged as shown in Figure~\ref{fig: four Heronian diamonds}. Suppose that each~$\xx_i$ satisfies one of the  conditions~\eqref{eq: condition for H to be CM 1}-\eqref{eq: condition for H to be CM 3}. 
    Then the corresponding interlocking spherical Cayley-Menger diamonds (cf.~Proposition~\ref{prop: 6.2 analogue}) satisfy the coherence condition~\eqref{eq: coherence condition}.
\end{proposition}
    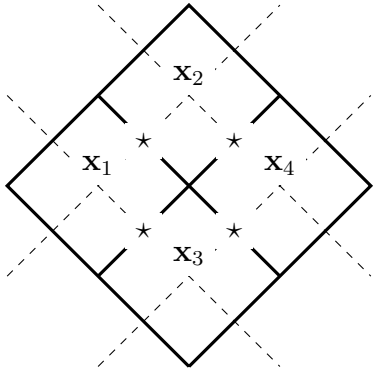
\begin{figure}[htbp!]
    \centering
    \begin{tikzpicture}[scale=0.6,every node/.style={rectangle,fill=white,inner sep=2pt}]

        \begin{scope}[on background layer]
\draw[dashed] (-2,8) --(4,2);
\draw[dashed] (-4,6) --(2,0);
\draw[dashed] (-4,2) --(2,8);
\draw[dashed] (-2,0) --(4,6);
\draw[very thick] (-2,6) -- node {$\star$} (0,4) -- node {$\star$} (2,2);
    \end{scope}
    \draw[very thick] (-2,2) -- node {$\star$} (0,4) 
        -- node {$\star$} (2,6);
        \draw[very thick]
        (0,0) 
        -- (-2,2) 
        -- (-4,4) 
        -- (-2,6) 
        -- (0,8) 
        -- (2,6) 
        -- (4,4) 
        -- (2,2) 
        -- (0,0);
        
        \draw (-2,4) node[above] {$\xx_1$};
        \draw (0,6) node[above] {$\xx_2$};
        \draw (2,4) node[above] {$\xx_4$};
        \draw (0,2) node[above] {$\xx_3$};

    \end{tikzpicture}
    \caption{Four interlocking spherical Heronian diamonds as in Proposition~\ref{prop: 6.3 analogue}. The starred entries appear in the coherence condition on the corresponding spherical Cayley-Menger diamonds.}
    \label{fig: four Heronian diamonds}
\end{figure}
\begin{proof}
    Applying Lemma~\ref{lem: 6.1 analogue} to~$\xx_1,\xx_2,\xx_3,\xx_4$, we see that each side of the coherence condition is the product of the starred entries in Figure~\ref{fig: four Heronian diamonds}.
\end{proof}

We next discuss the other direction: constructing a spherical Heronian frieze from a spherical Cayley-Menger frieze. 

\begin{proposition}\label{prop: 6.4 analogue}
Let~$\xx=(a,b,c,d,e,f)$ be a spherical Cayley-Menger diamond such that~$\left(e\left(e-\frac{4}{\kap}\right),f\left(f-\frac{4}{\kap}\right)\right)\ne (0,0)$ and 
\begin{equation}\label{eq: condition H_kap not zero}
    \hkap(b,c,e)\hkap(a,d,e)\hkap(a,b,f)\hkap(c,d,f)\ne 0\, .
\end{equation}
    Then there exist exactly two spherical Heronian diamonds restricting to~$\xx$. These two diamonds differ from each other by the simultaneous sign change of~$p,q,r,s$.
\end{proposition}
\begin{proof}
    Since~$\xx$ is a spherical Cayley-Menger diamond, by Lemma~\ref{lem: partial H_kap} and condition~\eqref{eq: condition H_kap not zero}, the partial derivatives in Lemma~\ref{lem: 6.1 analogue} do not vanish.  This implies that for the tuple~$(a,b,c,d,e,f,p,q,r,s)$ to be a spherical Heronian diamond, we must have~${pqrs\ne 0}$.  The pairwise products of~$p,q,r,s$ must be given by~\eqref{eq: partial-xe,xf}-\eqref{eq: partial-xa,xc}.  This means that these numbers are uniquely determined by~$\xx$ up to a simultaneous change of sign.

    Consider the following candidate values for~$p,q,r,s$. 
 Set~$p^2=H(b,c,e)$ and select a sign for~$p$ (the other choice will cause all signs to change simultaneously, giving the other possible solution).  Then, inspired by equations~\eqref{eq: partial-xe,xf}-\eqref{eq: partial-xa,xc}, let
 \begin{equation*} \label{eq: values for q r s}
     q=-\mytfrac{\partial_\rightarrow \scmfour (x)}{2p}\, , \quad r=\mytfrac{\partial_{\>\smallseline} \scmfour (x)}{2p}\, , \quad s=\mytfrac{\partial_\uparrow \scmfour (x)}{2p}\, .
 \end{equation*}
  We need to check that the collection~$(a,b,c,d,e,f,p,q,r,s)$ satisfies equations~\eqref{eq:def1sphp}-\eqref{eq: def1sph erminuss} and therefore is a spherical Heronian diamond.

  The Heron relations~\eqref{eq:def1sphp}-\eqref{eq:def1sphs} are a consequence of our choice of~$p$ together with Lemma~\ref{lem: partial H_kap}.  Since we know that~${p\ne0}$ (as we've assumed~$H(b,c,e)\ne 0$), equation~\eqref{eq:def1sph p+qanalogue} is equivalent to
  \begin{multline*}
      2\hkap(b,c,e)-\partial_\rightarrow\scmfour (\xx)=\\\partial_{\>\smallseline} \scmfour (\xx)+\partial_\uparrow \scmfour (\xx)+\kap \left(a\hkap(b,c,e)-b\partial_\rightarrow \scmfour (\xx)-e\partial_{\>\smallseline} \scmfour (\xx)\right)
  \end{multline*}
which can be checked manually.  Equation~\eqref{eq: def1sph bretanalogue} is equivalent to
\begin{multline*}
    4ef\left(1-\mytfrac{\kap e}{4}\right)=\\ \hkap(b,c,e)-\partial_\rightarrow \scmfour (\xx)+\hkap(a,d,e)+(a-b+c-d)-\kap e(a-b)(c-d)\, ,
\end{multline*}
which can be verified by direct computation.
Similarly, equation~\eqref{eq: def1sph erminuss} is equivalent to
\begin{equation*}
    e\left(\partial_{\>\smallseline} \scmfour (\xx) -\partial_\uparrow \scmfour (\xx)\right)=2\hkap(b,c,e)(a-d)-\partial_\rightarrow \scmfour (\xx)(b-c)\, ,
\end{equation*}
which can also be checked manually. (Contrary to the remark in the proof of \cite[Proposition 6.4]{fomin_heronian_2021}, this verification does not rely on the condition~$M(\textbf{x})=0$.)\qedhere\end{proof}

The following result fully clarifies the relationship between spherical Heronian friezes and spherical Cayley-Menger friezes 

\begin{theorem} \label{thm: big final thm}
Let~$\zfrieze=(z_\alpha)_{\alpha\in I_n}$ be a spherical Heronian frieze such that 
\begin{equation}\label{eq: interior not zero}
    z_{(i,j)}\not\in \left\{0,\mytfrac{4}{\kap}\right\} \text{ for any } (i,j)\in \Z^2 \text{ such that } 2\le j-i\le n-2
\end{equation}
(same as~\eqref{eq: non-vanishing condition} or~\eqref{eq: condition interior not zero}). Then the restriction of~$\zfrieze$ to~$I_n^{SCM}$ is a coherent spherical Cayley-Menger frieze.

Conversely, let~$\zscm=(z_\alpha)_{\alpha\in I_n^{SCM}}$ be a coherent spherical Cayley-Menger frieze satisfying~\eqref{eq: interior not zero}.  Suppose that for every~$(i,j)\in\Z^2$ with~$2\le j-i\le n-1$, we have 
\begin{equation}\label{eq: areas not zero either}
    \hkap\left(z_{(i,j)},z_{(i+1,j)},z_{(i+\frac{1}{2},\neline)}\right)\ne 0, \quad \hkap\left(z_{(i,j-1)},z_{(i,j)},z_{(\seline,j-\frac{1}{2})}\right)\ne 0
\end{equation}
(same as~\eqref{eq: condition triangles not zero}).
Then there exists a spherical Heronian frieze~$\zfrieze=(z_\alpha)_{\alpha\in I_n}$ that extends~$\zscm$.  This extension is unique up to a global change of sign of the entries indexed by~$I_n\setminus I_n^{SCM}$.
    
\end{theorem}

\begin{proof}
    First assume that~$\zfrieze=(z_\alpha)_{\alpha\in I_n}$ is a spherical Heronian frieze satisfying~\eqref{eq: interior not zero}.  By Proposition~\ref{prop: 6.2 analogue}, each spherical Heronian diamond in the frieze can be restricted to a spherical Cayley-Menger diamond (disregard the values indexed by~$\left(i,j+\frac{1}{2})\right)$ and~$\left(i+\frac{1}{2},j\right)$) to get a collection~$\zscm$.  By definition~$\zscm$ is a spherical Cayley-Menger frieze.  By Corollary~\ref{cor: 3.12 analogue}, the spherical Heronian frieze~$\zfrieze$, hence the spherical Cayley-Menger frieze~$\zscm$, corresponds to a polygon~$P$.  It follows by Theorem~\ref{thm: 5.12 analogue} that the spherical Cayley-Menger frieze~$\zscm$ is coherent. 

    Now let~$\zscm=(z_\alpha)_{\alpha\in I_n^{SCM}}$ be a coherent spherical Cayley-Menger frieze satisfying~\eqref{eq: interior not zero}-\eqref{eq: areas not zero either}.  Let~$\pi$ be a traversing path in~$I_n^{SCM}$,  with~$\overline{\pi}$  its thickening, see Definition~\ref{def: thickening}.  Examining~$I_n$, denote by~$\widetilde{\pi}$ the two adjacent traversing paths that together restrict to~$\overline{\pi}$ in~$I_n^{SCM}$. Note that~$\widetilde{\pi}$ as a set is strictly larger than~$\overline{\pi}$, as it contains values corresponding to midpoint indices~$(i,j+\frac{1}{2}),(i+\frac{1}{2},j)$ that do not appear in~$I_n^{SCM}$.

    Let~$\xdiamond_{i,j}(\zscm)$ be a spherical Cayley-Menger diamond in~$\overline{\pi}$.  Proposition~\ref{prop: 6.4 analogue} gives two spherical Heronian diamonds that restrict to~$\xdiamond(\zscm)$ and differ by a simultaneous change of sign of the midpoint indices that appear in~$\widetilde{\pi}\setminus\overline{\pi}$.  Select one to create a first spherical Heronian diamond~$\xx$. As we travel along~$\overline{\pi}$ the diamonds overlap, so our choice of signs 
 for the midpoint values in~$\xx$ determines all subsequent sign choices.  We get exactly two sets of interlocking spherical Heronian diamonds~$\widetilde{\zfrieze_1}$ and~$\widetilde{\zfrieze_2}$ that agree with~$\zscm$ on~$\overline{\pi}$, and differ by a simultaneous sign change of the midpoint entries in~$\widetilde{\pi}\setminus \overline{\pi}$.

We have two sets of entries corresponding to~$\widetilde{\pi}$. We want to extend these arrays to spherical Heronian friezes that restrict to~$\zscm$.  Let~$\xx_1,\xx_2,\xx_3$ be three spherical Heronian diamonds indexed by~$\widetilde{\pi}$ that share a single node~$z_{(i,j)}$.  By Corollary~\ref{cor: 2.16 analogue}, we construct~$\xx_4$, a fourth spherical Heronian diamond containing~$z_{(i,j)}$. The conditions~\eqref{eq: interior not zero}-\eqref{eq: areas not zero either}, together with the boundary conditions on spherical Cayley-Menger friezes~\eqref{eq: boundary condition top and bottom row CM}-\eqref{eq: boundary condition lines 2 CM} mean that all four spherical Heronian diamonds meet conditions~\eqref{eq: condition for H to be CM 1}-\eqref{eq: condition for H to be CM 3}, so by Proposition~\ref{prop: 6.3 analogue} the four corresponding spherical Cayley-Menger diamonds satisfy the coherence condition~\eqref{eq: coherence condition}.  We also satisfy the conditions of Proposition~\ref{prop: 5.15 analogue}; the fact that the new entry of~$\xx_4$ is unique means that it agrees with the corresponding entry in~$\zscm$.

Propagate away from~$\widetilde{\pi}$ in this way to construct unique spherical Heronian friezes~$\zfrieze_1$ and~$\zfrieze_2$ that extend~$\widetilde{\zfrieze_1}$ and~$\widetilde{\zfrieze_2}$ respectively and restrict to~$\zscm$.  To~find values near the boundary, we propagate using rules~\eqref{eq: bdry} and~\eqref{eq: bdry 4}-\eqref{eq: bdry 6} which agree with the corresponding rules for spherical Cayley-Menger friezes~\eqref{eq: boundary condition top and bottom row CM}-\eqref{eq: boundary condition lines 2 CM}. 

The final claim is that the two spherical Heronian friezes differ by a simultaneous change of sign of the entries indexed by~$I_n\setminus I_n^{SCM}$. Above, we selected an initial spherical Cayley-Menger diamond~$\xdiamond_{i,j}(\zscm)$ which generated spherical Heronian friezes~$\zfrieze_1$ and~$\zfrieze_2$. Suppose we start with a different spherical Cayley-Menger diamond and generate two spherical Heronian friezes~$\zfrieze_3$ and~$\zfrieze_4$.  By Proposition~\ref{prop: 6.4 analogue} again, the first spherical Cayley-Menger diamond~$\xdiamond_{i,j}(\zscm)$ only extends to two spherical Heronian diamonds, the frieze~$\zfrieze_3$ must have one and~$\zfrieze_4$ the other.  Therefore~$\zfrieze_3$ and~$\zfrieze_4$ must be the same as~$\zfrieze_1$ and~$\zfrieze_2$. \qedhere

    \end{proof}

\begin{proof}[Proof of Theorem~\ref{thm: 5.20 analogue}]
By Theorem~\ref{thm: big final thm}, the spherical Cayley-Menger frieze~$\zscm$ can be extended to a spherical Heronian frieze, which by Corollary~\ref{cor: 3.12 analogue} arises from a polygon~$P$ on~$\bfS$.  Thus~$\zscm$ must also arise from~$P$.
\end{proof}

 \newpage

\begin{flushleft}
\bibliographystyle{myamsalpha}
\bibliography{references}
\RaggedRight
\end{flushleft}

\end{document}